\documentclass[11pt,reqno]{amsart}
\usepackage{amsmath,amssymb,amsthm,latexsym,cite,cancel}
\usepackage[small]{caption}
\usepackage{graphicx,wasysym,overpic,tikz,color}
\usepackage{subfigure,color}
\usepackage{cite}
\usepackage[colorlinks=true,urlcolor=blue,
citecolor=red,linkcolor=blue,linktocpage,pdfpagelabels,
bookmarksnumbered,bookmarksopen]{hyperref}
\usepackage[italian,english]{babel}
\usepackage{units}
\usepackage{enumitem}
\usepackage[left=2.1cm,right=2.1cm,top=2.71cm,bottom=2.71cm]{geometry}
\usepackage[hyperpageref]{backref}

\usepackage[colorinlistoftodos]{todonotes}

\allowdisplaybreaks

\newtheorem{theorem}{Theorem}[section]

\newtheorem{proposition}[theorem]{Proposition}
\newtheorem{lemma}[theorem]{Lemma}
\newtheorem{remark}[theorem]{Remark}

\newtheorem{corollary}[theorem]{Corollary}

\numberwithin{equation}{section}

\title[Magnetic nonlinear Schr\"odinger equation in $\mathbb{R}^2$]{Multi-peak solutions for a magnetic nonlinear Schr\"{o}dinger equation with exponential critical growth in $\mathbb{R}^{2}$}

\author[P. d'Avenia]{Pietro d'Avenia}

\address[P. d'Avenia]{\newline\indent
	Dipartimento di Meccanica, Matematica e Management
	\newline\indent
	Politecnico di Bari
	\newline\indent
	Via Orabona 4,  70125  Bari, Italy}
\email{\href{mailto:pietro.davenia@poliba.it}{pietro.davenia@poliba.it}}

\author[C. Ji]{Chao Ji}

\address[C. Ji]{\newline\indent
	School of Mathematics
	\newline\indent
	East China University of Science and Technology
	\newline\indent
	Shanghai 200237, PR China }
\email{\href{mailto:jichao@ecust.edu.cn}{jichao@ecust.edu.cn}}

\author[Q. Li]{Qingyue Li}
\address[Q. Li]{\newline\indent
	School of Mathematics
	\newline\indent
	East China University of Science and Technology
	\newline\indent
	Shanghai 200237, PR China }
\email{\href{mailto:lzulqy2021@163.com}{lzulqy2021@163.com}}

\subjclass[2020]{35J20,
	35Q55,
	35B33,
	35B25.
}
\keywords{Nonlinear Schr\"{o}dinger equation, Magnetic field, Multi-peak solutions, Exponential critical growth.}

\begin{document}
	
	\begin{abstract}
		In this paper, we study the following magnetic nonlinear  Schr\"odinger equation
		\[
		\left(\frac{\varepsilon}{i}\nabla-A(x)\right)^{2}u+V(x)u=f(|u|^{2})u\,\ \text{in}\,\ \mathbb{R}^{2},
		\]
		where $\varepsilon>0$ is a parameter, $A\in C(\mathbb{R}^{2}, \mathbb{R}^{2})$ is a  magnetic potential, $V\in C(\mathbb{R}^{2}, \mathbb{R})$ is a multi-well potential, and $f\in C^1(\mathbb{R}, \mathbb{R})$ exhibits exponential critical growth. By variational methods and a penalization technique, 
		we prove  the existence and multiplicity of nontrivial solutions for sufficiently small $\varepsilon>0$. If, in addition, $A\in C^1(\mathbb{R}^2,\mathbb{R}^2)$, we show that these solutions have multiple peaks and we establish their concentration behavior
		and exponential decay.
	\end{abstract}
	
	\maketitle
	
	\begin{center}
		\begin{minipage}{11cm}
			\tableofcontents
		\end{minipage}
	\end{center}

	\section{Introduction and main results}

	The nonlinear Schr\"odinger equation (NLS) is a fundamental model in mathematical physics, arising in various areas such as 
	quantum systems subject to external potentials, including complex and magnetic potentials.

	In this paper, we consider the following magnetic NLS equation 
	\begin{equation}\label{p1.1}
		\left(\frac{\varepsilon}{i}\nabla-A(x)\right)^{2}u+V(x)u=f(|u|^{2})u\quad \text{in } \mathbb{R}^{2},
	\end{equation}
	where $V\in C(\mathbb{R}^{2},\mathbb{R})$ is a multi-well potential, $A\in C(\mathbb{R}^{2},\mathbb{R}^2)$ is a magnetic potential, and $f\in C^1(\mathbb{R},\mathbb{R})$ has exponential critical growth at infinity.
	For small $\varepsilon>0$, we study the existence, multiplicity, concentration, and decay of multi-peak solutions.
	
	Equation \eqref{p1.1} arises when looking for standing wave solutions $\psi(x,t):=e^{-i\omega t/\hbar}u(x)$, with $\omega\in \mathbb{R}$, to the nonlinear Schr\"odinger equation
	\[i\hbar\frac{\partial \psi}{\partial t}=\left(\frac{\hbar}{i}\nabla-A(x)\right)^{2}\psi+U(x)\psi-f(|\psi|^{2})\psi, \quad (x,t)\in\mathbb{R}^{N}\times\mathbb{R},\]
	taking $N=2$. In particular, the operator
	\[\psi\mapsto \left(i\hbar\frac{\partial}{\partial t}-U(x)\right)\psi-\left(\frac{\hbar}{i}\nabla-A(x)\right)^{2}\psi\]
	appears through the minimal coupling rule 
	(see e.g. \cite[Section 4.1]{Peskin})
	when the interaction between the field $\psi$ and the external electromagnetic field (with electrostatic potential $U$ and magnetostatic potential $A$) is considered. Then, taking $\varepsilon = \hbar$ and defining $V = U - \omega$, we recover \eqref{p1.1}.
	
	In the semiclassical limit $\varepsilon\to 0^+$, the non-magnetic Schr\"odinger equation
	\[
	-\varepsilon^2 \Delta u+V(x)u=f(|u|^{2})u\quad \text{in } \mathbb{R}^{N},
	\]
	has been extensively studied. Existence and multiplicity results for bound state solutions go back to the pioneering work of Floer and Weinstein \cite{Floer}, who considered the case $N=1$ and $f=\gamma i_{\mathbb{R}}$ with $\gamma>0$.
	Subsequently, many authors used various methods to study the high-dimensional cases. In \cite{Pino}, del Pino and Felmer studied the existence and concentration behavior of the solutions for the following semilinear elliptic problem
	\[
	\begin{cases}
		-\varepsilon^2\Delta u+V(x)u=f(u) \text{ in } \Omega,\\
		u=0 \text{ on } \partial\Omega,  u>0 \text{ in } \Omega,
	\end{cases}
	\]
	on a (possibly unbounded) domain  $\Omega\subset\mathbb{R}^N$ with $N\geq 3$. The potential function $V$ is assumed to be locally H\"older continuous bounded below away from zero and satisfying the condition
	\begin{equation}\label{VassdPF}
		\mathop{\inf}\limits_{x\in \Lambda} V(x)<\mathop{\min}\limits_{x\in \partial\Lambda} V(x),
	\end{equation}
	where $\Lambda$ is  a bounded open set compactly contained in $\Omega$ and the nonlinearity $f$ satisfies certain subcritical growth conditions.
	In \cite{Gui}, Gui investigated the existence and concentration of multi-bump solutions for  $\Omega = \mathbb{R}^N$, replacing \eqref{VassdPF} with
	\[\inf_{x\in \Omega_{i}} V(x)<\inf_{x\in \partial\Omega_{i}} V(x),\quad i=1,2,...,k,\]
	where $\Omega_1,\ldots, \Omega_k$ are $k$ disjoint bounded regions. The
	nonlinearity $f$ satisfies  $0\leq f^{\prime}(u)\leq a_1+a_2u^{p-1}$ for constants $a_1,a_2>0$ and $ 1 < p < \frac{N+2}{N-2}$ (for $N > 2$) or $p>1$ (for $N = 1, 2$). The proofs of both these results are variational and use {\em penalization} techniques. For further results on the existence, multiplicity and concentration of solutions under  various assumptions on the potential  $V$ and the nonlinearity $f$, see
	\cite{Byeon2007, Arma, Pino1998, Byeon2013, Byeon2013b, Byeon2003, d'Avenia2012} and the references therein.

	The first result on the magnetic NLS equation was obtained by Esteban and Lions \cite{-Esteban1989}, who used the concentration-compactness principle and minimization approach to find solutions for fixed $\varepsilon>0$  and $N=2,3$. For $N\geq 3$, in \cite{jeanjean}, Cingolani, Jeanjean, and Secchi  proved the existence of a semiclassical multi-peak solution $u:\mathbb{R}^N\rightarrow \mathbb{C}$ under nearly optimal conditions on the nonlinearity. For a more complete overview when $N\geq 2$, we refer to \cite{-Alves2011,Bartsch,-Cingolani2003} and their references.
	
	For $N=2$ (namely for \eqref{p1.1}) and for the exponential critical  growth case, in \cite{-Barile}, the existence has been proved for $\varepsilon>0$ {\em fixed}. In \cite{d'Avenia2022}, when $A\in C^{\alpha}(\mathbb{R}^2,\mathbb{R}^2)$, with $\alpha\in (0,1]$, and
	\[
	0<V_{0}=\mathop{\min}\limits_{x\in \Lambda}V(x)<\mathop{\min}\limits_{x\in \partial\Lambda} V(x),
	\]
	where $\Lambda\subset \mathbb{R}^2$ is a bounded open set, using the penalization technique as in \cite{Pino} and the Ljusternik-Schnirelmann theory, multiplicity and concentration of single-peak solutions for problem \eqref{p1.1} have been obtained when $\varepsilon>0$ is small.
	Later,  in \cite{d'Avenia2023}, under the  condition $A\in C(\mathbb{R}^2,\mathbb{R}^2)$ and a global assumption on the potential $V$, namely,
	\begin{equation*}
		V_{\infty}=\mathop{\liminf}\limits_{|x|\rightarrow +\infty} V(x)>V_0=\mathop{\inf}\limits_{x\in \mathbb{R}^2} V(x)>0,
	\end{equation*}
	the authors investigated the existence, multiplicity, and, with the additional assumption $A\in C^1(\mathbb{R}^2,\mathbb{R}^2)$, concentration, and decay of nontrivial solutions for small $\varepsilon>0$.

	Motivated by these results, 
	as a continuation of our joint research project,
	a natural problem is to study multi-peak solutions of the magnetic nonlinear  Schr\"odinger equation \eqref{p1.1} with exponential critical growth.
	In this paper, the term $k$-peak refers to the presence of $k$ distinguished local maxima of the modulus, located in $k$ different potential wells; these are not asserted to be the only local maxima.
	On the potential $V\in C(\mathbb{R}^2,\mathbb{R})$ we assume:
	\begin{enumerate}[label=($V_\arabic*$),ref=$V_\arabic*$]
		
		\item \label{V1} there exists \( V_0 > 0 \) such that \( V(x) \geq V_0 \) for all \( x \in \mathbb{R}^2 \);
		
		\item \label{V2} there exist \(m\) bounded domains\footnote{Here, by {\em domain} we mean an open and connected set.} \(\Lambda_j\subset\mathbb{R}^2\) (\(j=1,2,\ldots,m\)) whose closures are mutually disjoint and such that
		\[
		V_j:=\min_{x\in\overline{\Lambda}_j}V(x)
		<a_j:=\min_{x\in\partial\Lambda_j}V(x),
		\quad \forall j\in\{1,2,\ldots,m\}.
		\]
	\end{enumerate}
	
	Observe that
	\begin{equation}\label{Kj}
		K_j:=\{x\in \Lambda_{j}: V(x)=V_{j}\}\neq \emptyset
		\quad \text{and}\quad K_j\Subset\Lambda_j,\quad \forall j\in \{1,2,...,m\}.
	\end{equation}
	
	In addition, we denote $V^{*}:=\max\{ V_{j}:j=1,2,...,m\}$.

	Moreover, our nonlinearity $f$ satisfies:
	\begin{enumerate}[label=($f_\arabic*$),ref=$f_\arabic*$]
		
		\item \label{f1} $f\in C^1(\mathbb{R},\mathbb{R})$ and \( f(t) = 0 \) if \( t \leq 0 \);
		
		\item \label{f2} there holds
		\[
		\lim_{t \to +\infty} \frac{t f(t^2)}{e^{\alpha t^2}} =
		\begin{cases}
			0, & \text{if } \alpha > 4\pi, \\
			+\infty, & \text{if } 0 < \alpha < 4\pi;
		\end{cases}
		\]
		
		\item \label{f3} there exists \( \Theta > 2 \) such that, for every \( t > 0 \),
		\[
		0 < \Theta F(t) \leq 2 t f(t), \text{ where } F(t) = \int_0^t f(s) \, ds;
		\]
		
		\item \label{f4} there exist \( p > 2 \) and
		\[
		C_p >  \left( \frac{8m\Theta}{\Theta-2} \right)^{\frac{p-2}{2}} \left( \frac{p - 2}{p} \right)^{\frac{p-2}{2}} V^*S_p^{\frac{p}{2}}  > 0
		\]
		such that
		\[
		f'(t) \geq \frac{p - 2}{2} C_p t^{\frac{p-4}{2}}, \quad \forall t > 0,
		\]
		where 
		\[
		S_p := \inf_{u \in H^1(\mathbb{R}^2, \mathbb{R}) \setminus \{0\}} 
		\frac{1}{|u|_p^2}\int_{\mathbb{R}^2} (|\nabla u|^2 + |u|^2) \, dx;
		\]
		
		\item \label{f5} \( f'(t) \leq (e^{4\pi t} - 1) \) for any \( t \geq 0 \).
		
	\end{enumerate}

	By conditions \eqref{f4} and \eqref{f5},  there exist two positive constants $C_{1}$ and $C_{2}$ such that
	\[C_{1}t^{\frac{p-4}{2}}\leq f^{\prime}(t)< C_{2}t\,\ \text{as}\,\ t\rightarrow 0^{+},\] which implies that $p\geq 6$.

	Our main result is the following.
	\begin{theorem}\label{1.1}
		For each nonempty $\Gamma\subset \{1,2,...,m\}$, there exist $\varepsilon_{0},\gamma>0$ such that, for every $\varepsilon\in(0,\varepsilon_{0})$, problem \eqref{p1.1} has a nontrivial solution $u_{\varepsilon}$ satisfying $\mathop{\sup}\limits_{x\in \mathbb{R}^{2}}|u_{\varepsilon}(x)|^2\geq \gamma$.
		Moreover, if $A\in C^1(\mathbb{R}^2,\mathbb{R}^2)$, then 
		there exists $R>0$ such that, for every $\varepsilon\in(0,\varepsilon_0)$, there exist mutually distinct local maximum points 	$y_{\varepsilon,j}\in\Lambda_j$, $j\in\Gamma$, of $|u_\varepsilon|$ such
		that
		\begin{equation}\label{zhendang}
			|u_{\varepsilon}|^2_{L^{\infty}(\mathbb{R}^{2}\setminus\cup_{j\in \Gamma}B_{\varepsilon R}(y_{\varepsilon,j}))}<\gamma,\quad |u_{\varepsilon}|^2_{L^{\infty}(B_{\varepsilon R}(y_{\varepsilon,j}))}\geq \gamma,
		\end{equation}	
		$\operatorname{dist}(y_{\varepsilon,j},K_j)\to0$ as
		$\varepsilon\to0^+$ for every $j\in\Gamma$, and, in particular,
		$V(y_{\varepsilon,j})\to V_j$. Moreover there exist constants $C,\sigma>0$
		(independent of $x$ and $\varepsilon$) such that
		\begin{equation}\label{decay}
			|u_\varepsilon(x)|\leq C e^{-\frac{\sigma}{\varepsilon}d_\varepsilon(x)},\quad x\in\mathbb{R}^2,\qquad
			d_\varepsilon(x):=\min_{j\in\Gamma}|x-y_{\varepsilon,j}|.
		\end{equation}
	\end{theorem}
	As an immediate consequence, we have
	\begin{corollary}
		If the assumptions of Theorem \ref{1.1} hold and
		\(A\in C^1(\mathbb R^2,\mathbb R^2)\), then
		there exists $\varepsilon^{*}>0$ such that, for any $\varepsilon\in (0,\varepsilon^{*})$, problem \eqref{p1.1} has at least $2^{m}-1$ multi-peak solutions, among which $m$ concentrate in a single potential well.
	\end{corollary}

	Our approach relies on variational methods combined with a penalization technique. The proof of Theorem \ref{1.1} is inspired by the approach in \cite{Gui}, but it requires a more delicate analysis due to the presence of the magnetic potential $A$. As $\varepsilon\rightarrow 0^+$, our solutions become highly concentrated near the local minimum points of the potential $V$, forming sharp peaks and decaying rapidly to zero elsewhere. We construct solutions with $s=|\Gamma|$ spatially separated peaks, one near each $K_j$ for $j\in\Gamma$, whose corresponding local profiles approach ground state solutions of \eqref{fude} with $\mu=V_j$.
	Moreover, the presence of the multi-well potential requires  delicate estimates and a careful and deep analysis of some technical aspects,  especially as shown in Proposition \ref{fengjie}.

	This paper is organized as follows.
	In Section \ref{Sec2}, we introduce some preliminary results and study the limit problem.
	In Section \ref{Secthree}, we introduce and investigate a modified problem by verifying the Palais-Smale condition for the corresponding functional and proving a {\em Splitting Theorem}, which plays a crucial role in obtaining our result. 
	Section \ref{Sec3} is devoted to proving the existence result for the modified problem. 
	Finally, we conclude the paper in Section \ref{Sec4} with the proof of Theorem \ref{1.1}. Appendix \ref{proofexistenceC} contains the proof of Lemma \ref{existenceC}.

	\subsection*{Notation}
	\begin{itemize}
		\item $C, C_1, C_2, \ldots$ denote positive constants whose exact values can change from line to line;
		\item $B_{R}(y)$ denotes the open disk centered at $y\in\mathbb{R}^2$ with radius $R>0$ and $B^{c}_{R}(y)$ denotes its complement in $\mathbb{R}^{2}$;
		\item $\Vert \cdot \Vert$, $| \cdot |_{q}$, and $|\cdot |_{\infty}$ denote the usual norms of the spaces $H^{1}(\mathbb{R}^{2}, \mathbb{K})$, $L^{q}(\mathbb{R}^{2}, \mathbb{K})$, and $L^{\infty}(\mathbb{R}^{2}, \mathbb{K})$, respectively, where $\mathbb{K}=\mathbb{R}$ or $\mathbb{K}=\mathbb{C}$.
	\end{itemize}

	\section{Preliminaries and the limit problem}\label{Sec2}

	Let us start this section by introducing our functional setting. 
	Following \cite{Lieb}, if $A\in L_{\rm loc}^2(\mathbb{R}^2,\mathbb{R}^2)$, let
	\[H^{1}_{A}(\mathbb{R}^{2},\mathbb{C}):=\left\{u\in L^{2}(\mathbb{R}^{2},\mathbb{C}): |\nabla_{A}u|\in L^{2}(\mathbb{R}^{2},\mathbb{R})\right\}\]
	where
	\[\nabla_{A}u:=\left(\frac{\nabla}{i}-A\right)u,\]
	which is a Hilbert space equipped with the inner product
	\[
	\langle u, v \rangle_{A} := \text{Re}\int_{\mathbb{R}^2} (\nabla_{A} u \overline{\nabla_{A} v} +  u\overline{v})dx.
	\]
	Here, $\operatorname{Re}$ denotes the real part of a complex number, and the bar represents complex conjugation, respectively. Then the norm is given by $\|u\|_{A} := \sqrt{\langle u, u \rangle_{A}}$.
	
	It is well known that the presence of the magnetic potential implies some technical difficulties with respect to the case with $A=0$.
	Even if, by
	\[
	\frac{1}{i}\nabla u=\nabla_{A}u+A(x)u
	\quad
	\text{a.e. in }\mathbb{R}^2,
	\]
	we have that
	\begin{equation}\label{eaineq}
		|\nabla u|\le |\nabla_{A}u|+|A(x)u|
		\quad
		\text{a.e. in }\mathbb{R}^2,
	\end{equation}
	as observed in \cite[Section 7.20]{Lieb}, if $u\in H_A^{1}(\mathbb{R}^{2},\mathbb{C})$, then $u$ does not necessarily belong to $H^{1}(\mathbb{R}^{2},\mathbb{C})$. Moreover, if we take $A(x)=x$  and $u(x)=(1+|x|^2)^{-1}$ for $x\in\mathbb{R}^2$,
	then, $u\in H^{1}(\mathbb{R}^{2},\mathbb{C})$ but not in $H_A^{1}(\mathbb{R}^{2},\mathbb{C})$. Thus, to handle the usual Sobolev spaces, it is classical to use the {\em diamagnetic inequality}
	\begin{equation}\label{p2.1}
		|\nabla |u|| \leq |\nabla_{A} u|
	\end{equation}
	that holds a.e. in $\mathbb{R}^2$ and for every $\ u\in H^{1}_{A}(\mathbb{R}^{2},\mathbb{C})$ (see \cite[Theorem 7.21]{Lieb}).
	This inequality is a crucial tool in our subsequent analysis.
	
	Let, now, \[H^1_{A,V}(\mathbb{R}^{2},\mathbb{C}):=\left\{u\in H^{1}_{A}(\mathbb{R}^{2},\mathbb{C}): \int_{\mathbb{R}^2}   V(x)|u|^2dx<+\infty\right\}\]
	endowed with the inner product
	\[
	\langle u, v \rangle_{A,V} := \text{Re}\int_{\mathbb{R}^2} (\nabla_Au \overline{\nabla_A v} +  V(x)u\overline{v})dx.
	\]
	$H^1_{A,V}(\mathbb{R}^{2},\mathbb{C})$ is a Hilbert space and on it we will consider the norm $\|u\|_{A,V} := \sqrt{\langle u, u \rangle_{A,V}}$. 
	
	Following classical arguments (see e.g. \cite{Lieb}), we have the following result.
	
	\begin{lemma}\label{dense}
		If	$A \in L^2_{\rm loc}(\mathbb{R}^2,\mathbb{R}^2)$ and $V\in L^1_{\rm loc}(\mathbb{R}^2,\mathbb{R})$ with \( V \geq V_0>0 \) a.e. in $\mathbb{R}^2$, the space $C_c^\infty(\mathbb{R}^2,\mathbb{C})$ is dense in $H^1_{A,V}(\mathbb{R}^{2},\mathbb{C})$.
	\end{lemma}
	\begin{proof}
		Let us divide the proof into three steps.
		
		Step 1. Approximation of bounded, compactly supported functions.
		
		Assume first that $u\in H^1_{A,V}(\mathbb{R}^{2},\mathbb{C})$ is bounded and has compact support. Then $u\in H^1(\mathbb{R}^2,\mathbb{C})$,  because, using \eqref{eaineq},
		\begin{align*}
			\int_{\mathbb{R}^2}(|\nabla u|^2+|u|^2) dx
			&\le
			\int_{\operatorname{supp} u}
			(2|\nabla_{A}u|^2 + 2|A(x)u|^2+|u|^2) dx\\
			&\leq
			2\int_{\operatorname{supp} u}
			|\nabla_{A}u|^2 dx
			+C+\frac{1}{V_0}\int_{\operatorname{supp} u}V_0|u|^2dx
			\\&\le
			C\left(1+\int_{\mathbb{R}^2}
			(|\nabla_{A}u|^2+V(x)|u|^2)
			dx\right).
		\end{align*}   
		Let $j$ be a standard mollifier, i.e.
		\[j \in C_c^\infty(\mathbb{R}^2,\mathbb{R}),
		\quad j \ge 0,
		\quad \int_{\mathbb{R}^2} j dx = 1,
		\quad \operatorname{supp}  j \subset B_1(0),\]
		and, if $R>0$, let us set $j_R(x) := R^2 j(Rx)$ and $u_R := j_R * u$.
		By \cite[Theorem 2.16]{Lieb} we have $u_R\in C^\infty(\mathbb{R}^2,\mathbb{C})$ and
		\begin{equation}\label{convmoll}
			u_R \to u \ \text{in } L^2(\mathbb{R}^2,\mathbb{C}) \quad \text{and} \quad  \nabla u_R \to \nabla u \ \text{in } L^2(\mathbb{R}^2,\mathbb{C}^{2})  \text{ as } R\to+\infty.
		\end{equation}
		If $\Omega$ is a bounded domain such that 
		$\operatorname{supp} u \subset \Omega$, for $R$ sufficiently large, by \cite[Proposition 4.18]{Brezis}, we also have 
		$\operatorname{supp}( u_R) \subset \overline{\operatorname{supp}(u) + \operatorname{supp}(j_R)} \subset \Omega$ (so that $u_R\in C_c^\infty(\mathbb{R}^2,\mathbb{C})$).
		Moreover, for any $x \in \mathbb{R}^2$,
		\[
		|u_R(x)|
		\le |u|_\infty \int_{\mathbb{R}^2} j_R(x-y) \, dy
		= |u|_\infty.
		\]
		Then, since
		$|A|^2 |u-u_R|^2 \le C |A|^2 \in L^1(\Omega,\mathbb{R})$,
		by dominated convergence we obtain $|A(u-u_R)|_2 \to 0$ and so, using \eqref{convmoll}, $|\nabla_{A}(u-u_R)|_2 \to 0$.
		Similarly, 
		we also have
		\[
		\int_{\mathbb{R}^2} V(x) |u-u_R|^2 dx \to 0
		\]
		so that $\|u-u_R\|_{A,V}\to 0$ as $R\to+\infty$.

		Step 2. The set of compactly supported functions in $H^1_{A,V}(\mathbb{R}^{2},\mathbb{C})$ is dense in $H^1_{A,V}(\mathbb{R}^{2},\mathbb{C})$.
		
		Let $u\in H^1_{A,V}(\mathbb{R}^{2},\mathbb{C})$, 
		$R\geq1$, and
		$\varphi_R \in C_c^\infty(\mathbb{R}^2,[0,1])$ be such that 
		$\varphi_R \equiv 1$ on $B_R(0)$, $\varphi_R \equiv 0$ on $\mathbb{R}^2\setminus B_{2R}(0)$, and $|\nabla\varphi_R|\leq C$, where $C>0$ is independent of $R$. One has
		\[
		\int_{\mathbb{R}^2} V(x) |u-\varphi_R u|^2 \,dx \leq \int_{|x|\ge R} V(x)|u|^2\,dx \to 0 \quad\text{as }R\to+\infty
		\]
		and
		\begin{align*}
			\int_{\mathbb{R}^2}|\nabla_{A}(\varphi_R u-u)|^2 dx
			& \le 	2 \int_{\mathbb{R}^2}\left((1-\varphi_R)^2|\nabla_{A}u|^2 + |\nabla\varphi_R|^2|u|^2\right)dx\\
			&\leq 
			2 \int_{|x|\geq R }|\nabla_{A}u|^2 dx + C
			\int_{ |x|\geq R} |u|^2 dx
			\to 0\quad \text{as}\,\ R\to +\infty.
		\end{align*}
		Hence $\|\varphi_R u-u\|_{A,V}\to0$ as $R\to +\infty$.

		Step 3. Conclusion.
		
		Given $u \in H^1_{A,V}(\mathbb{R}^{2},\mathbb{C})$, by the previous step we may assume that $u$ has compact support.  Let us show that $u$ can be approximated by a sequence of bounded functions ${u_k}\in H^1_{A,V}(\mathbb{R}^{2},\mathbb{C})$ such that $|u_k| \leq |u|$ a.e. in $\mathbb{R}^2$. Combined with Step 1, this completes the proof.
		
		Let $\phi \in C_c^\infty(\mathbb{R},[0,1])$ be such that $\phi(t)\equiv 1$ for $|t|\le 1$, $\phi(t)\equiv 0$ for $|t|\ge 2$, and define 
		\[\phi_k:=\phi(\cdot/k)
		\quad
		\text{and}
		\quad
		u_k:=u\phi_k(|u|)
		\quad
		\text{for } k\in\mathbb{N}.
		\]
		Then $|u_k|\le |u|$ and $|u_k|\leq 2k$ a.e. in $\mathbb{R}^2$.
		Since $A\in L^2_{\rm loc}(\mathbb R^2,\mathbb{R}^2)$ and $u\in L^2_{\rm loc}(\mathbb R^2,\mathbb{C})$, H\"older's inequality gives
		$Au\in L^1_{\rm loc}(\mathbb R^2,\mathbb{C}^2)$. Hence
		$\nabla u=i(\nabla_Au+Au)\in L^1_{\rm loc}(\mathbb R^2,\mathbb{C}^2)$ and
		$u\in W^{1,1}_{\rm loc}(\mathbb R^2,\mathbb C)$. Applying the 
		chain rule in $W^{1,1}_{\rm loc}$ to the Lipschitz map
		$z\mapsto z\phi_k(|z|)$, and using the diamagnetic inequality so that
		$|u|\in H^1(\mathbb R^2,\mathbb R)$, we obtain
		\[\nabla_{A} u_k = \phi_k(|u|)\nabla_{A} u + \frac{u}{i}\phi'_k(|u|)\nabla |u|
		\quad
		\text{a.e. in }\mathbb{R}^2.\]
		Now, since $|\phi_k|\le 1$ and $\phi_k(|u|)\to 1$ a.e. in $\mathbb{R}^2$, by the Lebesgue Dominated Convergence Theorem
		$
		\phi_k(|u|)\nabla_{A} u
		\to
		\nabla_{A} u$
		in $L^2(\mathbb{R}^2,\mathbb{C}^2)$.
		
		Moreover, being
		$\phi_k'=\phi'(\cdot/k)/k$ it follows that 
		$|u\,\phi_k'(|u|)| \le C\,\mathbf{1}_{\{k\leq |u|\leq 2 k\}},$
		for some constant $C>0$, so that $
		|u\,\phi_k'(|u|)\nabla |u|| \le C\,\mathbf{1}_{\{|u|\ge k\}}\,|\nabla |u||$ a.e. in $\mathbb{R}^2$. Then, the diamagnetic inequality and the Lebesgue Dominated Convergence Theorem imply
		$|u\,\phi_k'(|u|)\nabla |u||_{2} \to 0$ as $k\to+\infty$.
		Therefore, $\nabla_{A}u_k
		\to
		\nabla_{A}u$ in $L^2(\mathbb{R}^2,\mathbb{C}^2)$.
		Finally, since $|u_k|\le |u|$ and $u_k\to u$ a.e., 
		again by the Lebesgue Dominated Convergence Theorem,
		we can conclude.
	\end{proof}
	
	Additionally, by a standard change of variables, equation \eqref{p1.1} is equivalent to
	\begin{equation}\label{origi}
		\left(\frac{1}{i}\nabla-A_{\varepsilon}(x)\right)^{2}u+V_{\varepsilon}(x)u=f(|u|^{2})u\,\ \text{in}\,\ \mathbb{R}^{2},
	\end{equation}
	where $A_{\varepsilon}(x)=A({\varepsilon}x)$ and $V_{\varepsilon}(x)=V({\varepsilon}x)$.
	Thus, for any fixed $\varepsilon>0$ we consider $H_\varepsilon:=H^1_{A_\varepsilon,V_\varepsilon}(\mathbb{R}^2,\mathbb{C})$, denoting  
	\(
	\langle u, v \rangle_{\varepsilon} := \langle u, v \rangle_{A_\varepsilon,V_\varepsilon}
	\)
	and
	$\|u\|_{\varepsilon} := \sqrt{\langle u, u \rangle_{\varepsilon}}$.

	Since our main interest is in the case $A\in C(\mathbb{R}^2,\mathbb{R}^2)$, for simplicity we restrict ourselves to this setting from now on.
	
	Thanks to the diamagnetic inequality \eqref{p2.1}, if $u\in H_{\varepsilon}$, then $|u|\in H^{1}(\mathbb{R}^{2},\mathbb{R})$ and $\||u|\|\leq C\|u\|_{\varepsilon}$.
	Consequently, $H_{\varepsilon}$ is continuously embedded in
	$L^{r}(\mathbb{R}^{2},\mathbb{C})$ for every $2\leq r<+\infty$ and compactly
	embedded in $L_{\rm loc}^{r}(\mathbb{R}^{2},\mathbb{C})$ for every
	$1\leq r<+\infty$ (see, e.g., \cite[Lemma 2.1]{d'Avenia2023}).
	
	Regarding the nonlinearity, by \cite{d'Avenia2022} and assumptions 
	\eqref{f1}--\eqref{f2}, it follows that for a fixed $q>2$,
	for any $\xi>0$ and $\alpha>4\pi$, there exists a constant 
	$C=C(\alpha,\xi,q)>0$ such that
	\begin{equation}\label{p2.2}
		f(t)\leq\xi+Ct^{(q-2)/2}(e^{\alpha t}-1) \text{ for any }  t\geq0
	\end{equation}
	and by (\ref{f3}), we have
	\begin{equation}\label{p2.3}
		F(t)\leq\xi t+Ct^{q/2}(e^{\alpha t}-1) \text{ for any } t\geq0.
	\end{equation}
	In addition, by \eqref{p2.2} and \eqref{p2.3}, we deduce that
	\begin{equation}\label{p2.4}
		f(t^{2})t^{2}\leq\xi t^{2}+C|t|^{q}(e^{\alpha t^{2}}-1)  \text{ for any } t\in \mathbb{R}
	\end{equation}
	and
	\begin{equation}\label{p2.5}
		F(t^{2})\leq\xi t^{2}+C|t|^{q}(e^{\alpha t^{2}}-1)  \text{ for any }  t\in \mathbb{R}.
	\end{equation}
	A useful tool in the study of nonlinearities with exponential critical growth is the following version of the Trudinger-Moser inequality, as stated in \cite[Lemma 1.2]{Alves2016} (see also \cite[Lemma 2.1]{Cao} or \cite[Lemma 1]{J}).
	\begin{lemma}\label{2.1}
		If $\beta>0$ and $u\in H^{1}(\mathbb{R}^{2},\mathbb{R})$, then
		\[\int_{\mathbb{R}^{2}}(e^{\beta u^{2}}-1)dx<+\infty.\]
		Additionally, if $|\nabla u|_2^2\leq1$, $|u|_{2}\leq M<+\infty$, and $0<\beta<4\pi$, then there exists a positive constant $C=C(\beta,M)$ such that
		\[\int_{\mathbb{R}^{2}}(e^{\beta u^{2}}-1)dx\leq C.\]
	\end{lemma}
	
	In the same spirit, since for every $u\in H^{1}(\mathbb{R}^{2},\mathbb{C})$, by \cite[Theorem 6.17]{Lieb},
	$|u|\in H^{1}(\mathbb{R}^{2},\mathbb{R})$ and $|\nabla |u||_2\leq |\nabla u|_2$, applying the previous lemma to $|u|$ we get its {\em complex version} as follows.
	\begin{corollary}\label{2.1c}
		If $\beta>0$ and $u\in H^{1}(\mathbb{R}^{2},\mathbb{C})$, then
		\[\int_{\mathbb{R}^{2}}(e^{\beta |u|^{2}}-1)dx<+\infty.\]
		Additionally, if $|\nabla u|_2^2\leq1$, $|u|_{2}\leq M<+\infty$, and $0<\beta<4\pi$, then there exists a positive constant $C=C(\beta,M)$ such that
		\[\int_{\mathbb{R}^{2}}(e^{\beta |u|^{2}}-1)dx\leq C.\]
	\end{corollary}
	With such inequalities in hand, the following generalization of the Dominated Convergence Theorem will be useful.
	\begin{lemma}\label{GDCT}
		Let $E\subset\mathbb{R}^N$ be a measurable set, $\{f_n\}$ and $\{g_n\}$ be two sequences of measurable
		functions, and $f$ and $g$ be measurable functions such that $f_n,f:E\longrightarrow\mathbb{C}$, $g_n,g:E\longrightarrow[0,+\infty)$, $ |f_n|\leq g_n$, $f_n\to f$, and $g_n\to g$ a.e. in $E$.
		Assume, in addition, that $g_n,g\in L^1(E,\mathbb{R})$ and
		\[
		\lim_{n\to\infty}\int_E g_n\,dx=\int_E g\,dx.
		\]
		Then $ f_n\to f $ in $L^1(E,\mathbb{C})$.
	\end{lemma}
	
	\begin{proof}
		Passing to the pointwise limit we obtain $|f|\leq g$ a.e. in $E$. Then, $g_n+g-|f_n-f|\geq 0$  a.e. in $E$, and so Fatou's lemma and the assumed convergence of the integrals of $g_n$ gives 
		\begin{align*}
			2\int_E g\,dx
			&=\int_E \liminf_{n\to\infty} \bigl(g_n+g-|f_n-f|\bigr)\,dx
			\leq \liminf_{n\to\infty}\int_E\bigl(g_n+g-|f_n-f|\bigr)\,dx\\
			&= 2\int_E g\,dx -\limsup_{n\to\infty}\int_E|f_n-f|\,dx.
		\end{align*}
		Therefore
		\[
		\limsup_{n\to\infty}\int_E|f_n-f|\,dx\leq0
		\]
		and we can conclude.
	\end{proof}
	
	Next, we recall some tools that will be used in the sequel.

	In the spirit of \cite[Lemma 2.3]{-Arioli2003} and \cite[Lemma 2.1, Corollary 2.2]{jeanjean}, the continuity (and, in particular, the local boundedness) of both $A$ and $V$ implies the following useful results.
	It is worth noting that, although the present proof is carried out in the case $N=2$, it actually applies to all dimensions 
	$N\geq3$ as well, and provides a simpler argument than that in \cite{-Arioli2003,jeanjean}.
	
	\begin{lemma}\label{coro}
		If $A\in C(\mathbb{R}^2,\mathbb{R}^2)$ and $V\in C(\mathbb{R}^2,\mathbb{R})$ satisfies (\ref{V1}), the following equivalences hold.
		\begin{enumerate}[label=\textup{(\roman*)},ref=\roman*]
			\item \label{equivi} For  \( \varepsilon  \) small enough,  the norm
			\[\left(\int_{B_{2|\log\varepsilon|}(0)}(|\nabla_{A_{\varepsilon}} u|^2+V_{\varepsilon}(x)|u|^{2})dx\right)^{\frac{1}{2}}\]
			is equivalent to the standard norm of $H^{1}(B_{2|\log\varepsilon|}(0), \mathbb{C})$. Moreover, this equivalence is uniform, i.e.
			there exist $\varepsilon_0\in(0,1)$ and $C_{1},C_{2}>0$ such that, for every $\varepsilon\in (0,\varepsilon_0]$, 
			\begin{align*}
				C_1 \int_{B_{2|\log\varepsilon|}(0)}(|\nabla u|^2+|u|^{2})dx
				&\leq
				\int_{B_{2|\log\varepsilon|}(0)}(|\nabla_{A_{\varepsilon}} u|^2+V_{\varepsilon}(x)|u|^{2})dx
				\leq 
				C_2 \int_{B_{2|\log\varepsilon|}(0)}(|\nabla u|^2+|u|^{2})dx.
			\end{align*}
			\item \label{equivii} If $K\subset \mathbb{R}^2$ is a bounded open set, then for  $\varepsilon>0$ small enough, the norm
			\[\left(\int_{K}(|\nabla_{A_{\varepsilon}} u|^2+V_{\varepsilon}(x)|u|^{2})dx\right)^{\frac{1}{2}}\]
			is uniformly equivalent to the standard norm of $H^{1}(K, \mathbb{C})$.
			\item \label{equiviii} For $A_{0}\in \mathbb{R}^{2}$ and $b>0$ fixed, the norm
			\[\left(\int_{\mathbb{R}^{2}}(|\nabla_{A_0} u|^2+b|u|^{2})dx\right)^{\frac{1}{2}}\]
			is  equivalent to the standard norm of $H^{1}(\mathbb{R}^{2}, \mathbb{C})$.
		\end{enumerate}
	\end{lemma}
	\begin{proof}
		Let us prove (\ref{equivi}).
		Observe that, if $x\in B_{2|\log\varepsilon|}(0)$, then $\varepsilon x\in B_{2\varepsilon|\log\varepsilon|}(0)$, with $\varepsilon|\log\varepsilon|\to 0^+$ as $\varepsilon\to 0^+$. Then, for $\varepsilon$ small enough, by the continuity of $A$ and $V$, there exists a constant $C>0$ (independent of $\varepsilon$) such that
		\begin{equation}\label{bddAeVe}
			\int_{B_{2|\log\varepsilon|}(0)} |A_\varepsilon(x)u|^2 dx \leq C \int_{B_{2|\log\varepsilon|}(0)} |u|^2 dx\,\ \text{and}\,\ \int_{B_{2|\log\varepsilon|}(0)} V_\varepsilon(x)|u|^2 dx \leq C \int_{B_{2|\log\varepsilon|}(0)} |u|^2 dx.
		\end{equation}
		Consequently, 
		\[
		\begin{aligned}
			\int_{B_{2|\log\varepsilon|}(0)} (\left|  \nabla_{A_\varepsilon} u \right|^2 + V_\varepsilon(x)|u|^2) dx & \leq \int_{B_{2|\log\varepsilon|}(0)} \big(2\left( |\nabla u|^2 + |A_\varepsilon(x)u|^2 \right) + V_\varepsilon(x)|u|^2\big) dx \\
			& \leq 2 \int_{B_{2|\log\varepsilon|}(0)} |\nabla u|^2 dx + C\int_{B_{2|\log\varepsilon|}(0)} |u|^2 dx.
		\end{aligned}
		\]
		Moreover, using \eqref{eaineq}, \eqref{bddAeVe}, and (\ref{V1}),
		\begin{align*}
			\int_{B_{2|\log\varepsilon|}(0)}(|\nabla u|^2+|u|^2) dx
			&\le
			\int_{B_{2|\log\varepsilon|}(0)}
			(2|\nabla_{A_\varepsilon}u|^2
			+2|A_\varepsilon(x)u|^2+|u|^2)
			dx\\
			&\leq
			2\int_{B_{2|\log\varepsilon|}(0)}
			|\nabla_{A_\varepsilon}u|^2 dx
			+\frac{2C+1}{V_0}\int_{B_{2|\log\varepsilon|}(0)}V_0|u|^2dx
			\\
			&\le
			C\int_{B_{2|\log\varepsilon|}(0)}
			(|\nabla_{A_\varepsilon}u|^2+V_\varepsilon(x)|u|^2)
			dx.
		\end{align*}

		To prove (\ref{equivii}) we proceed analogously, using that, if $\varepsilon$ is small enough, due to the continuity of $A$ and $V$, there exists $C>0$ such that
		\[
		\int_{K} |A_\varepsilon(x)u|^2 dx  \leq C \int_{K} |u|^2 dx\,\ \text{and}\,\ \int_{K} V_\varepsilon(x)|u|^2 dx\leq C \int_{K} |u|^2 dx.
		\]

		Finally the same arguments can be used to prove (\ref{equiviii}).
	\end{proof}

	To conclude this section, for every
	$\mu>0$, let us consider the following real valued limit problem
	\begin{equation}\label{shide}
		-\Delta u+\mu u=f(u^{2})u,
		\quad u\in H^1(\mathbb{R}^{2},\mathbb{R}),
	\end{equation}
	whose associated $C^{1}$-functional, defined on $H^{1}(\mathbb{R}^{2},\mathbb{R})$, is given by
	\[I_{\mu}(u)=\frac{1}{2}\int_{\mathbb{R}^{2}}(|\nabla u|^{2}+\mu u^{2})dx-\frac{1}{2}\int_{\mathbb{R}^{2}}F(u^{2})dx,\]
	and the corresponding complex valued limit problem
	\begin{equation}\label{fude}
		-\Delta u+\mu u=f(|u|^2)u,
		\quad u\in H^1(\mathbb{R}^{2},\mathbb{C}),
	\end{equation}
	whose associated $C^{1}$-functional defined on $H^{1}(\mathbb{R}^{2},\mathbb{C})$ is given by
	\[I^c_{\mu}(u)=\frac{1}{2}\int_{\mathbb{R}^{2}}(|\nabla u|^{2}+\mu|u|^{2})dx-\frac{1}{2}\int_{\mathbb{R}^{2}}F(|u|^{2})dx.\]
	By assumptions (\ref{f1}) and (\ref{f4}), for any $u\in H^{1}(\mathbb{R}^{2},\mathbb{R})\setminus\{0\}$ (respectively $u\in H^{1}(\mathbb{R}^{2},\mathbb{C})\setminus\{0\}$), there exists a unique $t_{\mu}(u)>0$ (respectively $t^c_{\mu}(u)>0$) such that 
	\begin{equation*}
		I_{\mu}(t_{\mu}(u)u)=\max_{t\geq0}I_{\mu}(tu) \text{ and } t_{\mu}(u)u\in \mathcal{N}_{\mu}
	\end{equation*}
	where 
	\[\mathcal{N}_{\mu}=\{u\in H^{1}(\mathbb{R}^{2},\mathbb{R})\setminus\{0\}:I_\mu{}'(u)[u]=0\}\]
	(respectively
	\begin{equation}\label{IcNeh}
		I^c_{\mu}(t^c_{\mu}(u)u)=\max_{t\geq0}I^c_{\mu}(tu) \text{ and } t^c_{\mu}(u)u\in \mathcal{N}^c_{\mu}
	\end{equation}
	with 
	\[\mathcal{N}^c_{\mu}=\{u\in H^{1}(\mathbb{R}^{2},\mathbb{C})\setminus\{0\}:I_{\mu}^c{}'(u)[u]=0\}\]
	in the complex valued case).
	
	Then, arguing as in \cite[Lemma 4.1 and Theorem 4.2]{willem1997}, we also have that
	\begin{equation*}
		0< c_{\mu}
		:=\inf_{u\in\mathcal{N}_{\mu} } I_{\mu}(u)
		=\inf_{u\in H^{1}(\mathbb{R}^{2},\mathbb{R})\setminus\{0\}} \max_{t\geq0}I_{\mu}(tu)
	\end{equation*}
	and
	\begin{equation}\label{jitaic}
		0
		< c_{\mu}^c
		:=\inf_{u\in\mathcal{N}_{\mu}^c } I^c_{\mu}(u)
		=\inf_{u\in H^{1}(\mathbb{R}^{2},\mathbb{C})\setminus\{0\}} \max_{t\geq0}I^c_{\mu}(tu).
	\end{equation}
	
	In addition, recalling that a ground state solution is a nontrivial solution that minimizes the associated energy functional among all possible nontrivial solutions,
	using in particular (\ref{f4}), we know the following result.
	
	\begin{lemma}[{\cite[Lemma 3.1]{d'Avenia2023}}]\label{existence}
		For any $\mu\in(0, V^{*}]$, there exists a positive and radially symmetric ground state solution $w_{\mu}\in H^1(\mathbb{R}^2,\mathbb{R})$ for problem \eqref{shide}.
	\end{lemma}
	
	\begin{remark}\label{reggs}
		The ground state solution $w_{\mu}$ belongs to $C^{2}(\mathbb{R}^{2},\mathbb{R})\cap L^{\infty}(\mathbb{R}^{2},\mathbb{R})$, and both $w_{\mu}$ and its gradient decay exponentially at infinity (see \cite[Proposition 2.1]{Zhang}).
		Moreover, for any $0<\mu_{1}<\mu_{2}\leq V^{*}$, 
		we have that 
		\begin{equation*}
			0<c_{\mu_1}=I_{\mu_{1}}(w_{\mu_{1}})\leq \mathop{\max}\limits_{t\geq0}I_{\mu_{1}}(tw_{\mu_{2}})<\mathop{\max}\limits_{t\geq0}I_{\mu_{2}}(tw_{\mu_{2}})=I_{\mu_{2}}(w_{\mu_{2}})=c_{\mu_{2}}.
		\end{equation*}
	\end{remark}
	
	Analogously, in the complex case we have the following result. 
	
	\begin{lemma}\label{existenceC}
		For any $\mu\in(0, V^{*}]$, there exists a real-valued, nonnegative, radially nonincreasing ground state solution of \eqref{fude}. Moreover
		\begin{equation}\label{cmu}
			c_\mu=c_\mu^c<\frac{\Theta-2}{16m\Theta}.
		\end{equation}
	\end{lemma}
	For completeness we give the proof of the previous Lemma in Appendix \ref{proofexistenceC}, which is based on the proof of \cite[Theorem 1.4, Corollary 1.5]{Alves2012}.
	
	Hence, combining Lemmas \ref{existence} and \ref{existenceC} we have the following consequence. 
	\begin{corollary}\label{existencecor}
		For any $\mu\in(0, V^{*}]$, there exists a positive and radially symmetric ground state solution $w_{\mu}\in H^1(\mathbb{R}^2,\mathbb{R})\cap C^{2}(\mathbb{R}^{2},\mathbb{R})\cap L^{\infty}(\mathbb{R}^{2},\mathbb{R})$ for problem \eqref{fude} which decays exponentially at infinity together with its gradient. Moreover, its energy level $c_\mu$ satisfies \eqref{cmu} and is increasing with respect to $\mu$.
	\end{corollary}

	\section{The modified problem}\label{Secthree}
	From now on, let us fix a nonempty set $\Gamma\subset \{1,2,...,m\}$. Without loss of generality,  we may assume that $\Gamma=\{1,2,...,s\}$ for some $1\leq s\leq m$.
	Following the approach in \cite{Pino}, we introduce an appropriate penalization by modifying  the nonlinearity $f$ outside the domain $\Lambda_\Gamma:=\bigcup_{j\in \Gamma}\Lambda_{j}$ so that the concentration outside the domain $\Lambda_{\Gamma}$ is avoided and an adequate balance in the concentration is achieved. We then show that the modified functional satisfies the Palais-Smale (PS) condition. 
	
	First observe that, according to Corollary \ref{existencecor}, there exists $k>0$ sufficiently large such that, for $\mu\in (0,V^{*}]$,
	\begin{equation}\label{p3.1}
		0<c_{\mu}\leq c_{V^{*}}<\frac{1}{8m}\left(\frac{1}{2}-\frac{1}{\Theta}-\frac{1}{2k}\right)<\frac{\Theta-2}{16m\Theta}.
	\end{equation}
	Then, by the assumptions on $f$, there exists a unique $a>0$ such that $k f(a)=V_{0}$, where $V_0$ is given in \eqref{V1}.
	Hence we consider the function
	\begin{align*}
		\widehat{f}(t):=
		\begin{cases}
			f(t), &  t\leq a,\\
			V_0/k, &  t>a,
		\end{cases}
		\qquad
		\widehat{F}(t):=\int_0^t\widehat{f}(s)ds,
	\end{align*}
	and introduce the penalized nonlinearity
	\[g(x, t):=\chi_{\Lambda_\Gamma}(x)f(t)+(1-\chi_{\Lambda_\Gamma}(x))\widehat{f}(t),\]
	where $\chi_{\Lambda_\Gamma}$ denotes the characteristic function on $\Lambda_\Gamma$.
	
	In view of \eqref{f1}-\eqref{f5}, the function $g(x,t)$ satisfies the following properties:
	
	\begin{enumerate}[label=($g_\arabic*$),ref=$g_\arabic*$]
		
		\item \label{g1} \( g(x, t) = 0 \) if \( t \leq 0 \);
		
		\item \label{g2} \( \mathop{\lim}\limits_{t \to 0^+}  g(x, t) = 0 \) uniformly in \( x \in \mathbb{R}^2 \);
		
		\item \label{g3} \( g(x, t) \leq f(t) \) for all \( t \geq 0 \) and uniformly in \( x \in \mathbb{R}^2 \);
		
		\item \label{g4} \( 0 < \Theta G(x, t) \leq 2 g(x, t) t \) for all \( t > 0 \), \( x \in \Lambda_\Gamma \), where \( \displaystyle G(x, t) := \int_0^t g(x, s) \, ds \);
		
		\item \label{g5} \( 0 < G(x, t) \leq g(x, t) t \leq \displaystyle \frac{V_0}{k} t \) for all \( t > 0 \), \( x \in \Lambda_\Gamma^c \);

		\item \label{g6} for each \( x \in \Lambda_\Gamma \), the function \( t \mapsto g(x, t) \) increases strictly in \( t \in (0, +\infty) \) and for each \( x \in \Lambda_\Gamma^c \), the function \( t \mapsto g(x, t) \) increases strictly in \( t \in (0, a) \);
		
		\item \label{g7} 
		for every fixed \(x\in\mathbb R^2\), the map \(t\mapsto g(x,t)\) is
		differentiable for a.e. \(t>0\), and
		\[
		0\leq\partial_tg(x,t)\leq e^{4\pi t}-1
		\]
		at every differentiability point.
	\end{enumerate}

	Then we consider the \emph{modified problem}
	\begin{equation}\label{p3.2}
		\left(\frac{1}{i}\nabla-A_{\varepsilon}(x)\right)^{2}u+V_{\varepsilon}(x)u=g(\varepsilon x, |u|^{2})u\,\ \text{in}\,\ \mathbb{R}^{2}.
	\end{equation}
	
	\begin{remark}\label{remark}
		Note that, if $u$ is a solution of problem \eqref{p3.2} with
		\[|u(x)|^{2}\leq a \text{ for a.e.} \quad x\in \Lambda_{\Gamma, \varepsilon}^{c},\,\ \Lambda_{\Gamma, \varepsilon}:=\{x\in\mathbb{R}^{2}:\varepsilon x\in \Lambda_\Gamma\},\]
		then $u$ is a solution of problem \eqref{origi}.
	\end{remark}

	The functional associated to problem \eqref{p3.2} is
	\[J_{\varepsilon}(u):=\frac{1}{2}\int_{\mathbb{R}^{2}}(|\nabla_{A_{\varepsilon}} u|^{2}+V_{\varepsilon}(x)|u|^{2})dx-\frac{1}{2}\int_{\mathbb{R}^{2}}G(\varepsilon x, |u|^{2})dx,\,\ u\in H_{\varepsilon}.\]
	By the assumptions on $f$ and using standard tools we can verify that
	$J_{\varepsilon} \in C^{1}(H_{\varepsilon}, \mathbb{R})$, and its nonzero critical points are nontrivial weak solutions of the modified problem \eqref{p3.2}.
	Moreover, the functional $J_{\varepsilon}$  satisfies the Mountain Pass Geometry (see \cite[Lemma 3.1]{d'Avenia2022}) and there exists  a Palais-Smale sequence of $J_{\varepsilon}$ at level $c_{\varepsilon}$ ((PS)$_{c_{\varepsilon}}$), where \[c_{\varepsilon}:=\mathop{\inf}\limits_{\gamma\in \Gamma_{\varepsilon}}\mathop{\sup}\limits_{t\in [0,1]}J_{\varepsilon}(\gamma(t))
	\quad
	\text{and}
	\quad
	\Gamma_{\varepsilon}=\{\gamma\in C([0,1], H_{\varepsilon}): \gamma(0)=0\,\ \text{and}\,\ J_{\varepsilon}(\gamma(1))<0\}.
	\]
	
	If we consider $\mathcal{N}_{\varepsilon}:=\{u\in H_{\varepsilon}\setminus\{0\}:J^{\prime}_{\varepsilon}(u)[u]=0\}$, for such a value we have the following characterization.
	\begin{lemma}\label{lem_same_levels}
		For every $\varepsilon>0$,
		\begin{equation*}
			c_{\varepsilon}=
			\mathop{\inf}\limits_{u\in H_{\varepsilon}\setminus\{0\}}
			\mathop{\sup}\limits_{t\geq0}J_{\varepsilon}(tu)
			=\mathop{\inf}\limits_{\substack{u\in H_\varepsilon\setminus\{0\}\\
					|\operatorname{supp}u\cap\Lambda_{\Gamma,\varepsilon}|\neq0}}
			\mathop{\max}\limits_{t\geq0}J_\varepsilon(tu)
			=\mathop{\inf}\limits_{u\in\mathcal{N}_{\varepsilon}}J_{\varepsilon}(u).
		\end{equation*}
	\end{lemma}
	\begin{proof}
		For every $u\in H_{\varepsilon}\setminus\{0\}$ with
		$|\operatorname{supp}u\cap\Lambda_{\Gamma,\varepsilon}|\neq0$, by \eqref{f4} and the definition of $g(x,t)$, we have
		\[
		J_{\varepsilon}(tu)
		\leq \frac{t^2}{2}\|u\|_\varepsilon^2
		-\frac{1}{2}\int_{\Lambda_{\Gamma, \varepsilon}}F(|tu|^{2})dx
		\leq
		\frac{t^2}{2}\|u\|_\varepsilon^2-\frac{C_p}{p}t^p\int_{\Lambda_{\Gamma, \varepsilon}}|u|^pdx,
		\]
		which implies that
		\[J_{\varepsilon}(tu)\rightarrow -\infty\,\ \text{as}\,\ t\rightarrow+\infty.\]
		Moreover, for every $u\in H_{\varepsilon}\setminus\{0\}$ with
		$|\operatorname{supp}u\cap\Lambda_{\Gamma,\varepsilon}|\neq0$ and $t>0$, we have that $J^{\prime}_{\varepsilon}(tu)[tu]=0$ if and only if
		\begin{equation}\label{Neder}
			\|u\|_{\varepsilon}^2
			-
			\int_{\Lambda_{\Gamma, \varepsilon}}g(\varepsilon x, |tu|^{2})|u|^2dx
			-\int_{\Lambda^c_{\Gamma, \varepsilon}}g(\varepsilon x, |tu|^{2})|u|^2dx=0.
		\end{equation}
		By \eqref{g6} and since
		$|\operatorname{supp}u\cap\Lambda_{\Gamma,\varepsilon}|\neq0$, the left hand side of \eqref{Neder} is a strictly decreasing function in $t$. Moreover it is positive for $t>0$ small enough and, since, by (\ref{f4}),
		\[
		\int_{\Lambda_{\Gamma, \varepsilon}}g(\varepsilon x, |tu|^{2})|u|^2dx
		+\int_{\Lambda^c_{\Gamma, \varepsilon}}g(\varepsilon x, |tu|^{2})|u|^2dx
		\geq
		\int_{\Lambda_{\Gamma, \varepsilon}}f(|tu|^{2})|u|^2dx
		\geq C_pt^{p-2} \int_{\Lambda_{\Gamma, \varepsilon}}|u|^p dx,
		\]
		it is negative for $t>0$ large. Then there exists a unique $t_{\varepsilon}(u)>0$ such that $\ t_{\varepsilon}(u)u\in\mathcal{N}_{\varepsilon}$ and the sign of the left hand side of \eqref{Neder}, 
		implies that
		\begin{equation*}
			\mathop{\max}\limits_{t\geq0}J_{\varepsilon}(tu)=J_{\varepsilon}(t_{\varepsilon}(u)u).
		\end{equation*}
		
		For each $u\in H_{\varepsilon}\setminus\{0\}$ with
		$|\operatorname{supp}u\cap\Lambda_{\Gamma,\varepsilon}|=0$, since
		\[g(\varepsilon x, |tu|^2)=\widehat{f}(|tu|^2)\leq \frac{V_0}{k}\leq \frac{V(\varepsilon x)}{k} \text{ for } t> 0 \text{ and }x\in\Lambda_{\Gamma, \varepsilon}^c,\]
		using \eqref{g5}
		we have 
		\[
		J_\varepsilon(tu)
		\geq\frac{t^2}{2}\|u\|_\varepsilon^2
		-\frac{V_0t^2}{2k}\int_{\Lambda_{\Gamma, \varepsilon}^c}|u|^{2}dx
		\geq
		\frac{t^2}{2}|\nabla_{A_\varepsilon} u|_2^2
		+\frac{t^2}{2}\left(1-\frac{1}{k}\right)\int_{\Lambda_{\Gamma, \varepsilon}^c}V(\varepsilon x)|u|^{2}dx
		\to +\infty
		\text{ as }t\to+\infty,
		\]
		so that
		\[
		\sup_{t\geq0}J_\varepsilon(tu)=+\infty.
		\]
		Thus these {\em directions} do not affect the infimum over
		$H_\varepsilon\setminus\{0\}$ and then
		\begin{equation}\label{geqce}
			\inf_{u\in H_\varepsilon\setminus\{0\}}
			\sup_{t\ge0}J_\varepsilon(tu)
			=\inf_{\substack{u\in H_\varepsilon\setminus\{0\}\\
					|\operatorname{supp}u\cap\Lambda_{\Gamma,\varepsilon}|\neq0}}
			\max_{t\ge0} J_\varepsilon(tu)
			\geq \inf_{u\in\mathcal{N}_{\varepsilon}}J_{\varepsilon}(u).
		\end{equation}
		Moreover, if $u\in \mathcal{N}_{\varepsilon}$, then $t=1$ is a critical point of the map
		$t>0\mapsto J_\varepsilon(tu)$. Thus, since if
		$|\operatorname{supp}u\cap\Lambda_{\Gamma,\varepsilon}|=0$, for $t>0$ we have
		\begin{align*}
			J^{\prime}_{\varepsilon}(tu)[u]
			\geq t|\nabla_{A_\varepsilon} u|_2^2
			+t\left(1-\frac{1}{k}\right)\int_{\Lambda_{\Gamma, \varepsilon}^c}V(\varepsilon x)|u|^{2}dx
			>0,
		\end{align*}
		it follows that if $u\in \mathcal{N}_{\varepsilon}$, then
		$|\operatorname{supp}u\cap\Lambda_{\Gamma,\varepsilon}|\neq0$, and
		\(
		J_{\varepsilon}(u)= \max_{t\ge0} J_\varepsilon(tu).
		\)
		Then
		\[
		\inf_{\mathcal{N}_{\varepsilon}}J_{\varepsilon}(u)
		\geq
		\inf_{\substack{u\in H_\varepsilon\setminus\{0\}\\ |\operatorname{supp}u\cap\Lambda_{\Gamma,\varepsilon}|\neq0}}
		\max_{t\ge0} J_\varepsilon(tu)
		\]
		which, combined with \eqref{geqce}, implies
		\[
		\inf_{u\in H_\varepsilon\setminus\{0\}}
		\sup_{t\ge0}J_\varepsilon(tu)
		=\inf_{\substack{u\in H_\varepsilon\setminus\{0\}\\
				|\operatorname{supp}u\cap\Lambda_{\Gamma,\varepsilon}|\neq0}}
		\max_{t\ge0}J_\varepsilon(tu)
		=
		\inf_{u\in\mathcal{N}_{\varepsilon}}J_{\varepsilon}(u).
		\]
		
		Next, since $J_\varepsilon(tu)<0$ for any $u\in H_{\varepsilon}\setminus\{0\}$ with
		$|\operatorname{supp}u\cap\Lambda_{\Gamma,\varepsilon}|\neq0$ and for $t$ sufficiently large, it follows that
		\[
		\inf_{\substack{u\in H_\varepsilon\setminus\{0\}\\ |\operatorname{supp}u\cap\Lambda_{\Gamma,\varepsilon}|\neq0}}
		\max_{t\ge0} J_\varepsilon(tu)
		\geq c_\varepsilon.
		\]
		Moreover, for any $\gamma\in\Gamma_\varepsilon$, we define $\beta(t):=J_\varepsilon'(\gamma(t))[\gamma(t)]$ for $t\in[0,1]$. 
		We first show that $\beta(t)>0$ whenever $t\in(0,1)$ and $0<\|\gamma(t)\|_\varepsilon$ is sufficiently small, which is possible since $\gamma(0)=0$, $\gamma(1)\neq 0$, and $\gamma$ is continuous.
		
		By \eqref{g3} and \eqref{p2.4}, fixed \(q>2\), for any \(\xi>0\) and \(\alpha>4\pi\), there exists \(C>0\) such that  
		\begin{equation*}
			g(\varepsilon x,s^2)s^2 \le \xi s^2 + C |s|^q (e^{\alpha s^2}-1) \quad \text{for all } s\in\mathbb{R}\text{ and a.e. }x\in\mathbb{R}^2, 
		\end{equation*}
		so that, using the Sobolev and H\"older inequalities and
		\begin{equation}\label{e}
			(e^\tau - 1)^\sigma \leq e^{\tau \sigma} - 1, \text{ for } \tau \geq 0 \text{ and } \sigma > 1,
		\end{equation}
		for every $t\in[0,1]$,
		\begin{align*}
			\int_{\mathbb{R}^2} g(\varepsilon x,|\gamma(t)|^2)|\gamma(t)|^2 dx
			&\leq
			\frac{\xi}{V_0} \|\gamma(t)\|_\varepsilon^2
			+ C 
			|\gamma(t)|_{2q}^q \Bigl( \int_{\mathbb{R}^2} (e^{\alpha |\gamma(t)|^2}-1)^2 dx \Bigr)^{1/2}\\
			&\le
			\frac{\xi}{V_0} \|\gamma(t)\|_\varepsilon^2
			+C \|\gamma(t)\|_\varepsilon^q \Bigl( \int_{\mathbb{R}^2} (e^{2\alpha |\gamma(t)|^2}-1) dx \Bigr)^{1/2}. 
		\end{align*}
		Now observe that, by the diamagnetic inequality \eqref{p2.1}, if $\gamma(t)\neq 0$,
		\[
		\frac{|\gamma(t)|}{\|\gamma(t)\|_\varepsilon}\in H^1(\mathbb{R}^2,\mathbb{R}),\quad
		\Bigl|\frac{|\gamma(t)|}{\|\gamma(t)\|_\varepsilon}\Bigr|_2^2\le\frac1{V_0},\quad
		\Bigl|\nabla\frac{|\gamma(t)|}{\|\gamma(t)\|_\varepsilon}\Bigr|_2^2\le1.
		\]  
		Hence, being \(\|\gamma(t)\|_\varepsilon\) small\footnote{Here we need that $2\alpha\|\gamma(t)\|_\varepsilon^2<4\pi$.}, by Lemma \ref{2.1},  
		\begin{equation*}
			\int_{\mathbb{R}^2} (e^{2\alpha |\gamma(t)|^2}-1)\,dx 
			= \int_{\mathbb{R}^2} \Bigl( e^{2\alpha \|\gamma(t)\|_\varepsilon^2 \bigl(\frac{|\gamma(t)|}{\|\gamma(t)\|_\varepsilon}\bigr)^2} -1 \Bigr) dx
			\le  C.
		\end{equation*}
		Then, taking $0<\xi<V_0$, whenever $0<\|\gamma(t)\|_\varepsilon$ is sufficiently small,
		\[
		\beta(t) = \|\gamma(t)\|_\varepsilon^2 - \int_{\mathbb{R}^2} g(\varepsilon x,|\gamma(t)|^2)|\gamma(t)|^2\,dx \ge \Bigl(1-\frac{\xi}{V_0}\Bigr)\|\gamma(t)\|_\varepsilon^2 - C\|\gamma(t)\|_\varepsilon^q >0.
		\] 
		Therefore there exists $t_0\in(0,1)$ such that $0<\|\gamma(t_0)\|_\varepsilon$ is sufficiently small and $\beta(t_0)>0$.
		
		On the other hand, 
		since for every $u\in H_\varepsilon$, using (\ref{g4}) and (\ref{g5}),
		\[
		\int_{\mathbb{R}^2} G(\varepsilon x,|u|^2)dx \leq \int_{\mathbb{R}^2} g(\varepsilon x,|u|^2)|u|^2dx,
		\]
		and $J_\varepsilon(\gamma(1))<0$, we get
		\[
		\beta(1)\leq 2 J_\varepsilon(\gamma(1))<0.
		\]
		Thus, the continuity of $\beta$ implies that there exists $t\in(t_0,1)$ such that $\beta(t)=0$, i.e. $\gamma(t)\in\mathcal{N}_\varepsilon$.
		Hence,
		\[
		\sup_{t\in [0,1]}J_{\varepsilon}(\gamma(t))
		\geq \inf_{\mathcal{N}_{\varepsilon}}J_{\varepsilon}(u)
		\]
		and so, taking the infimum on $\Gamma_\varepsilon$, we deduce that
		\[
		c_{\varepsilon}\geq \inf_{u\in\mathcal{N}_{\varepsilon}}J_{\varepsilon}(u),
		\]
		which allows us to conclude the proof.
	\end{proof}

	Note that the function $f$ satisfies all the properties given for $g$, except for condition \eqref{g5}.  
	This condition is essential in ensuring that $J_{\varepsilon}$ satisfies the $(PS)$ condition, as shown in the next lemma.
	\begin{lemma}\label{3.1}
		The functional $J_{\varepsilon}$ satisfies the $(PS)_d$ condition at any positive level 
		$\displaystyle d<\left(\frac{1}{2}-\frac{1}{\Theta}-\frac{1}{2k}\right).$
	\end{lemma}
	\begin{proof}
		Let $\{u_{n}\}$ be a $(PS)_d$ sequence for $J_\varepsilon$. By
		conditions \eqref{g4} and \eqref{g5}, we obtain
		\begin{equation}\label{bddPS}
			\begin{split}
				d+o_{n}(1)+o_{n}(1)\|u_{n}\|_{\varepsilon}\geq{}&J_{\varepsilon}(u_{n})-\frac{1}{\Theta}J_{\varepsilon}^{\prime}(u_{n})[u_{n}]\\
				=&{}\left(\frac{1}{2}-\frac{1}{\Theta}\right)\|u_{n}\|_{\varepsilon}^{2}
				+\int_{\mathbb{R}^{2}}\left(\frac{1}{\Theta}g(\varepsilon x, |u_{n}|^{2})|u_{n}|^{2}-\frac{1}{2}G(\varepsilon x, |u_{n}|^{2})\right)dx\\
				\geq&{}\left(\frac{1}{2}-\frac{1}{\Theta}\right)\|u_{n}\|_{\varepsilon}^{2}
				+\int_{\Lambda_{\Gamma, \varepsilon}^{c}}\left(\frac{1}{\Theta}g(\varepsilon x, |u_{n}|^{2})|u_{n}|^{2}-\frac{1}{2}G(\varepsilon x, |u_{n}|^{2})\right)dx\\
				\geq&{}\left(\frac{1}{2}-\frac{1}{\Theta}\right)\|u_{n}\|_{\varepsilon}^{2}
				-\frac{1}{2}\int_{\Lambda_{\Gamma, \varepsilon}^{c}}G(\varepsilon x, |u_{n}|^{2})dx\\
				\geq&{}\left(\frac{1}{2}-\frac{1}{\Theta}\right)\|u_{n}\|_{\varepsilon}^{2}
				-\frac{1}{2k}\int_{\mathbb{R}^{2}}V_{\varepsilon}(x)|u_{n}|^{2}dx\geq\left(\frac{1}{2}-\frac{1}{\Theta}-\frac{1}{2k}\right)\|u_{n}\|_{\varepsilon}^{2},
			\end{split}
		\end{equation}
		which implies that $\|u_{n}\|_{\varepsilon}$ is bounded and, for $n$ large enough,
		\[\left(\frac{1}{2}-\frac{1}{\Theta}-\frac{1}{2k}\right)\|u_{n}\|_{\varepsilon}^{2}\leq d+o_n(1)<\left(\frac{1}{2}-\frac{1}{\Theta}-\frac{1}{2k}\right)\]
		so that
		\[\ell^2:=\limsup_{n\to +\infty} |\nabla |u_n||_2^2 \leq
		\limsup_{n\to +\infty} \Bigl( |\nabla |u_n||_2^2 + V_0 |u_n|_2^2 \Bigr)\leq\limsup_{n\to +\infty}\|u_n\|^2_\varepsilon
		\leq\left(\frac{1}{2}-\frac{1}{\Theta}-\frac{1}{2k}\right)^{-1}d<1.\]
		Thus, up to a subsequence, \(u_{n} \rightharpoonup u\) in \(H_{\varepsilon}\) and 
		\(u_{n} \to u\) in \(L_{\mathrm{loc}}^{q}(\mathbb{R}^{2}, \mathbb{C})\) for all \(1\leq q<+\infty\) (so that $u_n\to u$ a.e. in $\mathbb{R}^2$) as \(n \to +\infty\).
		
		Now observe that, for every \(\phi \in C_c^\infty(\mathbb{R}^2, \mathbb{C})\),
		\[
		g(\varepsilon x, |u_{n}|^{2}) u_{n} \overline{\phi}
		\to g(\varepsilon x, |u|^{2}) u \overline{\phi}
		\quad \text{a.e. in } \mathbb{R}^2 \text{ as } n \to +\infty,
		\]
		by \eqref{g3} and \eqref{p2.2}, fixed \(q > 2\), for any \(\xi > 0\) and \(\alpha > 4\pi\), there exists a constant $C=C (q, \alpha, \xi)>0$, such that
		\begin{align*}
			|g(\varepsilon x,|u_n|^2)u_n \overline{\phi}| &\leq f(|u_n|^2) |u_n||\phi| \leq \xi |u_n||\phi| + C|u_n|^{q-1}(e^{\alpha|u_n|^2} - 1)|\phi|\\
			&\to \xi |u||\phi| + C|u|^{q-1}(e^{\alpha|u|^2} - 1)|\phi| \quad \text{a.e. in } \mathbb{R}^2 \text{ as } n \to +\infty,
		\end{align*}
		and the function whose limit is being evaluated and its limit are   in $L^1(\mathbb{R}^2,\mathbb{R})$, since, by the H\"older and Sobolev inequalities, \eqref{e}, and Lemma \ref{2.1},
		\begin{equation*}
			\int_{\mathbb{R}^2} \big(\xi |u||\phi| + C|u|^{q-1}(e^{\alpha|u|^2} - 1)|\phi|\big) dx  \leq \xi |u|_2|\phi|_2 + C|u|^{q-1}_{3(q-1)}\left(\int_{\mathbb{R}^2}(e^{3\alpha|u|^2} - 1)dx\right)^{1/3}|\phi|_3<+\infty.  
		\end{equation*}

		Moreover, let \(r^*, \tau > 1\) such that \(r^* \tau \ell^2 < 1\). Using \eqref{e}, Lemma \ref{2.1}, the H\"older and Sobolev inequalities, and taking
		\[
		\kappa \in \left(0, \frac{1}{r^* \tau} - \ell^2\right) \quad \text{and} \quad \alpha \in \left(4\pi, \frac{4\pi}{r^* \tau (\ell^2 + \kappa)}\right),
		\]
		we have that, for \(n\) large enough, such that \(|\nabla|u_n||_2^2 \leq \ell^2 + \kappa\),
		\begin{align*}
			\int_{\operatorname{supp}(\phi)} \left(|u_n|^{q-1}(e^{\alpha|u_n|^2} - 1)\right)^{r^*} dx &\leq |u_n|_{L^{r^* \tau'(q-1)}(\operatorname{supp}(\phi),\mathbb{C})}^{r^*(q-1)} \left(\int_{\mathbb{R}^2} (e^{\alpha r^* \tau |u_n|^2} - 1) dx\right)^{1/\tau}  \\
			&\leq C \|u_n\|_\varepsilon^{r^*(q-1)} \left(\int_{\mathbb{R}^2} \left(e^{\alpha r^* \tau (\ell^2 + \kappa) \left(\frac{|u_n|}{\sqrt{\ell^2 + \kappa}}\right)^2} - 1\right) dx\right)^{1/\tau} \leq C,
		\end{align*}
		where \(\tau'\) is the conjugate exponent of \(\tau\).
		Then $\{|u_n|^{q-1}(e^{\alpha|u_n|^2} - 1)\}$ is bounded in $L^{r^*}(\operatorname{supp}(\phi),\mathbb{R})$ and so, up to a subsequence, it converges weakly. Moreover, applying, for instance, the {Severini--Egorov theorem}, since$|u_n|^{q-1}(e^{\alpha|u_n|^2} - 1) \to |u|^{q-1}(e^{\alpha|u|^2} - 1)$ a.e. on  $\mathbb{R}^2$, then, for each  $\eta>0$, there exists a measurable set $S_\eta \subset \operatorname{supp}(\phi)$ such that $|\operatorname{supp}(\phi) \setminus S_\eta| < \eta$ and $|u_n|^{q-1}(e^{\alpha|u_n|^2} - 1) \to |u|^{q-1}(e^{\alpha|u|^2} - 1)$ uniformly on $S_\eta$. Thus,
		\begin{align*}
			&\left| \int_{S_\eta} [|u_n|^{q-1}(e^{\alpha|u_n|^2} - 1) - |u|^{q-1}(e^{\alpha|u|^2} - 1)]\phi dx\right|\\
			&\qquad \le
			|\phi|_\infty |\operatorname{supp}(\phi)|   \left||u_n|^{q-1}(e^{\alpha|u_n|^2} - 1) - |u|^{q-1}(e^{\alpha|u|^2} - 1)\right|_{L^\infty(S_\eta)}\to 0.
		\end{align*}
		Moreover, by the H\"older inequality, 
		\begin{align*}
			&\left| \int_{\operatorname{supp}(\phi)\setminus   S_\eta } \left[|u_n|^{q-1}(e^{\alpha|u_n|^2} - 1) - |u|^{q-1}(e^{\alpha|u|^2} - 1)\right]\phi dx \right|\\
			&\qquad
			\le
			\left||u_n|^{q-1}(e^{\alpha|u_n|^2} - 1) - |u|^{q-1}(e^{\alpha|u|^2} - 1)\right|_{L^{r^*}(\operatorname{supp}(\phi)\setminus   S_\eta)}
			|\phi|_{L^{{r^*}'}(\operatorname{supp}(\phi)\setminus   S_\eta)},
		\end{align*}
		where ${r^*}'$ is the conjugate exponent of $r^*$. Then, since, by Fatou's Lemma,
		\[
		\int_{\operatorname{supp}(\phi)} |u|^{r^*(q-1)}(e^{\alpha|u|^2} - 1)^{r^*} dx\le \liminf_n \int_{\operatorname{supp}(\phi)}  |u_n|^{r^*(q-1)}(e^{\alpha|u_n|^2} - 1)^{r^*} dx \leq C,
		\]
		and $
		|\phi|_{L^{{r^*}'}(\operatorname{supp}(\phi)\setminus   S_\eta)} \leq |\phi|_\infty \eta^{1/{{r^*}'}}$ we can conclude that
		\begin{align*}
			&\int_{\mathbb{R}^2} |u_n|^{q-1}(e^{\alpha|u_n|^2} - 1)|\phi| dx
			=\int_{\operatorname{supp}(\phi)} |u_n|^{q-1}(e^{\alpha|u_n|^2} - 1)|\phi| dx\\
			&\qquad
			\to \int_{\operatorname{supp}(\phi)} |u|^{q-1}(e^{\alpha|u|^2} - 1)|\phi| dx
			=\int_{\mathbb{R}^2} |u|^{q-1}(e^{\alpha|u|^2} - 1)|\phi| dx.
		\end{align*}
		Finally, since, up to a subsequence, also \(|u_n| \rightharpoonup |u|\) in \(L^2(\mathbb{R}^2, \mathbb{R})\), we have, as \(n \to +\infty\),
		\[
		\xi \int_{\mathbb{R}^2} |u_n||\phi| dx + C \int_{\mathbb{R}^2} |u_n|^{q-1}(e^{\alpha|u_n|^2} - 1)|\phi| dx \to \xi \int_{\mathbb{R}^2} |u||\phi| dx + C \int_{\mathbb{R}^2} |u|^{q-1}(e^{\alpha|u|^2} - 1)|\phi| dx.
		\]
		Hence, Lemma \ref{GDCT} yields
		\[
		g(\varepsilon x,|u_n|^2)u_n\overline{\phi}
		\to g(\varepsilon x,|u|^2)u\overline{\phi}
		\quad\text{in }L^1(\mathbb{R}^2,\mathbb{C})
		\]
		and so, in particular, for every
		\(\phi \in C_c^\infty(\mathbb{R}^2, \mathbb{C})\),
		\begin{equation}\label{nonlinearity}
			\operatorname{Re}\int_{\mathbb{R}^{2}} g(\varepsilon x, |u_{n}|^{2}) u_{n} \overline{\phi} \,dx
			\to \operatorname{Re}\int_{\mathbb{R}^{2}} g(\varepsilon x, |u|^{2}) u \overline{\phi} \,dx.
		\end{equation}
		Thus, using the density of $C_c^\infty(\mathbb{R}^2,\mathbb{C})$ in $H_\varepsilon$ (Lemma \ref{dense}) and the weak convergence, \(u\) is a critical point of \(J_{\varepsilon}\).
		
		Let \(R > 0\) be such that \(\Lambda_{\Gamma, \varepsilon} \subset B_{R/2}(0)\). We show that for any given \(\zeta > 0\), for \(R\) large enough,
		\begin{equation}\label{compact1}
			\limsup_{n\to +\infty} \int_{B_{R}^{c}(0)} \bigl(|\nabla_{A_{\varepsilon}} u_{n}|^{2} + V_{\varepsilon}(x) |u_{n}|^{2}\bigr) \,dx \le \zeta. 
		\end{equation}
		Let \(\phi_{R} \in C^{\infty}(\mathbb{R}^{2},[0,1] )\) be a cut-off function such that $\phi_{R} = 0$ in $B_{R/2}(0)$, $\phi_{R} = 1$ in $B_{R}^{c}(0)$,  $|\nabla \phi_{R}| \le C/R$, where \(C > 0\) is a constant independent of \(R\). Since
		\begin{equation}\label{nephiR}
			\nabla_{A_{\varepsilon}}(u_{n} \phi_{R}) = \phi_{R} \nabla_{A_{\varepsilon}} u_{n}
			-i u_{n} \nabla \phi_{R}
			\quad
			\text{a.e. in }\mathbb{R}^2,
		\end{equation}
		the sequence \(\{\phi_{R} u_{n}\}\) is bounded in \(H_{\varepsilon}\), and so we have \(J_{\varepsilon}'(u_{n})(\phi_{R} u_{n}) = o_{n}(1)\), that is,
		\[
		\operatorname{Re}\int_{\mathbb{R}^{2}} \nabla_{A_{\varepsilon}} u_{n} \overline{\nabla_{A_{\varepsilon}}(\phi_{R} u_{n})} \,dx
		+ \int_{\mathbb{R}^{2}} V_{\varepsilon}(x) |u_{n}|^{2} \phi_{R} \,dx
		= \int_{\mathbb{R}^{2}} g(\varepsilon x, |u_{n}|^{2}) |u_{n}|^{2} \phi_{R} \,dx + o_{n}(1).
		\]
		Then, using \eqref{nephiR} and \eqref{g5}, we have
		\begin{align*}
			\int_{\mathbb{R}^{2}} \bigl(|\nabla_{A_{\varepsilon}} u_{n}|^{2} + V_{\varepsilon}(x) |u_{n}|^{2}\bigr) \phi_{R} \,dx
			&=
			\operatorname{Re}\int_{\mathbb{R}^{2}}\nabla_{A_{\varepsilon}} u_{n} \overline{\nabla_{A_{\varepsilon}}(u_{n} \phi_{R})}dx - \operatorname{Re}\int_{\mathbb{R}^{2}}i \overline{u}_{n} \nabla_{A_{\varepsilon}} u_{n}\nabla \phi_{R}\,dx\\
			&\qquad
			+\int_{\mathbb{R}^{2}}  V_{\varepsilon}(x) |u_{n}|^{2} \phi_{R} \,dx\\
			&=
			\int_{\mathbb{R}^{2}} g(\varepsilon x, |u_{n}|^{2}) |u_{n}|^{2} \phi_{R} \,dx
			- \operatorname{Re}\int_{\mathbb{R}^{2}} i \overline{u}_{n}\nabla_{A_{\varepsilon}} u_{n} \nabla \phi_{R} \,dx + o_{n}(1) \\
			&\le \frac{1}{k} \int_{\mathbb{R}^{2}} V_{\varepsilon}(x) |u_{n}|^{2} \phi_{R} \,dx
			- \operatorname{Re}\int_{\mathbb{R}^{2}} i \overline{u}_{n} \nabla_{A_{\varepsilon}} u_{n} \nabla \phi_{R} \,dx + o_{n}(1).
		\end{align*}
		Then, the H\"older inequality and the boundedness of \(\{u_{n}\}\) in \(H_{\varepsilon}\) imply
		\[
		\Bigl(1 - \frac{1}{k}\Bigr) \int_{\mathbb{R}^{2}} \bigl(|\nabla_{A_{\varepsilon}} u_{n}|^{2} + V_{\varepsilon}(x) |u_{n}|^{2}\bigr) \phi_{R} \,dx
		\le \frac{C}{R} |u_{n}|_{2} |\nabla_{A_{\varepsilon}} u_{n}|_{2} + o_{n}(1)
		\le \frac{C_1}{R} + o_{n}(1)
		\]
		and so we get \eqref{compact1}.
		
		Let, now,
		\[
		P(x,t) := g(\varepsilon x, t^{2}) t^2, \qquad Q(t) := t(e^{\alpha t^{2}} - 1),\quad t \in \mathbb{R},
		\]
		where
		\[
		\alpha \in \left(4\pi, \frac{4\pi}{\tau(\ell^2 + \kappa)}\right)
		\quad
		\text{with}
		\quad
		\kappa \in (0,1-\ell^2),
		\tau\in\left(1,\min\left\{2,\frac{1}{\ell^2+\kappa}\right\}\right).
		\]
		Using \eqref{g3} and \eqref{f2}, we have
		\begin{equation*}
			\lim_{t \to +\infty} \frac{P(x,t)}{Q(t)} = 0 \quad \text{uniformly for } x \in \mathbb{R}^{2}.
		\end{equation*}
		Moreover, by \eqref{e}, the H\"older and Sobolev inequalities and Lemma~\ref{2.1}, there exists $n_0\in\mathbb N$ such that
		\begin{equation*}
			\sup_{n\geq n_0} \int_{\mathbb{R}^{2}} Q(|u_{n}|)\,dx
			\leq 
			\sup_{n\geq n_0}|u_n|_{\tau'}\left[\int_{\mathbb{R}^2} ( e^{\alpha\tau  (\ell^2 + \kappa) \left( \frac{|u_n|}{\sqrt{\ell^2 + \kappa}} \right)^2} - 1 ) dx\right]^\frac{1}{\tau}
			\le C.  
		\end{equation*}
		where $|\nabla |u_n||_2^2 \leq \ell^2+\kappa$ for $n\geq n_0$ and
		$\tau'>2$ is the conjugate exponent of $\tau$.
		Each of the finitely many omitted integrals (for $n<n_0$) is finite by Lemma \ref{2.1} and the H\"older inequality. Then,
		the
		same bound holds with the supremum taken over all $n\in\mathbb N$.
		
		Since also
		\[
		g(\varepsilon x, |u_{n}|^{2}) |u_{n}|^{2} \to g(\varepsilon x, |u|^{2}) |u|^{2} \quad \text{a.e. in } \mathbb{R}^{2} \text{ as } n \to +\infty,
		\]
		the first part of \cite[Theorem A.I]{Berestycki}
		implies
		\[
		\lim_{n\to +\infty} \int_{B_{R}(0)} \bigl| g(\varepsilon x, |u_{n}|^{2}) |u_{n}|^{2} - g(\varepsilon x, |u|^{2}) |u|^{2} \bigr| \,dx = 0.
		\]
		Moreover, by \eqref{g5} and \eqref{compact1}, for $R$ sufficiently large, we have
		\begin{align*}
			&\limsup_{n\to+\infty}\int_{B_R^c(0)}
			\left|g(\varepsilon x,|u_n|^2)|u_n|^2
			-g(\varepsilon x,|u|^2)|u|^2\right|\,dx\\
			&\quad\leq \frac1k\limsup_{n\to+\infty}
			\int_{B_R^c(0)}
			\left(|\nabla_{A_\varepsilon}u_n|^2+
			V(\varepsilon x)|u_n|^2\right)\,dx
			+\frac1k\int_{B_R^c(0)}
			\left(|\nabla_{A_\varepsilon}u|^2+
			V(\varepsilon x)|u|^2\right)\,dx
			<\frac{2\zeta}{k}
		\end{align*}
		for every \(\zeta > 0\). Hence
		\[
		\int_{\mathbb{R}^{2}} g(\varepsilon x, |u_{n}|^{2}) |u_{n}|^{2} \,dx \to \int_{\mathbb{R}^{2}} g(\varepsilon x, |u|^{2}) |u|^{2} \,dx \quad \text{as } n \to +\infty.
		\]
		Finally, since \(J_{\varepsilon}'(u) = 0\), we have
		\[
		o_{n}(1) = J_{\varepsilon}'(u_{n})[u_{n}] = \|u_{n}\|_{\varepsilon}^{2} - \int_{\mathbb{R}^{2}} g(\varepsilon x, |u_{n}|^{2}) |u_{n}|^{2} \,dx
		= \|u_{n}\|_{\varepsilon}^{2} - \|u\|_{\varepsilon}^{2} + o_{n}(1).
		\]
		Thus, the sequence \(\{u_{n}\}\) converges strongly to \(u\) in \(H_{\varepsilon}\).
	\end{proof}
	
	We next investigate the behavior of sequences $\{u_{n}\}\subset H_{\varepsilon_{n}}$ with 
	\begin{equation*}
		J_{\varepsilon_{n}}(u_{n})\rightarrow c,\quad
		\|J^{\prime}_{\varepsilon_{n}}(u_{n})\|_{\varepsilon_n}:=
		\sup_{\phi_n\in H_{\varepsilon_n},\|\phi_n\|_{\varepsilon_n}\leq 1}|J^{\prime}_{\varepsilon_{n}}(u_{n})[\phi_n]|\to 0,
		\quad
		\varepsilon_{n}\rightarrow 0^+, \text{ as } n\rightarrow +\infty,
	\end{equation*}
	($(PS)_c^*$ for short) for some suitable values of $c$.

	\begin{proposition}\label{fengjie}
		Let $\{u_n\}$ be a $(PS)_c^*$ sequence with $\displaystyle c \in\left(0,\frac{1}{8}\left(\frac{1}{2} - \frac{1}{\Theta} - \frac{1}{2k}\right)\right)$. Then there exist a subsequence, still denoted by $\{u_n\}$, a positive integer $l$, sequences of points $\{y_{n,j}\} \subset \mathbb{R}^2$ with $j = 1, \ldots, l$, such that,  as $n \to +\infty$, $\varepsilon_n y_{n,j} \to y_j \in \overline{\Lambda}_\Gamma$, $|y_{n,i} - y_{n,j}| \to +\infty$ for $i\neq j$, 
		and
		\begin{equation}\label{decomposition}
			\left\| u_n - \sum_{j=1}^l e^{iA(y_j)\cdot(\cdot - y_{n,j})} U_{0,j}(\cdot - y_{n,j}) \chi_{\varepsilon_n}(\cdot - y_{n,j}) \right\|_{\varepsilon_n} \to 0, \quad \text{as } n \to +\infty,
		\end{equation}
		where $\chi_\varepsilon := \chi(\cdot/(-\log \varepsilon))$ for $0 < \varepsilon < 1$, $\chi$ is a cut-off function in $C_c^\infty(\mathbb{R}^2,[0,1])$ satisfying $\chi(x) = 1$ for $|x| \leq 1$, $\chi(x) = 0$ for $|x| \geq 2$, and $|\nabla \chi| \leq 2$, and each $U_{0,j}$ is a nontrivial solution of
		\begin{equation}\label{Ueqn}
			-\Delta u + V(y_j)u = g_{0,j}(x, |u|^2)u, \quad u\in H^1(\mathbb{R}^2, \mathbb{C}),
		\end{equation}
		with
		\[
		g_{0,j}(x,t):=\eta_j(x)f(t)+(1-\eta_j(x))\widehat f(t),
		\quad
		G_{0,j}(x,t):=\int_0^t g_{0,j}(x,s)\,ds,
		\]
		\[
		\chi_{\Lambda_\Gamma}(\varepsilon_n(\cdot+y_{n,j}))
		\rightharpoonup\eta_j
		\quad\hbox{in }L^p_{\mathrm{loc}}(\mathbb{R}^2)
		\quad\hbox{for every }1\leq p<+\infty,
		\qquad 0\leq\eta_j\leq1\quad\hbox{a.e. in }\mathbb{R}^2.
		\]
		Moreover, if
		\[
		J_{0,j}(u) := \frac{1}{2} \int_{\mathbb{R}^2} \left( |\nabla u|^2 + V(y_j)|u|^2 \right) dx - \frac{1}{2} \int_{\mathbb{R}^2} G_{0,j}(x, |u|^2) dx,
		\]
		we have
		\begin{equation}\label{djl616:profile-energy-lower}
			J_{0,j}(U_{0,j})\ge c^c_{V(y_j)}>0
		\end{equation}
		and the energy is completely recovered by the profiles, namely,
		\[
		c=\sum_{j=1}^lJ_{0,j}(U_{0,j}).
		\]
	\end{proposition}

	\begin{proof}
		Since $\{u_{n}\}$ is a $(PS)_c^{*}$ sequence, arguing as in \eqref{bddPS}, we obtain that $\|u_{n}\|_{\varepsilon_n}$ is bounded and
		\[\left(\frac{1}{2}-\frac{1}{\Theta}-\frac{1}{2k}\right)\|u_{n}\|_{\varepsilon_{n}}^{2}\leq c+o_n(1).\]
		Hence, by the diamagnetic inequality \eqref{p2.1} and \eqref{V1}, we obtain that
		\[
		\left(\frac{1}{2}-\frac{1}{\Theta}-\frac{1}{2k}\right)
		\mathop{\limsup}\limits_{n\rightarrow +\infty}(|\nabla|u_{n}||_{2}^{2}+V_0|u_n|_2^2)
		\leq\left(\frac{1}{2}-\frac{1}{\Theta}-\frac{1}{2k}\right)\mathop{\limsup}\limits_{n\rightarrow +\infty}\|u_{n}\|_{\varepsilon_{n}}^{2}
		\leq c< \frac{1}{8}\left(\frac{1}{2}-\frac{1}{\Theta}-\frac{1}{2k}\right),
		\]
		from which,  
		\begin{equation}\label{p3.5}
			\iota^{2}:=\mathop{\limsup}\limits_{n\rightarrow +\infty}(|\nabla|u_{n}||_{2}^{2}+V_0|u_n|_2^2)< \frac{1}{8}.
		\end{equation}

		We claim that there exist a sequence $\{y_{n,1}\} \subset \mathbb{R}^2$ and constants $R_1, \beta_1 > 0$ such that
		\begin{equation}
			\label{p3.6}
			\liminf_{n\rightarrow +\infty}\int_{B_{R_1}(y_{n,1})} |u_n|^2 dx > \beta_1.
		\end{equation}
		Assume by contradiction that for every $R > 0$,
		\[
		\lim_{n \to +\infty} \sup_{y \in \mathbb{R}^2} \int_{B_R(y)} |u_n|^2 dx = 0.
		\]
		Being $\{|u_n|\}$ bounded in $H^1(\mathbb{R}^2, \mathbb{R})$, from \cite[Chapter 6, Lemma 8.4]{Kavian} it follows that $|u_n|_\tau \to 0$ as $n \to +\infty$ for any $\tau > 2$. Fixed $q > 2$, for any $\xi > 0$ and $\alpha > 4\pi$, from \eqref{g3}, \eqref{p2.2} and $J'_{\varepsilon_n}(u_n)[u_n] \to 0$ as $n \to +\infty$, there exists $C > 0$ such that
		\begin{equation}
			\label{p3.7}
			\begin{split}
				\|u_n\|_{\varepsilon_n}^2 &= \int_{\mathbb{R}^2} g(\varepsilon_n x, |u_n|^2)|u_n|^2 dx + o_n(1) 
				\leq \xi \int_{\mathbb{R}^2} |u_n|^2 dx + C \int_{\mathbb{R}^2} |u_n|^q (e^{\alpha|u_n|^2} - 1) dx + o_n(1) \\
				&\leq \frac{\xi}{V_0} \|u_n\|_{\varepsilon_n}^2 + C \int_{\mathbb{R}^2} |u_n|^q (e^{\alpha|u_n|^2} - 1) dx + o_n(1).
			\end{split}
		\end{equation}
		Fix $\kappa \in (0, 1 - \iota^{2})$ and take $r \in \left(1, \frac{1}{\iota^{2} + \kappa}\right)$, $\alpha \in \left(4\pi, \frac{4\pi}{r(\iota^{2} + \kappa)}\right)$. Then, the H\"older inequality, \eqref{e}, and Lemma \ref{2.1} imply that, for $n$ large enough (such that $|\nabla |u_n||_2^2 \leq \iota^{2} + \kappa$),
		\[
		\int_{\mathbb{R}^2} |u_n|^q (e^{\alpha|u_n|^2} - 1) dx \leq |u_n|_{qr'}^q \left( \int_{\mathbb{R}^2} (e^{\alpha r(\iota^{2} + \kappa)\left( \frac{|u_n|}{\sqrt{\iota^{2} + \kappa}} \right)^2}  - 1) dx \right)^{\frac{1}{r}} \leq C|u_n|_{qr'}^q \to 0 \text{ as } n \to +\infty.
		\]
		Therefore, it follows from \eqref{p3.7} that $\|u_n\|_{\varepsilon_n} \to 0$ as $n \to +\infty$.
		Then, by (\ref{g4}), (\ref{g5}), and \eqref{p3.7}, we obtain
		\[
		0\leq \int_{\mathbb{R}^2} G(\varepsilon_n x, |u_n|^2)\, dx
		\leq \int_{\mathbb{R}^2} g(\varepsilon_n x, |u_n|^2)|u_n|^2\, dx
		= o_n(1)
		\]
		and so $c
		=0$, contradicting the assumption $c > 0$.

		Now, let $v_{n}:=u_{n}(\cdot+y_{n,1})$.
		Due to the Sobolev inequality, $\{v_n\}$ is bounded in $L^r(\mathbb{R}^2, \mathbb{C})$ for $2\leq r<+\infty$.
		Moreover, due to the diamagnetic inequality \eqref{p2.1}, $\{|v_{n}|\}$ is bounded in $H^{1}(\mathbb{R}^{2},\mathbb{R})$ and so, there exists $\omega_{1}\in H^1(\mathbb{R}^2,\mathbb{R})$ such that, up to a subsequence, if necessary,  $|v_{n}|\to \omega_{1}$ weakly in $H^{1}(\mathbb{R}^{2},\mathbb{R})$ and strongly in $L^{r}_{\rm loc}(\mathbb{R}^2,\mathbb{R})$ ($1\leq r<+\infty$), with $\omega_{1}\neq0$ by \eqref{p3.6}.
		
		Next, we claim that, up to a subsequence, $\varepsilon_{n}y_{n,1} \to y_1 \in \overline{\Lambda}_\Gamma$ as $n \to +\infty$. 
		Arguing by contradiction, suppose that this is not the case.
		Then, after passing to a further subsequence, there exists $\delta>0$
		such that
		$\operatorname{dist}(\varepsilon_ny_{n,1},\overline\Lambda_\Gamma)\geq\delta$
		and so, for $n$ large enough (such that $2\varepsilon_n|\log\varepsilon_n|<\delta$),
		we obtain
		\begin{equation}\label{B2logen}
			\varepsilon_{n}x+\varepsilon_{n}y_{n,1}\in \Lambda_\Gamma^c \quad \text{for all } x\in B_{2|\log \varepsilon_{n}|}(0).
		\end{equation}
		
		Since
		\[
		\nabla_{A_{\varepsilon_n}} (u_{n}\chi_{\varepsilon_n}(\cdot-y_{n,1}))
		=\chi_{\varepsilon_n}(\cdot-y_{n,1})\nabla_{A_{\varepsilon_n}} u_{n}- iu_{n}\nabla \chi_{\varepsilon_n}(\cdot-y_{n,1}),
		\]
		for $n$ large enough, being $\{\|u_n\|_{\varepsilon_n}\}$ bounded, we have
		\begin{equation}\label{before}
			\begin{split}
				\|u_{n}\chi_{\varepsilon_n}(\cdot-y_{n,1})\|_{\varepsilon_n}^2
				&=|\chi_{\varepsilon_n}(\cdot-y_{n,1})\nabla_{A_{\varepsilon_n}} u_{n}|_2^2-2\operatorname{Im}\int_{\mathbb{R}^2}\chi_{\varepsilon_n}(\cdot-y_{n,1})\nabla\chi_{\varepsilon_n}(\cdot-y_{n,1})\overline{u}_{n}\nabla_{A_{\varepsilon_n}} u_{n}dx  \\
				&\quad   +|u_{n}\nabla \chi_{\varepsilon_n}(\cdot-y_{n,1})|_2^2               +\int_{\mathbb{R}^2}V_{\varepsilon_n}(x)|\chi_{\varepsilon_n}(\cdot-y_{n,1})u_{n}|^2dx\\
				&\leq\|u_n\|_{\varepsilon_n}^2
				+\frac{4}{\log^2 \varepsilon_n}|u_{n}|_2^2
				+\frac{4}{|\log\varepsilon_n|}|u_{n}|_2 |\nabla_{A_{\varepsilon_n}} u_{n}|_2
				\leq C.
			\end{split}
		\end{equation}
		
		Then, it follows that 
		\begin{align*}
			o_n(1)
			&=J_{\varepsilon_n}^{\prime}(u_{n})[u_{n}\chi_{\varepsilon_{n}}(x-y_{n,1})]\\
			&=\operatorname{Re}\int_{\mathbb{R}^2} \nabla_{A_{\varepsilon_{n}}} u_{n} \overline{\nabla_{A_{\varepsilon_n}} (u_{n}\chi_{\varepsilon_{n}}(x-y_{n,1}))}dx\\
			&\quad
			+ \int_{\mathbb{R}^2} V_{\varepsilon_n}(x)|u_{n}|^{2}\chi_{\varepsilon_{n}}(x-y_{n,1})dx
			-\int_{\mathbb{R}^2}g(\varepsilon_{n}x, |u_{n}|^{2})|u_{n}|^{2}\chi_{\varepsilon_{n}}(x-y_{n,1})dx\\
			&=
			\int_{\mathbb{R}^2} (|\nabla_{A_{\varepsilon_{n}}}u_{n}|^{2} +  V_{\varepsilon_n}(x)|u_{n}|^{2})\chi_{\varepsilon_{n}}(x-y_{n,1})dx\\
			&\quad
			-\operatorname{Im}\int_{\mathbb{R}^2} \nabla \chi_{\varepsilon_{n}}(x-y_{n,1})\nabla_{A_{\varepsilon_{n}}}u_{n} \overline{u}_n dx
			-\int_{\mathbb{R}^2}g(\varepsilon_{n}x, |u_{n}|^{2})|u_{n}|^{2}\chi_{\varepsilon_{n}}(x-y_{n,1})dx.     	    
		\end{align*}
		Changing variable and using 
		\eqref{V1}, we get
		\begin{equation*}
			\begin{split}	
				\int_{\mathbb{R}^2} (|\nabla_{A_{\varepsilon_{n}}}u_{n}|^{2} +  V_{\varepsilon_n}(x)|u_{n}|^{2})\chi_{\varepsilon_{n}}(x-y_{n,1})dx
				&=
				\int_{\mathbb{R}^2} \left(|\nabla_{A_{\varepsilon_n}(x+y_{n,1})}v_n|^{2} +  V_{\varepsilon_n}( x+y_{n,1})|v_n|^{2}\right)\chi_{\varepsilon_{n}}(x)dx\\
				&\geq V_0 \int_{\mathbb{R}^2}   |v_n|^{2}\chi_{\varepsilon_{n}}(x)dx.
			\end{split} 
		\end{equation*}
		
		Moreover, by the H\"older inequality,
		\begin{equation*}
			\left|\operatorname{Im}\int_{\mathbb{R}^2} \nabla \chi_{\varepsilon_{n}}(x-y_{n,1})\nabla_{A_{\varepsilon_{n}}}u_{n} \overline{u}_n dx\right|
			\leq 
			\frac{2}{|\log \varepsilon_n|}|u_n|_2|\nabla_{A_{\varepsilon_{n}}}u_{n}|_2=o_n(1).
		\end{equation*}
		Finally, changing variable, for $n$ sufficiently large, from \eqref{B2logen} and \eqref{g5}, it follows that
		\begin{align*}
			\int_{\mathbb{R}^2}g(\varepsilon_{n}x, |u_{n}|^{2})|u_{n}|^{2}\chi_{\varepsilon_{n}}(x-y_{n,1})dx
			&=\int_{\mathbb{R}^2}g(\varepsilon_{n}(x+ y_{n,1}), |v_{n}|^{2})|v_{n}|^{2}\chi_{\varepsilon_{n}}(x)dx\\
			&=\int_{B_{2|\log\varepsilon_n|}(0)}g(\varepsilon_{n}(x+ y_{n,1}), |v_{n}|^{2})|v_{n}|^{2}\chi_{\varepsilon_{n}}(x)dx\\
			&\leq
			\frac{V_0}{k}\int_{\mathbb{R}^2}|v_n|^2\chi_{\varepsilon_{n}}(x)dx.
		\end{align*}
		Then
		\[
		\int_{\mathbb{R}^2}|v_n|^2\chi_{\varepsilon_{n}}(x)dx=o_n(1).
		\]
		Therefore,
		if $R>0$ is sufficiently large and $n$ such that $|\log\varepsilon_n|>R$,
		\begin{align*}
			0< \int_{B_R(0)} |\omega_1|^2dx
			&= \int_{B_R(0)} |v_n|^2dx +o_n(1)
			\leq \int_{B_{|\log\varepsilon_n|}(0)} |v_n|^2dx +o_n(1)
			\\
			&\leq \int_{B_{2|\log\varepsilon_n|}(0)} |v_n|^2\chi_{\varepsilon_n}(x) dx +o_n(1)
			=
			\int_{\mathbb{R}^2} |v_n|^2\chi_{\varepsilon_n}(x) dx +o_n(1)= o_n(1),
		\end{align*}
		which is 
		a contradiction.  Hence $y_1\in\overline{\Lambda}_\Gamma$.

		Observing that, if $R>0$, for $n$ large enough, $\varepsilon_n B_{2|\log\varepsilon_n|}(0)\subset B_R(0)$, and combining this with  $\varepsilon_n y_{n,1}\to y_1$, we obtain that there exists $R'>0$ such that, for $n$ large enough,  $\varepsilon_n (x+ y_{n,1}) \in B_{R'}(0)$ for every $x\in B_{2|\log\varepsilon_n|}(0)$, so that there exists $C>0$ independent of $n$ such that
		\begin{equation}\label{unifboundVA}
			\sup_{x\in B_{2|\log\varepsilon_n|}(0)} V_{\varepsilon_n}(x+y_{n,1})\leq C
			\quad\text{and}\quad
			\sup_{x\in B_{2|\log\varepsilon_n|}(0)} |A_{\varepsilon_n}(x+y_{n,1})|\leq C.
		\end{equation}
		Then, as in Lemma \ref{coro} (\ref{equivi}), we get that
		for $n$ sufficiently large, the norm
		\begin{equation}\label{fanshu}
			\left(\int_{B_{2|\log\varepsilon_n|}(0)} (|\nabla_{A_{\varepsilon_n}(x+y_{n,1})} u|^{2}+V_{\varepsilon_n}( x+y_{n,1})|u|^{2})dx\right)^{1/2}
		\end{equation}
		is uniformly equivalent to the standard norm of $H^{1}(B_{2|\log\varepsilon_n|}(0),\mathbb{C})$.
		
		Analogously, if $K\subset\mathbb{R}^2$ is a bounded open set, we obtain that the norm
		\begin{equation}\label{fanshu1}
			\left(\int_{K} (|\nabla_{A_{\varepsilon_n}(x+y_{n,1})} u|^{2}+V_{\varepsilon_n}( x+y_{n,1})|u|^{2})dx\right)^{1/2}
		\end{equation}
		is uniformly equivalent to the standard norm of $H^{1}(K,\mathbb{C})$.
		
		Now, for any $n\in \mathbb{N}$, let $w_n:=v_n\chi_{\varepsilon_n}$. Due to the equivalence of the norm in \eqref{fanshu} and \eqref{before},
		\begin{equation}\label{wnbdd}
			\begin{split}
				\|w_n\|^2
				&=
				\int_{B_{2|\log\varepsilon_n|}(0)} (|\nabla w_n|^{2}+|w_n|^{2})dx\\
				&\leq
				C\int_{B_{2|\log\varepsilon_n|}(0)} (|\nabla_{A_{\varepsilon_n}(x+y_{n,1})} w_n|^{2}+V_{\varepsilon_n}(x+y_{n,1})|w_n|^{2})dx\\
				&=
				C\int_{B_{2|\log\varepsilon_n|}(y_{n,1})} (|\nabla_{A_{\varepsilon_n}} (u_{n}\chi_{\varepsilon_n}(\cdot-y_{n,1}))|^{2}+V_{\varepsilon_n}(x)|u_{n}\chi_{\varepsilon_n}(\cdot-y_{n,1})|^{2})dx\\
				&\leq C\|u_{n}\chi_{\varepsilon_n}(\cdot-y_{n,1})\|_{\varepsilon_n}^2\leq C.
			\end{split}
		\end{equation}
		Then there exists $u_{0,1}\in H^1(\mathbb{R}^2,\mathbb{C})$ such that, up to a subsequence, $w_n\rightharpoonup u_{0,1}$ in $H^1(\mathbb{R}^2,\mathbb{C})$ and, since $\chi_{\varepsilon_n}\rightarrow 1$ a.e. in $\mathbb{R}^2$ as $n\to +\infty$, $|u_{0,1}|=\omega_1\neq 0$.
		Indeed, on every compact set $K\subset\mathbb R^2$ we have
		$\chi_{\varepsilon_n}=1$ for all sufficiently large $n$, and hence
		$w_n=v_n$ on $K$. The local compactness therefore gives both
		$|w_n|\to|u_{0,1}|$ and $|v_n|\to\omega_1$ in $L^2(K)$, which identifies
		$|u_{0,1}|=\omega_1$ a.e. on $K$.
		
		Furthermore, up to a subsequence,
		\begin{equation}\label{localweakprofile}
			\chi_{\Lambda_\Gamma}(\varepsilon_n(\cdot+y_{n,1}))
			\rightharpoonup\eta_1\quad\hbox{in }L^p_{\mathrm{loc}}(\mathbb R^2, \mathbb{R}) \text{ for every } 1\leq p<+\infty,
			\quad
			0\leq \eta_1\leq 1 \text{ a.e. in } \mathbb{R}^2,
		\end{equation}
		and, as a consequence,
		if \(h_n\to h\) in
		\(L^1(\mathbb{R}^2,\mathbb{C})\), then
		\begin{equation}\label{localweakproduct}
			\int_{\mathbb{R}^2}\chi_{\Lambda_\Gamma}
			(\varepsilon_nx+\varepsilon_ny_{n,1})h_n\,dx
			\longrightarrow\int_{\mathbb{R}^2}\eta_1h\,dx.
		\end{equation}
		Indeed, first observe that
		\[
		\begin{aligned}
			\bigg|\int_{\mathbb{R}^2}\chi_{\Lambda_\Gamma}
			(\varepsilon_nx+\varepsilon_ny_{n,1})h_n\,dx
			-\int_{\mathbb{R}^2}\eta_1h\,dx\bigg|
			&\leq |h_n-h|_{1}
			+\bigg|\int_{\mathbb{R}^2}
			\bigl(\chi_{\Lambda_\Gamma}(\varepsilon_nx+\varepsilon_ny_{n,1})
			-\eta_1(x)\bigr)h(x)\,dx\bigg|.
		\end{aligned}
		\]
		For \(M>0\), set
		\[
		\psi_M(x):=
		\begin{cases}
			h(x),& |h(x)|\leq M,\\
			M\dfrac{h(x)}{|h(x)|},& |h(x)|>M.
		\end{cases}
		\]
		Observe that \(\psi_M\in L^\infty(\mathbb{R}^2,\mathbb{C})\) and,
		fixed \(R>0\) and \(\delta>0\), since \(h\in L^1(\mathbb{R}^2,\mathbb{C})\), we have
		\[
		\int_{B_R(0)}|h-\psi_M| \,dx
		{\leq}\int_{\{|h|>M\}}|h|\,dx\longrightarrow0
		\quad\hbox{as }M\to+\infty.
		\]
		Choose \(M\) so large that
		\(|h-\psi_M|_{L^1(B_R(0))}<\delta\).
		Since
		\(\psi_M\in L^2(B_R(0),\mathbb{C})\), by \eqref{localweakprofile} applied to the real and imaginary parts of
		\(\psi_M\), we obtain
		\[
		\int_{B_R(0)}
		\bigl(\chi_{\Lambda_\Gamma}(\varepsilon_nx+\varepsilon_ny_{n,1})
		-\eta_1(x)\bigr)\psi_M(x)\,dx\longrightarrow0.
		\]
		Moreover, 
		\begin{align*}
			&\bigg|\int_{\mathbb{R}^2}
			\bigl(\chi_{\Lambda_\Gamma}(\varepsilon_nx+\varepsilon_ny_{n,1})
			-\eta_1(x)\bigr)h(x)\,dx\bigg|\\
			&\quad\leq
			\bigg|\int_{B_R(0)}
			\bigl(\chi_{\Lambda_\Gamma}(\varepsilon_nx+\varepsilon_ny_{n,1})
			-\eta_1(x)\bigr)\psi_M(x)\,dx\bigg|
			+2\int_{B_R(0)}|h-\psi_M| \,dx
			+2\int_{B_R^c(0)}|h| \,dx
		\end{align*}
		and so
		\[
		\limsup_{n\to+\infty}\bigg|\int_{\mathbb{R}^2}
		\bigl(\chi_{\Lambda_\Gamma}(\varepsilon_nx+\varepsilon_ny_{n,1})
		-\eta_1(x)\bigr)h(x)\,dx\bigg|
		\leq2\delta+2|h|_{L^1(B_R(0)^c)}
		\]
		which gives the desired conclusion by first letting \(\delta\to0^+\) and then \(R\to+\infty\).
		
		Now, for any $\varphi\in C_{c}^{\infty}(\mathbb{R}^{2},\mathbb{C})$, due to the equivalence of the norm in \eqref{fanshu1},
		we obtain
		\begin{equation*}
			\|\varphi(\cdot-y_{n,1})\|_{\varepsilon_{n}}^{2}
			=\int_{\operatorname{supp} \varphi} (|\nabla_{A_{\varepsilon_n}(x+y_{n,1})} \varphi|^{2}+ V_{\varepsilon_n}( x+y_{n,1}) |\varphi|^{2})dx\leq C\|\varphi\|^{2}< +\infty.
		\end{equation*}
		Thus, for $n$ large enough (such that $\operatorname{supp}  \varphi\subset B_{|\log\varepsilon_n|}(0)$),	
		\begin{align*}
			o_n(1)&=	J^{\prime}_{\varepsilon_{n}}(u_{n})[\varphi(\cdot-y_{n,1})]\\
			&= \operatorname{Re}\int_{\operatorname{supp}  \varphi}\left(\nabla_{A_{\varepsilon_n}(x+y_{n,1})}v_n\overline{\nabla_{A_{\varepsilon_n}(x+y_{n,1})}\varphi}+V_{\varepsilon_n}( x+y_{n,1})v_n\overline{\varphi}
			- g(\varepsilon_n( x+y_{n,1}),|v_n|^{2})v_n\overline{\varphi}\right)dx\\
			&= \operatorname{Re}\int_{\operatorname{supp}  \varphi}\left(\nabla_{A_{\varepsilon_n}(x+y_{n,1})}w_n\overline{\nabla_{A_{\varepsilon_n}(x+y_{n,1})}\varphi}+V_{\varepsilon_n}( x+y_{n,1})w_n\overline{\varphi}
			-g(\varepsilon_n( x+y_{n,1}),|w_n|^{2})w_n\overline{\varphi}\right)dx.
		\end{align*}
		Since \(w_n\to u_{0,1}\) in
		\(L^2(\operatorname{supp}\varphi,\mathbb C)\), by
		\cite[Lemma A.1]{willem1997}, after passing to a further subsequence,
		still denoted by \(\{w_n\}\), there exists
		\(v\in L^2(\operatorname{supp}\varphi,\mathbb R)\) such that
		\(|w_n|\leq v\) a.e. in \(\operatorname{supp}\varphi\).
		Using the H\"older inequality and  Lebesgue's Dominated Convergence Theorem, we obtain
		\begin{align*}
			& \left| \operatorname{Re}\int_{\operatorname{supp}  \varphi}\left(\nabla_{A_{\varepsilon_n}(x+y_{n,1})}w_n\overline{\nabla_{A_{\varepsilon_n}(x+y_{n,1})}\varphi}-\nabla_{A(y_1)}w_n\overline{\nabla_{A(y_1)}\varphi}\right)dx\right|\\ 
			&\leq \int_{\operatorname{supp}\varphi}
			|(\nabla_{A_{\varepsilon_n}(x+y_{n,1})}-\nabla_{A(y_1)})w_n|
			|\nabla_{A_{\varepsilon_n}(x+y_{n,1})}\varphi|\,dx\\
			&\quad+\int_{\operatorname{supp}\varphi}
			|\nabla_{A(y_1)}w_n|
			|(\nabla_{A_{\varepsilon_n}(x+y_{n,1})}-\nabla_{A(y_1)})\varphi|\,dx\\
			&\leq
			|(A_{\varepsilon_n}(x+y_{n,1})-A(y_1))w_n|_{L^2(\operatorname{supp}\varphi)}
			|\nabla_{A_{\varepsilon_n}(x+y_{n,1})}\varphi|_2+|\nabla_{A(y_1)}w_n|_2
			|(A_{\varepsilon_n}(x+y_{n,1})-A(y_1))\varphi|_2\\
			&\leq C|(A_{\varepsilon_n}(x+y_{n,1})-A(y_1))w_n|_{L^2(\operatorname{supp}\varphi)}+C|(A_{\varepsilon_n}(x+y_{n,1})-A(y_1))\varphi|_2=o_n(1),
		\end{align*}
		where we have also used that \[|A_{\varepsilon_n}(x+y_{n,1})-A(y_1)|^2|w_n|^2\leq Cv^2\in L^{1}(\operatorname{supp}  \varphi,\mathbb{R})\]
		and $|A_{\varepsilon_n}(x+y_{n,1})-A(y_1)|\to 0$ a.e. in $\operatorname{supp}  \varphi$.
		
		Similarly, we can obtain that
		\[
		\left| \operatorname{Re}\int_{\operatorname{supp}  \varphi}\left(V_{\varepsilon_n}(x+y_{n,1})-V(y_1)\right)w_n \overline{\varphi}dx\right|=o_n(1).
		\]

		Finally, since by the local compact
		embedding, after passing to a subsequence,
		\(w_n\to u_{0,1}\) in \(L^r(\operatorname{supp}\varphi,\mathbb{C})\) for every
		\(1\leq r<+\infty\), and a.e. in \(\operatorname{supp}\varphi\), using the  estimate \eqref{p3.5}, the inequality
		\(0\leq\widehat f\leq f\), and arguing as in the proof of \eqref{nonlinearity}, 
		\[
		\widehat f(|w_n|^2)w_n\overline\varphi 
		\to  \widehat f(|u_{0,1}|^2)u_{0,1}\overline\varphi,\quad
		\bigl(f(|w_n|^2)-\widehat f(|w_n|^2)\bigr)w_n\overline\varphi
		\to\bigl(f(|u_{0,1}|^2)-\widehat f(|u_{0,1}|^2)\bigr)u_{0,1}\overline\varphi\,\ \text{in}\,\ L^1(\mathbb{R}^2,\mathbb{C}) ,
		\]
		and so, applying \eqref{localweakproduct}, we obtain
		\begin{equation*}
			\operatorname{Re}\int_{\operatorname{supp}\varphi}
			g(\varepsilon_n(x+y_{n,1}),|v_n|^{2})v_n\overline{\varphi}\,dx
			\to\operatorname{Re}\int_{\mathbb{R}^2}
			g_{0,1}(x,|u_{0,1}|^{2})u_{0,1}\overline{\varphi}\,dx.
		\end{equation*}
		
		Hence,  using that $w_n\rightharpoonup u_{0,1}$ in $H^1(\mathbb{R}^{2},\mathbb{C})$, we conclude that
		\begin{align*}
			o_n(1)=
			\operatorname{Re}\int_{\mathbb{R}^2}\left(\nabla_{A(y_1)}u_{0,1}\overline{\nabla_{A(y_1)}\varphi}+V(y_1)u_{0,1}\overline{\varphi}
			-g_{0,1}(x,|u_{0,1}|^{2})u_{0,1}\overline{\varphi}\right)dx,
		\end{align*}
		which implies that 
		\begin{equation}\label{p3.9}
			\left(\frac{1}{i}\nabla-A(y_1)\right)^2 u_{0,1}+ V(y_1)u_{0,1}
			=g_{0,1}(x,|u_{0,1}|^{2})u_{0,1}
			\quad \text{in }\mathbb{R}^2.
		\end{equation}
		Then $U_{0,1}(x):=e^{-iA(y_1)\cdot x}u_{0,1}\in H^{1}(\mathbb{R}^2,\mathbb{C})\setminus\{0\}$ and, since
		\[
		-\Delta U_{0,1}=e^{-iA(y_1)\cdot x}\left(\frac{1}{i}\nabla-A(y_1)\right)^2 u_{0,1},
		\]
		it solves \eqref{Ueqn} for $j=1$.
		
		Let $\Phi(t):=J_{0,1}(tU_{0,1})$ for $t\geq0$. 
		Since \(U_{0,1}\) solves \eqref{Ueqn} for \(j=1\), 
		\begin{align*}
			\Phi'(t)
			&=t\left[
			\int_{\mathbb{R}^{2}}\left(|\nabla U_{0,1}|^{2}
			+V(y_1)|U_{0,1}|^{2}\right)\,dx
			-\int_{\mathbb{R}^{2}}g_{0,1}(x,t^{2}|U_{0,1}|^{2})
			|U_{0,1}|^{2}\,dx\right]\\
			&= t\int_{\mathbb{R}^{2}}\left[g_{0,1}(x,|U_{0,1}|^{2})-g_{0,1}(x,t^{2}|U_{0,1}|^{2})\right]
			|U_{0,1}|^{2}\,dx.
		\end{align*}
		Moreover \(t\mapsto g_{0,1}(x,t)\) is non-decreasing for every
		\(x\in\mathbb{R}^{2}\), so that \(t=1\) is a
		global maximizer of \(J_{0,1}(tU_{0,1})\). Hence, since $0\leq\hat{F}\leq F$,
		\begin{equation}\label{J01U01}
			J_{0,1}(U_{0,1})=\max_{t\geq0}J_{0,1}(tU_{0,1})
			\geq\max_{t\geq0}I^c_{V(y_1)}(tU_{0,1})\geq c^c_{V(y_1)}>0.
		\end{equation}

		Now, for every $n\in \mathbb{N}$, let us define
		\begin{equation*}
			\Psi_{n,1} := e^{iA(y_1) \cdot (\cdot - y_{n,1})}(U_{0,1} \chi_{\varepsilon_n})(\cdot - y_{n,1})=(u_{0,1} \chi_{\varepsilon_n})(\cdot - y_{n,1})
		\end{equation*}
		that, by \eqref{unifboundVA} and the shifted norm equivalence
		\eqref{fanshu}, 
		belongs to $H_{\varepsilon_n}$.
		Since
		$
		V(\varepsilon_n \cdot +\varepsilon_n y_{n,1})\,|U_{0,1}|^2\chi_{\varepsilon_n}^2\to V(y_1)|U_{0,1}|^2$ a.e. in $\mathbb{R}^2$
		and, by \eqref{unifboundVA},
		\[
		|V(\varepsilon_n \cdot+\varepsilon_n y_{n,1})|\,|U_{0,1}|^2\,\chi_{\varepsilon_n}^2\le C|U_{0,1}|^2\in L^1(\mathbb{R}^2,\mathbb{R}),
		\]
		applying the Lebesgue Dominated Convergence Theorem, we get
		\begin{equation}\label{mo0}
			\int_{\mathbb{R}^2} V_{\varepsilon_n}( x) |\Psi_{n,1}|^2 dx=\int_{\mathbb{R}^2} V_{\varepsilon_n}( x+y_{n,1}) |U_{0,1} \chi_{\varepsilon_n}|^2 dx \to V(y_1)  \int_{\mathbb{R}^2}  |U_{0,1} |^2 dx. 
		\end{equation}   
		In addition, 
		\begin{align*}
			\int_{\mathbb{R}^2}|\nabla_{A_{\varepsilon_n}} \Psi_{n,1}|^2dx
			&= \int_{\mathbb{R}^2} |\chi_{\varepsilon_n} \nabla U_{0,1}|^2 dx +  \int_{\mathbb{R}^2}\left( |U_{0,1} \nabla \chi_{\varepsilon_n}|^2  + 2 \operatorname{Re}  \chi_{\varepsilon_n} \nabla U_{0,1} \overline{U_{0,1}}\cdot \nabla \chi_{\varepsilon_n}\right) dx \\
			&\quad +\int_{\mathbb{R}^2}  |A(y_1) - A_{\varepsilon_n}(x+y_{n,1})|^2 |U_{0,1} \chi_{\varepsilon_n}|^2  dx\\
			&\quad - 2 \operatorname{Im} \int_{\mathbb{R}^2}  (A(y_1) - A_{\varepsilon_n}(x+y_{n,1})) U_{0,1} \chi_{\varepsilon_n} \overline{\nabla (U_{0,1} \chi_{\varepsilon_n})} dx.
		\end{align*}
		Arguing as for \eqref{mo0}, we have
		\[
		\lim_{n\to +\infty}\int_{\mathbb{R}^2}|A(y_1) - A_{\varepsilon_n}(x+y_{n,1})|^2 |U_{0,1} \chi_{\varepsilon_n}|^2dx=0.
		\]
		Moreover,
		\begin{equation*}
			\left|\int_{\mathbb{R}^2} \left(|U_{0,1} \nabla \chi_{\varepsilon_n}|^2  + 2 \operatorname{Re}  \chi_{\varepsilon_n} \nabla U_{0,1} \overline{U_{0,1}}\cdot \nabla \chi_{\varepsilon_n}\right) dx\right|\leq \frac{4}{\log^2 \varepsilon_n}|U_{0,1} |^2_2 + \frac{4}{|\log \varepsilon_n|}  |\nabla U_{0,1}|_2 |U_{0,1}|_2=o_n(1)  
		\end{equation*}
		and
		\begin{align*}
			&\left|  \operatorname{Im} \int_{\mathbb{R}^2}  (A(y_1) - A_{\varepsilon_n}(x+y_{n,1})) U_{0,1} \chi_{\varepsilon_n} \overline{\nabla (U_{0,1} \chi_{\varepsilon_n})} dx \right|\\
			&\quad\leq  \left( \int_{\mathbb{R}^2} |A(y_1) - A_{\varepsilon_n}(x+y_{n,1})|^2 |U_{0,1} \chi_{\varepsilon_n}|^2 dx \right)^{\frac{1}{2}} \left(2|\nabla U_{0,1}|_2^2+\frac{8}{\log^2\varepsilon_n}|U_{0,1}|_2^2\right)^\frac{1}{2}= o_n(1).
		\end{align*}
		Then
		\begin{equation}\label{mo}
			\begin{split}
				\|\Psi_{n,1}\|_{\varepsilon_n}^2 = \int_{\mathbb{R}^2} |\nabla U_{0,1}|^2 dx + V(y_1) \int_{\mathbb{R}^2} |U_{0,1}|^2 dx + o_n(1).
			\end{split}
		\end{equation}
		
		Furthermore,   by (\ref{g3}) and \eqref{p2.5}, fixed $q > 2$, for any $\xi > 0$  and $\alpha > 4\pi$,  there exists $C > 0$ such that
		\[
		|G(\varepsilon_n x+\varepsilon_{n}y_{n,1}, |U_{0,1}\chi_{\varepsilon_n}|^2)| \leq F(|U_{0,1}\chi_{\varepsilon_n}|^2) \leq \xi|U_{0,1}|^2 + C|U_{0,1}|^q(e^{\alpha|U_{0,1}|^2} - 1) 
		\]
		that belongs\footnote{Using again the H\"older inequality and Lemma \ref{2.1}.} to $L^1(\mathbb{R}^2, \mathbb{R})$.
		Then, since \(0\leq\chi_{\varepsilon_n}\leq1\),
		\(\chi_{\varepsilon_n}\to1\) a.e., and \(0\leq\widehat F\leq F\), the
		previous inequality and the Dominated Convergence Theorem give
		\[
		F(|U_{0,1}\chi_{\varepsilon_n}|^2)\to F(|U_{0,1}|^2),\qquad
		\widehat F(|U_{0,1}\chi_{\varepsilon_n}|^2)\to\widehat F(|U_{0,1}|^2)
		\quad\hbox{in }L^1(\mathbb{R}^2).
		\]
		Thus, since \(G(\varepsilon_nx+\varepsilon_ny_{n,1},\cdot)=\widehat F(\cdot)
		+\chi_{\Lambda_\Gamma}(\varepsilon_n(x+y_{n,1}))
		\bigl(F(\cdot)-\widehat F(\cdot)\bigr)\),
		\eqref{localweakproduct} gives
		\[
		\int_{\mathbb{R}^2}G(\varepsilon_n x, |\Psi_{n,1}|^2)\,dx
		=\int_{\mathbb{R}^2}G(\varepsilon_n x+\varepsilon_n y_{n,1},
		|U_{0,1}\chi_{\varepsilon_n}|^2)\,dx
		\to \int_{\mathbb{R}^2}G_{0,1}(x,|U_{0,1}|^2)\,dx.
		\]
		Therefore, 
		\begin{equation}
			\label{strong}
			J_{\varepsilon_n}(\Psi_{n,1}) \to J_{0,1}(U_{0,1}) \text{ as } n \to +\infty.
		\end{equation}

		Now, for every $n\in \mathbb{N}$, let us define \( u_n^1 := u_n - \Psi_{n,1} \).
		We claim that, after extracting a subsequence if necessary,  $\{u_n^1\}$ is a $(PS)_{c-J_{0,1}(U_{0,1})}^*$ sequence.

		First, let us prove that, up to  a subsequence,
		\begin{equation}\label{p3.13}
			J_{\varepsilon_{n}}(u_{n}^{1})=J_{\varepsilon_{n}}(u_{n})-J_{\varepsilon_{n}}(\Psi_{n,1})+o_{n}(1)
		\end{equation}
		which, combined with \eqref{strong}, implies that
		\begin{equation*}
			J_{\varepsilon_{n}}(u_{n}^{1})\rightarrow c-J_{0,1}(U_{0,1})\,\ \text{as}\,\ n\rightarrow +\infty.
		\end{equation*}
		To this end, let
		\begin{align*}
			A_n&:= \|u_n^1\|_{\varepsilon_n}^2-\|u_n\|_{\varepsilon_n}^2+\|\Psi_{n,1}\|_{\varepsilon_n}^2= 2\langle \Psi_{n,1},\Psi_{n,1}-u_n\rangle_{\varepsilon_n}\\
			&= 2 \operatorname{Re} \int_{B_{2|\log \varepsilon_n|}(0)}  \nabla_{A_{\varepsilon_n}(x+y_{n,1})} ( u_{0,1} \chi_{\varepsilon_n}-v_n
			) \overline{\nabla_{A_{\varepsilon_n}(x+y_{n,1})} (u_{0,1} \chi_{\varepsilon_n})}dx\\
			&\qquad
			+2 \operatorname{Re} \int_{B_{2|\log \varepsilon_n|}(0)}V_{\varepsilon_n}( x+y_{n,1}) (u_{0,1} \chi_{\varepsilon_n}-v_n) \overline{u}_{0,1} \chi_{\varepsilon_n} dx
		\end{align*}
		where we recall that $v_{n}=u_{n}(\cdot+y_{n,1})$.
		
		For $R>0$ and $n$ sufficiently large (such that $\overline{B_{R}(0)}\subset B_{|\log \varepsilon_n|}(0)$ so that $w_n=v_n\chi_{\varepsilon_n}=v_n$ on $\overline{B_{R}(0)}$),
		\begin{align*}  
			&\left|\operatorname{Re} \int_{B_{R}(0)} \biggl( \nabla_{A_{\varepsilon_n}(x+y_{n,1})} (u_{0,1} \chi_{\varepsilon_n}-v_n) \overline{\nabla_{A_{\varepsilon_n}(x+y_{n,1})} (u_{0,1}\chi_{\varepsilon_n})}
			-\nabla_{A(y_1)} (u_{0,1} - w_n ) \overline{\nabla_{A(y_1)} u_{0,1}}\biggr) dx\right|\\
			&\quad\leq
			\int_{B_{R}(0)}|A_{\varepsilon_n}(x+y_{n,1})-A(y_1)| |\nabla(u_{0,1} -w_n)| |u_{0,1}|dx
			+\int_{B_{R}(0)} |u_{0,1}-w_n||A(y_1)| |\nabla_{A(y_1)}u_{0,1} |dx\\
			&\qquad
			+ \int_{B_{R}(0)}
			|\nabla_{A_{\varepsilon_n}(x+y_{n,1})} u_{0,1}| |u_{0,1} -w_n| |A_{\varepsilon_n}(x+y_{n,1})|dx.
		\end{align*}
		Moreover, since $\varepsilon_n (x + y_{n,1})\to y_1$ for $x\in B_R(0)$, by the H\"older inequality, \eqref{unifboundVA}, and \eqref{wnbdd}, using the Lebesgue Dominated Convergence Theorem,
		\begin{align*}
			&\int_{B_{R}(0)}|A_{\varepsilon_n}(x+y_{n,1})-A(y_1)| |\nabla(u_{0,1} -w_n)| |u_{0,1}|dx\\
			&\quad\leq\left(\int_{B_{R}(0)}|A_{\varepsilon_n}(x+y_{n,1})-A(y_1)|^2|u_{0,1}|^2dx\right)^{1/2} |\nabla(u_{0,1} -w_n)|_2=o_n(1).
		\end{align*}
		In addition, $w_n\to u_{0,1}$ in $L_{\rm loc}^2(\mathbb{R}^2,\mathbb{C})$ implies 
		\begin{align*}
			\int_{B_{R}(0)} |u_{0,1}-w_n||A(y_1)| |\nabla_{A(y_1)}u_{0,1} |dx
			=o_n(1),
		\end{align*}
		and, together with the equivalence of the norm in \eqref{fanshu1} and \eqref{unifboundVA},
		\[
		\int_{B_{R}(0)}
		|\nabla_{A_{\varepsilon_n}(x+y_{n,1})} u_{0,1}| |u_{0,1} -w_n| |A_{\varepsilon_n}(x+y_{n,1})|dx
		\leq C\left(\int_{B_{R}(0)}
		|u_{0,1} -w_n|^2 dx\right)^\frac{1}{2}
		\|u_{0,1}\|
		=o_n(1)
		\]
		and
		\[\operatorname{Re}\int_{B_{R}(0)} V_{\varepsilon_n}(x+y_{n,1}) (u_{0,1} \chi_{\varepsilon_n}-v_n) \overline{u}_{0,1}\chi_{\varepsilon_n}  dx-\operatorname{Re}\int_{B_{R}(0)} V(y_1) (u_{0,1} -w_n) \overline{u}_{0,1}  dx=o_n(1).\]
		
		Consequently, fixed $R>0$, since $w_n\rightharpoonup u_{0,1}$ in $H^{1}(B_{R}(0),\mathbb{C})$ as $n\rightarrow+\infty$, for $n$ sufficiently large, we have
		\begin{align*}
			& 2 \operatorname{Re} \int_{B_{R}(0)}  \nabla_{A_{\varepsilon_n}(x+y_{n,1})} (u_{0,1} \chi_{\varepsilon_n}-v_n) \overline{\nabla_{A_{\varepsilon_n}(x+y_{n,1})} (u_{0,1}\chi_{\varepsilon_n})}dx\\
			&\qquad
			+2\operatorname{Re}\int_{B_{R}(0)} V_{\varepsilon_n}(x+y_{n,1}) (u_{0,1} \chi_{\varepsilon_n}-v_n) \overline{u}_{0,1}\chi_{\varepsilon_n}  dx\\
			&\quad=2\operatorname{Re} \int_{B_{R}(0)} \biggl( \nabla_{A(y_1)} ( u_{0,1}-w_n ) \overline{\nabla_{A(y_1)} u_{0,1}} + V(y_1) ( u_{0,1}-w_n ) \overline{u}_{0,1} \biggr) dx+o_n(1)=o_{n}(1).
		\end{align*}

		Moreover, since, due 
		to the equivalence of the norm in \eqref{fanshu},
		\[\|v_n\|_{H^1(B_{2|\log \varepsilon_n|}(0),\mathbb{C})}\leq C\left(\int_{B_{2|\log \varepsilon_n|}(0)} (|\nabla_{A_{\varepsilon_n}(x+y_{n,1})} v_n|^{2}+V_{\varepsilon_n}( x+y_{n,1})|v_n|^{2})dx\right)^{1/2} \leq C\|u_n\|_{\varepsilon_n}\leq C,\]
		and, for $n$ large enough, by \eqref{unifboundVA},
		\[
		|\nabla_{A_{\varepsilon_n}(x+y_{n,1})}(u_{0,1}\chi_{\varepsilon_n})|_2^2
		\leq2|\nabla(u_{0,1}\chi_{\varepsilon_n})|_2^2
		+2\int_{\mathbb{R}^2}|A_{\varepsilon_n}(x+y_{n,1})|^2|u_{0,1}\chi_{\varepsilon_n}|^2dx\leq
		C+\frac{16}{\log^2\varepsilon_n}|u_{0,1}|_2^2\leq C.
		\]
		and, analogously,
		\[
		\int_{B_R(0)^c}	|\nabla_{ A_{\varepsilon_n}(x+y_{n,1})} (u_{0,1}\chi_{\varepsilon_n})|^2dx
		\leq C \int_{B_R(0)^c}	(|\nabla u_{0,1}|^2+ |u_{0,1}|^2)dx,
		\]
		for any $\delta>0$,  there exists $R>0$ such that, for $n$ sufficiently large, 
		\begin{align*}
			&\biggl|\operatorname{Re} \int_{B_{2|\log \varepsilon_n|}(0)\setminus B_{R}(0)} \nabla_{A_{\varepsilon_n}(x+y_{n,1})} (u_{0,1} \chi_{\varepsilon_n}-v_n) \overline{\nabla_{A_{\varepsilon_n}(x+y_{n,1})} (u_{0,1}\chi_{\varepsilon_n})}dx\\
			&\qquad
			+\operatorname{Re} \int_{B_{2|\log \varepsilon_n|}(0)\setminus B_{R}(0)} V_{\varepsilon_n}(x+y_{n,1}) (u_{0,1} \chi_{\varepsilon_n}-v_n) \overline{u}_{0,1}\chi_{\varepsilon_n}  dx\biggr|\\
			&\quad\leq
			|\nabla_{ A_{\varepsilon_n}(x+y_{n,1})} ( u_{0,1} \chi_{\varepsilon_n}-v_n)|_{L^2(B_{2|\log \varepsilon_n|}(0))}|\nabla_{ A_{\varepsilon_n}(x+y_{n,1})} (u_{0,1}\chi_{\varepsilon_n})|_{L^2(B_{2|\log \varepsilon_n|}(0)\setminus B_{R}(0))}\\
			&\qquad+C|( u_{0,1} \chi_{\varepsilon_n}-v_n)|_{L^2(B_{2|\log \varepsilon_n|}(0))}|u_{0,1}|_{L^2(B_{2|\log \varepsilon_n|}(0)\setminus B_{R}(0))}\leq\delta.
		\end{align*}
		Consequently, $A_n=o_n(1)$.  
		
		Moreover, for the remaining part, observe that
		\begin{align*}
			&\int_{\mathbb{R}^2} \left(G(\varepsilon_{n}x, |u_{n}^{1}|^{2})-G(\varepsilon_{n}x, |u_{n}|^{2})+G(\varepsilon_{n}x, |\Psi_{n,1}|^{2})\right)dx\\
			&\quad=\int_{\mathbb{R}^2}\big(G(\varepsilon_n x+\varepsilon_ny_{n,1}, |v_n-u_{0,1}\chi_{\varepsilon_{n}}|^{2})-G(\varepsilon_n x+\varepsilon_ny_{n,1}, |v_n|^{2})+G(\varepsilon_n x+\varepsilon_ny_{n,1}, |u_{0,1}\chi_{\varepsilon_{n}}|^{2})\big)dx.
		\end{align*}

		Fixed $q>2$, for any $\xi>0$ and $\alpha>4\pi$, using \eqref{g3}, \eqref{p2.2}, and the Young inequality,  there exists a constant $C=C(\alpha,q, \xi)>0$ such that
		\begin{align*}
			&|G(\varepsilon_n x+\varepsilon_ny_{n,1},|v_n-u_{0,1}\chi_{\varepsilon_{n}}|^{2})-G(\varepsilon_n x+\varepsilon_ny_{n,1}, |v_n|^{2})|\\
			&\quad \leq 2\int_{0}^{1}|g(\varepsilon_n x+\varepsilon_ny_{n,1},|v_n-tu_{0,1}\chi_{\varepsilon_{n}}|^{2})||v_n-tu_{0,1}\chi_{\varepsilon_{n}}||u_{0,1}\chi_{\varepsilon_{n}}|dt\\
			&\quad \leq\int_{0}^{1}\left(\xi|v_n-tu_{0,1}\chi_{\varepsilon_{n}}||u_{0,1}\chi_{\varepsilon_{n}}|
			+C|v_n-tu_{0,1}\chi_{\varepsilon_{n}}|^{q-1}(e^{\alpha|v_n-tu_{0,1}\chi_{\varepsilon_{n}}|^{2}}-1)|u_{0,1}\chi_{\varepsilon_{n}}|\right)dt\\
			& \quad \leq \xi(|v_n|+|u_{0,1}|)|u_{0,1}|+C(|v_n|+|u_{0,1}|)^{q-1}(e^{\alpha(|v_n|+|u_{0,1}|)^{2}}-1)|u_{0,1}|\\
			& \quad \leq \xi\left(|v_n|^{2}+|v_n|^q(e^{\alpha(|v_n|+|u_{0,1}|)^{2}}-1)\right)
			+C\left(|u_{0,1}|^{2}+|u_{0,1}|^q(e^{\alpha(|v_n|+|u_{0,1}|)^{2}}-1)\right).
		\end{align*}
		Then, by \eqref{g3} and \eqref{p2.5}, we get
		\begin{equation}\label{p3.16}
			\begin{split}
				&|G(\varepsilon_n x+\varepsilon_ny_{n,1}, |v_n-u_{0,1}\chi_{\varepsilon_{n}}|^{2})-G(\varepsilon_n x+\varepsilon_ny_{n,1}, |v_n|^{2})+G(\varepsilon_n x+\varepsilon_ny_{n,1}, |u_{0,1}\chi_{\varepsilon_{n}}|^{2})|\\
				& \quad \leq \xi\left(|v_n|^{2}+|v_n|^q(e^{\alpha(|v_n|+|u_{0,1}|)^{2}}-1)\right)
				+C\left(|u_{0,1}|^{2}+|u_{0,1}|^q(e^{\alpha(|v_n|+|u_{0,1}|)^{2}}-1)\right).
			\end{split}
		\end{equation}
		
		Now, let 
		\begin{align*}
			G_{n,1}^{\xi}:=\max\bigg\{&\left|G(\varepsilon_n x+\varepsilon_ny_{n,1}, |v_n-u_{0,1}\chi_{\varepsilon_{n}}|^{2})-G(\varepsilon_n x+\varepsilon_ny_{n,1}, |v_n|^{2})+G(\varepsilon_n x+\varepsilon_ny_{n,1}, |u_{0,1}\chi_{\varepsilon_{n}}|^{2})\right|\\
			&-{}\xi(|v_n|^{2}+|v_n|^q(e^{\alpha(|v_n|+|u_{0,1}|)^{2}}-1)), 0\bigg\}.
		\end{align*}
		Note that, by \eqref{p3.16},
		\begin{equation*}
			G_{n,1}^{\xi}\leq C\left(|u_{0,1}|^{2}+|u_{0,1}|^q(e^{\alpha(|v_n|+|u_{0,1}|)^{2}}-1)\right)
		\end{equation*}
		and, since $v_n\rightarrow u_{0,1}$ a.e. in $\mathbb{R}^2$,
		\begin{equation*}
			G_{n,1}^{\xi}\rightarrow 0\,\ \text{and}\,\ e^{\alpha(|v_n|+|u_{0,1}|)^{2}}\rightarrow e^{4\alpha|u_{0,1}|^{2}}\,\ \text{a.e.}\,\ \text{in}\,\ \mathbb{R}^{2}\,\ \text{as}\,\ n\rightarrow +\infty.
		\end{equation*}
		
		Moreover, since \(|u_{0,1}|^{q}\in L^{r}(\mathbb{R}^{2},\mathbb{R})\) for every $r> 1$,
		by the H\"older inequality, \eqref{e}, and Lemma \ref{2.1}, $|u_{0,1}|^{q}(e^{4\alpha |u_{0,1}|^{2}}-1)\in  L^{1}(\mathbb{R}^{2},\mathbb{R})$, being
		\[
		\int_{\mathbb{R}^2} |u_{0,1}|^{q}(e^{4\alpha |u_{0,1}|^{2}}-1)dx
		\leq |u_{0,1}|_{rq}^{q}\left(\int_{\mathbb{R}^2} (e^{4\alpha r'|u_{0,1}|^{2}}-1)dx\right)^\frac{1}{r'},
		\]
		where $r'$ is the conjugate exponent of $r$.
		
		Now let us show that
		\begin{equation}\label{a2.15}
			\int_{\mathbb{R}^{2}}|u_{0,1}|^{q}(e^{\alpha(|v_{n}|+|u_{0,1}|)^{2}}-1)dx
			\longrightarrow
			\int_{\mathbb{R}^{2}}|u_{0,1}|^{q}(e^{4\alpha|u_{0,1}|^{2}}-1)dx.
		\end{equation}
		Since $|v_n|\rightharpoonup |u_{0,1}|$ in $H^{1}(\mathbb{R}^{2},\mathbb{R})$ as $n\rightarrow +\infty$, by the weak lower-semi continuity and \eqref{p3.5}, we have
		\[|\nabla|u_{0,1}||_{2}^{2}\leq  |\nabla|u_{0,1}||_{2}^{2}+V_0|u_{0,1}|_2^2\leq\mathop{\limsup}\limits_{n\rightarrow +\infty} (|\nabla|v_{n}||_{2}^{2}+V_0|v_n|_2^2)=\mathop{\limsup}\limits_{n\rightarrow +\infty} (|\nabla|u_{n}||_{2}^{2}+V_0|u_n|_2^2)=\iota^{2}<\frac{1}{8}.\]
		Then for any fixed $\kappa\in (0,1-4\iota^{2})$ and $n$ large enough,
		\begin{equation}\label{a2.16}
			|\nabla(|v_n|+|u_{0,1}|)|_{2}^{2}\leq 2|\nabla|v_n||_{2}^{2}+2|\nabla|u_{0,1}||_{2}^{2}\leq 4\iota^{2}+\kappa<1.
		\end{equation}
		Let
		$
		p^{*}\in \left(1,\frac{1}{4\iota^{2} + \kappa}\right)$, $\alpha\in \left(4\pi ,\frac{4\pi}{p^{*}(4\iota^{2} + \kappa)}\right)$.
		By \eqref{e}, \eqref{a2.16}, and Lemma \ref{2.1}, we get
		\begin{equation}\label{a2.18}
			\int_{\mathbb{R}^{2}}(e^{\alpha (|v_n|+|u_{0,1}|)^{2}} - 1)^{p^{*}}d x\leq \int_{\mathbb{R}^{2}}(e^{(4\iota^{2} + \kappa)p^{*}\alpha \left(\frac{|v_n|+|u_{0,1}|}{\sqrt{4\iota^{2} + \kappa}}\right)^{2}} - 1)d x\leq C, 
		\end{equation}
		and so, by Fatou's lemma,
		\(e^{4\alpha|u_{0,1}|^2}-1\in L^{p^*}(\mathbb R^2, \mathbb{R})\).\\
		Fix now
		\(\phi\in C_c^\infty(\mathbb R^2, \mathbb{R})\) and \(\sigma>0\).  By the
		Severini--Egorov theorem \cite[Theorem 4.29]{Brezis}, there is a
		measurable set \(A\subset\operatorname{supp}\phi\) such that
		\(|\operatorname{supp}\phi\setminus A|<\sigma\) and
		\(e^{\alpha(|v_n|+|u_{0,1}|)^2}-1\to e^{4\alpha|u_{0,1}|^2}-1\)
		uniformly on \(A\). If \((p^*)'=p^*/(p^*-1)\), then
		\begin{align*}
			\left|\int_{\mathbb R^2}
			\left(e^{\alpha(|v_n|+|u_{0,1}|)^2}-e^{4\alpha|u_{0,1}|^2}\right)
			\phi\,dx\right|
			&\leq |\phi|_\infty\left[ |A|
			\left|e^{\alpha(|v_n|+|u_{0,1}|)^2}
			-e^{4\alpha|u_{0,1}|^2}\right|_{L^\infty(A)}\right.\\
			&\quad+\left. |\operatorname{supp}\phi\setminus A|^\frac{1}{(p^*)'}
			\left(\left|e^{\alpha(|v_n|+|u_{0,1}|)^2}-1\right|_{p^*}
			+\left|e^{4\alpha|u_{0,1}|^2}-1\right|_{p^*}\right)\right]
			\\
			&\leq |\phi|_\infty \left[|A|
			\left|e^{\alpha(|v_n|+|u_{0,1}|)^2}
			-e^{4\alpha|u_{0,1}|^2}\right|_{L^\infty(A)}
			+C\sigma^{1/(p^*)'}\right].
		\end{align*}
		Letting first \(n\to+\infty\) and then \(\sigma\to0^+\), we obtain
		\[
		\int_{\mathbb R^2}\left(e^{\alpha(|v_n|+|u_{0,1}|)^2}-1\right)\phi\,dx
		\to\int_{\mathbb R^2}\left(e^{4\alpha|u_{0,1}|^2}-1\right)\phi\,dx.
		\]
		Since
		\(C_c^\infty(\mathbb R^2)\) is dense in
		\(L^{(p^*)'}(\mathbb R^2)\) and
		\(\{e^{\alpha(|v_n|+|u_{0,1}|)^2}-1\}\) is bounded in \(L^{p^*}\),
		it follows that
		\[
		e^{\alpha(|v_n|+|u_{0,1}|)^2}-1
		\rightharpoonup e^{4\alpha|u_{0,1}|^2}-1
		\quad\text{in }L^{p^*}(\mathbb R^2,\mathbb R).
		\]
		Therefore, since \(|u_{0,1}|^{q}\in L^{(p^{*})'}(\mathbb{R}^{2},\mathbb{R})\), the previous weak convergence tested in \(|u_{0,1}|^q\) gives \eqref{a2.15}, and, by \eqref{a2.18}, $|u_{0,1}|^{q}(e^{\alpha(|v_n|+ |u_{0,1}|)^{2}}-1)\in L^1(\mathbb{R}^2,\mathbb{R})$.
		Hence, Lemma \ref{GDCT}
		yields
		\[
		G_{n,1}^{\xi}\to0
		\quad\text{in }L^1(\mathbb{R}^2,\mathbb{R})
		\quad\text{as }n\to+\infty.
		\]
		On the other hand, by the definition of \(G_{n,1}^{\xi}\),
		\begin{align*}
			&\left|G(\varepsilon_n x+\varepsilon_ny_{n,1}, |v_n-u_{0,1}\chi_{\varepsilon_{n}}|^{2})-G(\varepsilon_n x+\varepsilon_ny_{n,1}, |v_n|^{2})+G(\varepsilon_n x+\varepsilon_ny_{n,1}, |u_{0,1}\chi_{\varepsilon_{n}}|^{2})\right|\\
			&\quad\leq \xi(|v_n|^{2} + |v_n|^{q}(e^{\alpha (|v_n|+|u_{0,1}|)^{2}} - 1)) + G_{n,1}^{\xi}.
		\end{align*}
		Using the H\"older inequality, the boundedness of the sequence \(\{|v_n|\}\) in $H^1(\mathbb{R}^2,\mathbb{R})$, the Sobolev inequality, and \eqref{a2.18}, we have
		\[
		\int_{\mathbb{R}^{2}}\left(|v_n|^{q}(e^{\alpha (|v_n|+|u_{0,1}|)^{2}} - 1)\right)\,dx
		\leq |v_n|_{q(p^*)'}^{q}\left(\int_{\mathbb{R}^{2}}(e^{\alpha (|v_n|+|u_{0,1}|)^{2}} - 1)^{p^{*}}dx\right)^{\frac{1}{p^{*}}} \leq C.
		\]
		Thus, the arbitrariness of \(\xi\) allows us to conclude that
		\[
		\int_{\mathbb{R}^2} \left(G(\varepsilon_{n}x, |u_{n}^{1}|^{2})-G(\varepsilon_{n}x, |u_{n}|^{2})+G(\varepsilon_{n}x, |\Psi_{n,1}|^{2})\right)dx=o_{n}(1).		
		\]
		Hence, the proof of \eqref{p3.13} is complete.
		
		Now let us prove that, up to a subsequence,
		\begin{equation}\label{p3.12}
			\|J^{\prime}_{\varepsilon_{n}}(u^1_{n})\|_{\varepsilon_n}
			\to 0		\,\ \text{as}\,\ n\rightarrow +\infty.
		\end{equation}
		For any $\phi_{n}\in H_{\varepsilon_{n}}$ with $\|\phi_{n}\|_{\varepsilon_{n}}\leq 1$,
		\begin{align*}
			J^{\prime}_{\varepsilon_{n}}(u_{n}^{1})[\phi_{n}]
			&=J^{\prime}_{\varepsilon_{n}}(u_{n})[\phi_{n}]-J^{\prime}_{\varepsilon_{n}}(\Psi_{n,1})[\phi_{n}]\\
			&\quad
			-\operatorname{Re}\int_{\mathbb{R}^{2}} \big[g(\varepsilon_{n}x, |u_{n}^{1}|^{2})u_{n}^{1}-g(\varepsilon_{n}x, |u_{n}|^{2})u_{n}
			+g(\varepsilon_{n}x,|\Psi_{n,1}|^{2})\Psi_{n,1}\big]\overline{\phi}_{n}dx.
		\end{align*}
		
		By assumption, $
		\|J^{\prime}_{\varepsilon_{n}}(u_{n})\|_{\varepsilon_n}
		=o_n(1)$. Moreover, 
		\begin{equation}\label{I1}
			\begin{split}
				J^{\prime}_{\varepsilon_{n}}(\Psi_{n,1})[\phi_{n}]
				&=
				\operatorname{Re}\int_{\mathbb{R}^2} \nabla_{A_{\varepsilon_n}(x+y_{n,1})} (u_{0,1}\chi_{\varepsilon_{n}}) \overline{\nabla_{A_{\varepsilon_n}(x+y_{n,1})} \phi_{n}(x+y_{n,1})}dx\\
				&\quad +\operatorname{Re}\int_{\mathbb{R}^2}  V_{\varepsilon_n}( x+y_{n,1})u_{0,1}\chi_{\varepsilon_{n}}\overline{\phi_{n}(x+y_{n,1})}dx\\
				&\quad 
				-\operatorname{Re}\int_{\mathbb{R}^2}g(\varepsilon_n x+\varepsilon_n y_{n,1}, |u_{0,1}\chi_{\varepsilon_{n}}|^{2})u_{0,1}\chi_{\varepsilon_{n}}\overline{\phi_{n}(x+y_{n,1})}dx.       
			\end{split}
		\end{equation}
		Since, by \eqref{unifboundVA},
		\[
		\begin{aligned}
			\int_{B_{2|\log\varepsilon_n|}(0)} \left|  \nabla_{A_{\varepsilon_n}(x+y_{n,1})}u_{0,1} \right|^2 dx & \leq\int_{B_{2|\log\varepsilon_n|}(0)} 2\left( |\nabla u_{0,1}|^2 + |A_{\varepsilon_n}(x+y_{n,1})u_{0,1}|^2 \right)  dx \\
			& \leq 2 \int_{B_{2|\log\varepsilon_n|}(0)} |\nabla u_{0,1}|^2 dx + C\int_{B_{2|\log\varepsilon_n|}(0)} |u_{0,1}|^2 dx\leq C\|u_{0,1}\|^2,
		\end{aligned}
		\] 
		and so
		\begin{align*}
			&\left|\operatorname{Re}\int_{\mathbb{R}^2}
			\big(i\overline{\phi_n(x+y_{n,1})}\nabla\chi_{\varepsilon_n}
			\nabla_{A_{\varepsilon_n}(x+y_{n,1})}u_{0,1}+iu_{0,1}\nabla\chi_{\varepsilon_n}
			\overline{\nabla_{A_{\varepsilon_n}(x+y_{n,1})}\phi_n(x+y_{n,1})}\big)
			dx\right|\\
			&\quad\leq\frac{2}{|\log \varepsilon_n|}
			\left(|\phi_n|_2|\nabla_{A_{\varepsilon_n}(x+y_{n,1})}u_{0,1}|_{L^2(B_{2|\log\varepsilon_n|}(0))}+|u_{0,1}|_2 |\nabla_{A_{\varepsilon_n}} \phi_{n}|_2\right)
			\leq \frac{C}{|\log \varepsilon_n|}=o_{n}(1)
		\end{align*}
		uniformly with respect to $\phi_{n}\in H_{\varepsilon_{n}}$ with $\|\phi_{n}\|_{\varepsilon_{n}}\leq 1$, for the first term of \eqref{I1} we have
		\begin{align*}
			&\operatorname{Re}\int_{\mathbb{R}^2}\nabla_{A_{\varepsilon_n}(x+y_{n,1})}(u_{0,1}\chi_{\varepsilon_n})
			\overline{\nabla_{A_{\varepsilon_n}(x+y_{n,1})}\phi_n(x+y_{n,1})}dx\\
			&\quad=\operatorname{Re}\int_{\mathbb{R}^2}\chi_{\varepsilon_n}\nabla_{A_{\varepsilon_n}(x+y_{n,1})}u_{0,1}
			\overline{\nabla_{A_{\varepsilon_n}(x+y_{n,1})}\phi_n(x+y_{n,1})}dx\\
			&\qquad-\operatorname{Re}\int_{\mathbb{R}^2}iu_{0,1}\nabla\chi_{\varepsilon_n}
			\overline{\nabla_{A_{\varepsilon_n}(x+y_{n,1})}\phi_n(x+y_{n,1})}dx\\
			&\quad=\operatorname{Re}\int_{\mathbb{R}^2}\chi_{\varepsilon_n}\nabla_{A_{\varepsilon_n}(x+y_{n,1})}u_{0,1}
			\overline{\nabla_{A_{\varepsilon_n}(x+y_{n,1})}\phi_n(x+y_{n,1})}dx\\
			&\qquad+\operatorname{Re}\int_{\mathbb{R}^2}i\overline{\phi_n(x+y_{n,1})}\nabla\chi_{\varepsilon_n}
			\nabla_{A_{\varepsilon_n}(x+y_{n,1})}u_{0,1}dx+o_n(1)\\
			&\quad=\operatorname{Re}\int_{\mathbb{R}^2}\nabla_{A_{\varepsilon_n}(x+y_{n,1})}u_{0,1}
			\overline{\nabla_{A_{\varepsilon_n}(x+y_{n,1})}(\chi_{\varepsilon_n}\phi_n(x+y_{n,1}))}dx+o_n(1).
		\end{align*}
		Moreover, by the H\"older inequality,
		\begin{align*}  
			&\left|\operatorname{Re} \int_{\mathbb{R}^2} \biggl( \nabla_{A_{\varepsilon_n}(x+y_{n,1})} u_{0,1} \overline{\nabla_{A_{\varepsilon_n}(x+y_{n,1})} (\chi_{\varepsilon_{n}}\phi_{n}(x+y_{n,1})})-\nabla_{A(y_1)} u_{0,1} \overline{\nabla_{A(y_1)} (\phi_{n}(x+y_{n,1})\chi_{\varepsilon_{n}})}\biggr) dx\right|\\
			&\quad\leq
			\int_{B_{2|\log \varepsilon_n|}(0)}|A_{\varepsilon_n}(x+y_{n,1})-A(y_1)| |u_{0,1}| |\nabla_{A_{\varepsilon_n}(x+y_{n,1})} (\chi_{\varepsilon_{n}}\phi_{n}(x+y_{n,1}))|dx\\
			&\qquad
			+\int_{B_{2|\log \varepsilon_n|}(0)} |\nabla_{A(y_1)} u_{0,1}||A_{\varepsilon_n}(x+y_{n,1})-A(y_1)| |(\chi_{\varepsilon_{n}}\phi_{n}(x+y_{n,1})) |dx\\
			&\quad\leq\left(\int_{B_{2|\log \varepsilon_n|}(0)}\left(|A_{\varepsilon_n}(x+y_{n,1})-A(y_1)| |u_{0,1} |\right)^2dx\right)^{1/2}|\nabla_{A_{\varepsilon_n}(x+y_{n,1})} (\chi_{\varepsilon_{n}}\phi_{n}(x+y_{n,1}))|_2\\
			&\qquad+\left(\int_{B_{2|\log \varepsilon_n|}(0)} \left(|A_{\varepsilon_n}(x+y_{n,1})-A(y_1)| |\nabla_{A(y_1)} u_{0,1} |\right)^2 dx\right)^{1/2}|\chi_{\varepsilon_{n}}\phi_{n}(x+y_{n,1})|_2.
		\end{align*}
		Since $\left(|A_{\varepsilon_n}(\cdot+y_{n,1})-A(y_1)| |u_{0,1}|\right)^2\chi_{B_{2|\log\varepsilon_n|}(0)}\to 0$ a.e. in $\mathbb{R}^2$ and, by \eqref{unifboundVA},
		\[
		\left(|A_{\varepsilon_n}(\cdot+y_{n,1})-A(y_1)| |u_{0,1}|\right)^2\chi_{B_{2|\log\varepsilon_n|}(0)}\le C|u_{0,1}|^2\in L^1(\mathbb{R}^2,\mathbb{R}),
		\]
		applying the Lebesgue Dominated Convergence Theorem, we get
		\[
		\int_{B_{2|\log \varepsilon_n|}(0)}\left(|A_{\varepsilon_n}(x+y_{n,1})-A(y_1)| |u_{0,1}|\right)^2dx=o_n(1).
		\]
		Similarly,  using \(C|\nabla_{A(y_1)}u_{0,1}|^2\in L^1(\mathbb R^2,\mathbb R)\) as dominating function, we obtain that
		\[
		\int_{B_{2|\log \varepsilon_n|}(0)} \left(|A_{\varepsilon_n}(x+y_{n,1})-A(y_1)| |\nabla_{A(y_1)} u_{0,1} |\right)^2dx=o_n(1).
		\]
		Furthermore, arguing as in \eqref{before}, there exists $C>0$ such that for every $\phi_{n}\in H_{\varepsilon_{n}}$ with $\|\phi_{n}\|_{\varepsilon_{n}}\leq 1$,
		\begin{equation}\label{before1}
			\begin{split}
				&\left(\int_{\mathbb{R}^2} (|\nabla_{A_{\varepsilon_n}(x+y_{n,1})} \big(\phi_{n}(x+y_{n,1})\chi_{\varepsilon_{n}}\big)|^{2}+V_{\varepsilon_n}( x+y_{n,1})|\phi_{n}(x+y_{n,1})\chi_{\varepsilon_{n}}|^{2})dx\right)^\frac{1}{2}\\
				&\qquad
				=\|\phi_{n}\chi_{\varepsilon_{n}}(x-y_{n,1})\|_{\varepsilon_n}\leq C.   
			\end{split}
		\end{equation}
		Thus
		\[
		\operatorname{Re}\int_{\mathbb{R}^2} \left(\nabla_{A_{\varepsilon_n}(x+y_{n,1})}(u_{0,1}\chi_{\varepsilon_{n}}) \overline{\nabla_{A_{\varepsilon_n}(x+y_{n,1})} \phi_{n}(x+y_{n,1})}
		- \nabla_{A(y_1)} u_{0,1} \overline{\nabla_{A(y_1)} (\phi_{n}(x+y_{n,1})\chi_{\varepsilon_{n}})}\right)dx=o_n(1)
		\]
		uniformly with respect to $\phi_{n}\in H_{\varepsilon_{n}}$ with $\|\phi_{n}\|_{\varepsilon_{n}}\leq 1$.
		
		Hence, using that $u_{0,1}$ is a non-trivial solution of \eqref{p3.9},  
		we obtain that
		\begin{equation*}
			\begin{split}
				&\operatorname{Re}\int_{\mathbb{R}^2}\nabla_{A_{\varepsilon_n}(x+y_{n,1})}(u_{0,1}\chi_{\varepsilon_n})
				\overline{\nabla_{A_{\varepsilon_n}(x+y_{n,1})}\phi_n(x+y_{n,1})}dx\\
				&\quad=-\operatorname{Re}\int_{\mathbb{R}^2}V(y_1)u_{0,1}\chi_{\varepsilon_n}
				\overline{\phi_n(x+y_{n,1})}dx
				+\operatorname{Re}\int_{\mathbb{R}^2} g_{0,1}(x,|u_{0,1}|^2)
				u_{0,1}\chi_{\varepsilon_n}\overline{\phi_n(x+y_{n,1})}dx+o_n(1).
			\end{split}
		\end{equation*}
		and so,  
		\begin{align*}
			J^{\prime}_{\varepsilon_n}(\Psi_{n,1})[\phi_n]
			&=\operatorname{Re}\int_{\mathbb{R}^2}\left(V_{\varepsilon_n}(x+y_{n,1})-V(y_1)\right)
			u_{0,1}\chi_{\varepsilon_n}\overline{\phi_n(x+y_{n,1})}dx\\
			&\quad+\operatorname{Re}\int_{\mathbb{R}^2}\left[g_{0,1}(x,|u_{0,1}|^2)
			-g(\varepsilon_nx+\varepsilon_ny_{n,1},|u_{0,1}\chi_{\varepsilon_n}|^2)\right] u_{0,1}\chi_{\varepsilon_n}\overline{\phi_n(x+y_{n,1})}dx+o_n(1),
		\end{align*}
		where $o_n(1)$ is uniform with respect to $\phi_{n}\in H_{\varepsilon_{n}}$ with $\|\phi_{n}\|_{\varepsilon_{n}}\leq 1$.
		
		The H\"older inequality, \eqref{unifboundVA}, and the Lebesgue Dominated Convergence Theorem imply
		\begin{align*}
			&\left|\operatorname{Re}\int_{\mathbb{R}^2}  \left(V_{\varepsilon_n}( x+y_{n,1})-V(y_1)\right)u_{0,1}\chi_{\varepsilon_{n}}\overline{\phi_{n}(x+y_{n,1})}dx\right|\\
			&\quad
			\leq\left(\int_{B_{2|\log\varepsilon_n|}(0)}  \left|V_{\varepsilon_n}( x+y_{n,1})-V(y_1)\right|^2|u_{0,1}|^2dx\right)^{1/2}|\chi_{\varepsilon_{n}}\phi_{n}(x+y_{n,1})|_2\\
			&\quad\leq C\left(\int_{B_{2|\log\varepsilon_n|}(0)}  \left|V_{\varepsilon_n}( x+y_{n,1})-V(y_1)\right|^2|u_{0,1}|^2dx\right)^{1/2}=o_{n}(1),
		\end{align*}
		where $C>0$ does not depend on $\phi_{n}\in H_{\varepsilon_{n}}$ with $\|\phi_{n}\|_{\varepsilon_{n}}\leq 1$.\\

		Furthermore, by \eqref{p2.2} and \eqref{e}, fixed $q>2$, for any $\xi>0$ and $\alpha>4\pi$, there exists $C=C(\alpha,\xi,q)$ such that
		\[
		f^2(|u_{0,1}|^2)|u_{0,1}|^2
		\leq \xi|u_{0,1}|^2+
		C|u_{0,1}|^{2(q-1)}(e^{2\alpha|u_{0,1}|^2}-1)
		\]
		and by the Sobolev embedding, the H\"older inequality, and Lemma \ref{2.1}, the function in the right hand side is in $L^1(\mathbb{R}^2,\mathbb{R})$, so that
		\begin{equation}\label{fhatfL2}
			(f(|u_{0,1}|^2)-\widehat f(|u_{0,1}|^2))u_{0,1}
			\in L^2(\mathbb{R}^2,\mathbb{C}).
		\end{equation}
		Set, now,
		\[
		\begin{aligned}
			R_n&
			:=\left\{\chi_{\Lambda_\Gamma}(\varepsilon_n(\cdot+y_{n,1}))
			\bigl[f(|u_{0,1}|^2)-f(|u_{0,1}\chi_{\varepsilon_n}|^2)\bigr]\right.\\
			&\qquad
			\left.+\bigl(1-\chi_{\Lambda_\Gamma}(\varepsilon_n(\cdot+y_{n,1}))\bigr)
			\bigl[\widehat f(|u_{0,1}|^2)
			-\widehat f(|u_{0,1}\chi_{\varepsilon_n}|^2)\bigr]\right\}
			u_{0,1}\chi_{\varepsilon_n}.
		\end{aligned}
		\]
		Since \(\chi_{\varepsilon_n}\to1\) a.e. in $\mathbb{R}^2$, the continuity of \(f\) and
		\(\widehat f\) gives \(R_n\to0\) a.e. in $\mathbb{R}^2$. Moreover,
		\[
		|R_n|^2\leq4f^2(|u_{0,1}|^2)|u_{0,1}|^2\in L^1(\mathbb{R}^2,\mathbb{R}).
		\]
		Thus
		\(|R_n|_{2}\to0\).
		Therefore, 
		\begin{align*}
			&\bigl[g_{0,1}(\cdot,|u_{0,1}|^2)
			-g(\varepsilon_n(\cdot+y_{n,1}),
			|u_{0,1}\chi_{\varepsilon_n}|^2)\bigr]
			u_{0,1}\chi_{\varepsilon_n}\\
			&\quad
			=\Bigl[\eta_1f(|u_{0,1}|^2)+(1-\eta_1)\widehat f(|u_{0,1}|^2)
			-\chi_{\Lambda_\Gamma}(\varepsilon_n(\cdot+y_{n,1}))
			f(|u_{0,1}|^2)\\
			&\qquad-\bigl(1-\chi_{\Lambda_\Gamma}(\varepsilon_n(\cdot+y_{n,1}))\bigr)\widehat f(|u_{0,1}|^2)\Bigr]
			u_{0,1}\chi_{\varepsilon_n}+R_n\\
			&\quad
			=\bigl(\eta_1-\chi_{\Lambda_\Gamma}
			(\varepsilon_n(\cdot+y_{n,1}))\bigr)
			\bigl(f(|u_{0,1}|^2)-\widehat f(|u_{0,1}|^2)\bigr)
			u_{0,1}\chi_{\varepsilon_n}+R_n.
		\end{align*}
		We claim that
		\begin{equation*}
			\sup_{\|\phi_n\|_{\varepsilon_n}\leq1}
			\left|\int_{\mathbb{R}^2}
			\bigl(\eta_1-\chi_{\Lambda_\Gamma}
			(\varepsilon_n(x+y_{n,1}))\bigr)
			\bigl(f(|u_{0,1}|^2)-\widehat f(|u_{0,1}|^2)\bigr)
			u_{0,1}\chi_{\varepsilon_n}
			\overline{\phi_n(x+y_{n,1})}\,dx\right|=o_n(1).
		\end{equation*}
		Indeed, if this were false, after passing to a subsequence there would exist
		\(\delta>0\) and \(\{\phi_n\}\subset  H_{\varepsilon_n}\) with $\|\phi_n\|_{\varepsilon_n}\leq1$ such that
		\begin{equation}\label{localweakuniformapplication_contr}
			\left|\int_{\mathbb{R}^2}
			\bigl(\eta_1-\chi_{\Lambda_\Gamma}
			(\varepsilon_n(x+y_{n,1}))\bigr)
			\bigl(f(|u_{0,1}|^2)-\widehat f(|u_{0,1}|^2)\bigr)
			u_{0,1}\chi_{\varepsilon_n}
			\overline{\phi_n(x+y_{n,1})}\,dx\right|\geq \delta.
		\end{equation}
		Set \(\zeta_n:=\chi_{\varepsilon_n}\phi_n(\cdot+y_{n,1})\). Since each $\zeta_n$ is supported in $B_{2|\log\varepsilon_n|}(0)$, \eqref{before1} and the uniform
		equivalence of \eqref{fanshu} with the standard
		\(H^1\)-norm on \(B_{2|\log\varepsilon_n|}(0)\)
		yields that
		\(\{\zeta_n\}\) is bounded in \(H^1(\mathbb{R}^2,\mathbb C)\), uniformly for
		\(\|\phi_n\|_{\varepsilon_n}\leq1\).
		Hence, up to a subsequence, there exists
		\(\zeta\in H^1(\mathbb{R}^2,\mathbb C)\) such that
		\(\zeta_n\rightharpoonup\zeta\) in \(H^1(\mathbb{R}^2,\mathbb C)\) and
		\(\zeta_n\to\zeta\) in \(L^2_{\rm loc}(\mathbb{R}^2,\mathbb C)\).
		For every \(R>0\), H\"older's inequality gives
		\begin{align*}
			|(f(|u_{0,1}|^2)-\widehat f(|u_{0,1}|^2))u_{0,1}(\overline{\zeta_n}-\overline\zeta)|_{1}
			&\leq |(f(|u_{0,1}|^2)-\widehat f(|u_{0,1}|^2))u_{0,1}|_{L^2(B_R(0))}
			|\zeta_n-\zeta|_{L^2(B_R(0))}\\
			&\quad
			+|(f(|u_{0,1}|^2)-\widehat f(|u_{0,1}|^2))u_{0,1}|_{L^2(B_R(0)^c)}
			\bigl(|\zeta_n|_{2}
			+|\zeta|_{2}\bigr).
		\end{align*}
		For fixed \(R>0\), the first term tends to zero as \(n\to+\infty\), while, due to the boundedness of $\{\zeta_n\}$ and \eqref{fhatfL2}, the second term tends to $0$ as $R\to +\infty$. Thus
		\[
		(f(|u_{0,1}|^2)-\widehat f(|u_{0,1}|^2))u_{0,1}\overline{\zeta_n}
		\to
		(f(|u_{0,1}|^2)-\widehat f(|u_{0,1}|^2))u_{0,1}\overline\zeta
		\quad
		\text{in }L^1(\mathbb{R}^2,\mathbb{C}).
		\] 
		Therefore, by \eqref{localweakproduct},
		\[
		\int_{\mathbb{R}^2}\chi_{\Lambda_\Gamma}(\varepsilon_nx+\varepsilon_ny_{n,1})
		(f(|u_{0,1}|^2)-\widehat f(|u_{0,1}|^2))u_{0,1} \overline{\zeta_n}\,dx\to\int_{\mathbb{R}^2}\eta_1(f(|u_{0,1}|^2)-\widehat f(|u_{0,1}|^2))u_{0,1} \overline\zeta\,dx
		\]
		and, since $\eta_1$ is a measurable function with \(0\leq\eta_1\leq1\), by \eqref{fhatfL2}
		\[
		\int_{\mathbb{R}^2}\eta_1(f(|u_{0,1}|^2)-\widehat f(|u_{0,1}|^2))u_{0,1}\overline{\zeta_n}\,dx
		\to\int_{\mathbb{R}^2}\eta_1(f(|u_{0,1}|^2)-\widehat f(|u_{0,1}|^2))u_{0,1}\overline\zeta\,dx.
		\]
		Then, subtracting the last two limits, we get a contradiction with \eqref{localweakuniformapplication_contr}.
		
		Finally, the Cauchy--Schwarz
		inequality and \eqref{V1} give
		\[
		\left|\int_{\mathbb{R}^2}R_n
		\overline{\phi_n(x+y_{n,1})}\,dx\right|
		\leq |R_n|_{L^2(\mathbb{R}^2)}
		|\phi_n|_{L^2(\mathbb{R}^2)}
		\leq V_0^{-1/2}|R_n|_{L^2(\mathbb{R}^2)}=o_n(1),
		\]
		uniformly for \(\|\phi_n\|_{\varepsilon_n}\leq1\). 
		
		Hence,
		\[
		\begin{aligned}
			&\left|\operatorname{Re}\int_{\mathbb{R}^2}
			\bigl[g_{0,1}(x,|u_{0,1}|^2)
			-g(\varepsilon_nx+\varepsilon_ny_{n,1},
			|u_{0,1}\chi_{\varepsilon_n}|^2)\bigr]
			u_{0,1}\chi_{\varepsilon_n}
			\overline{\phi_n(x+y_{n,1})}\,dx\right|=o_n(1),
		\end{aligned}
		\]
		uniformly with respect to
		\(\phi_n\in H_{\varepsilon_n}\) with
		\(\|\phi_n\|_{\varepsilon_n}\leq1\), and so
		\[
		\|J^{\prime}_{\varepsilon_{n}}(\Psi_{n,1})\|_{\varepsilon_n}
		=o_n(1).
		\]

		Finally, observe that, using the H\"older inequality,
		\[
		\left|\operatorname{Re}\int_{\mathbb{R}^{2}} \big[g(\varepsilon_{n}x, |u_{n}^{1}|^{2})u_{n}^{1}-g(\varepsilon_{n}x, |u_{n}|^{2})u_{n}
		+g(\varepsilon_{n}x,|\Psi_{n,1}|^{2})\Psi_{n,1}\big]\overline{\phi}_{n}dx\right|
		\leq
		|\mathfrak{G}_n|_2
		| \phi_n|_2
		\leq C  |\mathfrak{G}_n|_2,
		\]
		where
		\begin{align*}
			\mathfrak{G}_n&:=|g(\varepsilon_n (\cdot+ y_{n,1}), |v_{n}-u_{0,1}\chi_{\varepsilon_{n}}|^{2})(v_{n}-u_{0,1}\chi_{\varepsilon_{n}})-g(\varepsilon_n (\cdot+ y_{n,1}), |v_{n}|^{2})v_{n}\\
			&\quad
			+g(\varepsilon_n (\cdot + y_{n,1}), |u_{0,1}\chi_{\varepsilon_{n}}|^{2})u_{0,1}\chi_{\varepsilon_{n}}|.
		\end{align*}
		The function $\mathfrak{G}_n$ tends to $0$ as $n\to +\infty$ a.e. in $\mathbb{R}^2$.

		Since, for a.e. fixed \(x\in\mathbb R^2\), the map \(t\mapsto g(x,t)\) is continuous and piecewise \(C^1\), it is locally absolutely continuous and so, by the chain rule and \eqref{g7}, for all \(c,b\in\mathbb C\) and a.e. \(x\in\mathbb R^2\),
		\[
		\left|g(\varepsilon_n(x+y_{n,1}),|c-b|^2)
		-g(\varepsilon_n(x+y_{n,1}),|c|^2)\right|\leq2\int_0^1
		\bigl(e^{4\pi|c-sb|^2}-1\bigr)|c-sb|\,|b|\,ds,
		\]
		using also (\ref{g3}), and \eqref{p2.2}, fixed $q>2$, for any $\xi>0$ and $\alpha>4\pi$, there exists $C=C(\alpha,\xi,q)>0$ such that
		\begin{align*}
			\mathfrak{G}_n(x)
			&\leq
			\left|g(\varepsilon_n (x+ y_{n,1}), |v_n-u_{0,1}\chi_{\varepsilon_{n}}|^{2})-g(\varepsilon_n (x+ y_{n,1}), |v_n|^{2})\right||v_n|\\
			&\quad
			+|g(\varepsilon_n (x+ y_{n,1}), |u_{0,1}\chi_{\varepsilon_{n}}|^{2})-g(\varepsilon_n (x+ y_{n,1}), |v_n-u_{0,1}\chi_{\varepsilon_{n}}|^{2})||u_{0,1}|\\
			&\leq
			2|v_n| \int_{0}^{1}
			(e^{4\pi|v_n-s u_{0,1}\chi_{\varepsilon_{n}}|^{2}}-1)
			|v_n-s u_{0,1}\chi_{\varepsilon_{n}}||u_{0,1}\chi_{\varepsilon_{n}}|ds\\
			&\qquad
			+\left(\xi+C(|v_n|+|u_{0,1}|)^{q-2}(e^{\alpha(|v_n|+|u_{0,1}|)^{2}}-1)\right)|u_{0,1}|\\
			&\leq 
			\left[2(e^{4\pi(|v_n|+|u_{0,1}|)^{2}}-1)(|v_n|+|u_{0,1}|)|v_n|
			+\xi+C(|v_n|+|u_{0,1}|)^{q-2}(e^{\alpha(|v_n|+|u_{0,1}|)^{2}}-1)\right]|u_{0,1}|.
		\end{align*}

		Then, using also  \eqref{e},
		there exist $C=C(\alpha,\xi,q)>0$ such that
		\begin{equation}\label{p3.16b}
			\begin{split}
				\mathfrak{G}_n^2 &\leq \left[\xi^2+C(e^{2\alpha(|v_n|+|u_{0,1}|)^{2}}-1)\left((|v_n|^2+|u_{0,1}|^2)|v_n|^2+(|v_n|^{2q-4}+|u_{0,1}|^{2q-4})\right)\right]|u_{0,1}|^2.      
			\end{split}
		\end{equation}
		
		Now, let 
		\[g_{n,1}^{\xi}:=\max\{\mathfrak{G}_n^2-\xi^2|u_{0,1}|^2, 0\}.
		\]
		Note that $g_{n,1}^{\xi}\to 0$ a.e. in $\mathbb{R}^2$ and, by \eqref{p3.16b},
		\begin{equation}\label{a2.12b}
			g_{n,1}^{\xi}
			\leq
			C(e^{2\alpha(|v_n|+|u_{0,1}|)^{2}}-1)\left((|v_n|^2+|u_{0,1}|^2)|v_n|^2+(|v_n|^{2q-4}+|u_{0,1}|^{2q-4})\right)|u_{0,1}|^2.
		\end{equation}
		
		Moreover, since $v_n\rightarrow u_{0,1}$ a.e. in $\mathbb{R}^2$, the right hand side (up to the multiplicative constant) in \eqref{a2.12b} tends to $2(e^{8\alpha|u_{0,1}|^{2}}-1)(|u_{0,1}|^6+|u_{0,1}|^{2q-2})$ a.e. in  $\mathbb{R}^{2}$ as  $n\rightarrow +\infty$ that, since \(|u_{0,1}|\in L^{r}(\mathbb{R}^{2},\mathbb{R})\) for every $r\geq 2$,
		by the H\"older inequality, \eqref{e}, and Lemma \ref{2.1}, belongs to $L^{1}(\mathbb{R}^{2},\mathbb{R})$.
		
		Now, let us show that the right hand side (up to the multiplicative constant) in \eqref{a2.12b} belongs to $L^{1}(\mathbb{R}^{2},\mathbb{R})$ and tends to $2(e^{8\alpha|u_{0,1}|^{2}}-1)(|u_{0,1}|^6+|u_{0,1}|^{2q-2})$ in $L^{1}(\mathbb{R}^{2},\mathbb{R})$ as $n\rightarrow +\infty$.
		
		Let
		\(
		\kappa\in \left(0,\frac{1}{2}-4\iota^{2}\right)\), \(\alpha\in \left(4\pi, \frac{2\pi}{4\iota^{2}+\kappa}\right)\), \( p^{*}\in \left(1,\frac{2\pi}{\alpha(4\iota^{2} + \kappa)}\right)\),
		and
		\[
		\tau\in (1,\tau^*)
		\quad
		\tau^*:=
		\begin{cases}
			\displaystyle\min\left\{\frac{2\pi}{\alpha p^{*}(4\iota^{2} + \kappa)},\frac{1}{3-q}\right\}
			&\text{if }2<q<3,\vspace{5pt}\\
			\displaystyle\frac{2\pi}{\alpha p^{*}(4\iota^{2} + \kappa)}
			&\text{if }q\geq 3.
		\end{cases}
		\]
		Observe that, if \(\tau'=\tau/(\tau-1)\) is the conjugate exponent of \(\tau\), then $(2q-4)p^*\tau'>2$. Indeed, if \(2<q<3\), then \(\tau<1/(3-q)\), so that \(\tau'>1/(q-2)\) and $(2q-4)p^*\tau'>2p^*>2$, while, if \(q\geq3\), then \(p^*>1\) and \(\tau'>1\) give the desired inequality.
		Then, by \eqref{e}, the boundedness of $\{|v_n|\}$ in $H^1(\mathbb{R}^2,\mathbb{R})$, \eqref{a2.16}, and Lemma \ref{2.1}, we get
		\begin{align*}
			&\int_{\mathbb{R}^2} \left[(|v_n|^2+|u_{0,1}|^2)|v_n|^2+|v_n|^{2q-4}+|u_{0,1}|^{2q-4}\right]^{p^*}(e^{2\alpha(|v_n|+|u_{0,1}|)^{2}}-1)^{p^*} dx \\
			&\quad
			\leq \left(\int_{\mathbb{R}^2}\left[(|v_n|^2+|u_{0,1}|^2)|v_n|^2+|v_n|^{2q-4}+|u_{0,1}|^{2q-4}\right]^{p^*\tau'} dx\right)^\frac{1}{\tau'}\left(\int_{\mathbb{R}^2} (e^{2\alpha(|v_n|+|u_{0,1}|)^{2}}-1)^{p^*\tau}dx\right)^\frac{1}{\tau}\\
			&\quad
			\leq C \left( \int_{\mathbb{R}^2} ( e^{2\alpha p^* \tau(4\iota^2 + \kappa)\left( \frac{|v_n| + |u_{0,1}|}{\sqrt{4\iota^2 + \kappa}} \right)^2} - 1 ) dx \right)^{\frac{1}{\tau}} \leq C.       
		\end{align*}
		Thus, since \(|u_{0,1}|^{2}\in L^{(p^{*})'}(\mathbb{R}^{2},\mathbb{R})\), where  \((p^{*})'\) is the conjugate exponent of  \(p^{*}\), by H\"older's inequality, the right hand side of \eqref{a2.12b} belongs to $L^1(\mathbb{R}^{2},\mathbb{R})$. Moreover, up to a subsequence,  \(\{(e^{2\alpha(|v_n|+|u_{0,1}|)^{2}}-1)\left((|v_n|^2+|u_{0,1}|^2)|v_n|^2+|v_n|^{2q-4}+|u_{0,1}|^{2q-4}\right)\}\) converges weakly in \(L^{p^{*}}(\mathbb{R}^{2},\mathbb{R})\) and so, arguing as above, by the Severini--Egorov theorem we obtain that
		\begin{align*}
			&(e^{2\alpha(|v_n|+|u_{0,1}|)^{2}}-1)\left((|v_n|^2+|u_{0,1}|^2)|v_n|^2+|v_n|^{2q-4}+|u_{0,1}|^{2q-4}\right)\\
			&\quad
			\rightharpoonup  2(e^{8\alpha|u_{0,1}|^{2}}-1)(|u_{0,1}|^4+|u_{0,1}|^{2q-4})   \text{ in } L^{p^{*}}(\mathbb{R}^{2}, \mathbb{R}).
		\end{align*}
		Then, testing the previous weak convergence against \(|u_{0,1}|^2\), we obtain
		\begin{align*}
			&\int_{\mathbb R^2}(e^{2\alpha(|v_n|+|u_{0,1}|)^{2}}-1)
			\left((|v_n|^2+|u_{0,1}|^2)|v_n|^2
			+|v_n|^{2q-4}+|u_{0,1}|^{2q-4}\right)|u_{0,1}|^2\,dx\\
			&\quad\longrightarrow
			2\int_{\mathbb R^2}(e^{8\alpha|u_{0,1}|^{2}}-1)
			(|u_{0,1}|^6+|u_{0,1}|^{2q-2})\,dx.
		\end{align*}
		Hence, Lemma \ref{GDCT} yields $g_{n,1}^{\xi}\to 0$ in $L^1(\mathbb{R}^2,\mathbb{R})$ as $n\to+\infty$.
		Since $\mathfrak{G}_n^2\leq \xi^2|u_{0,1}|^2+g_{n,1}^{\xi}$ and $\xi>0$ is arbitrary, we obtain $|\mathfrak{G}_n|_2=o_n(1)$.
		
		Hence, the proof of \eqref{p3.12} is complete.
		
		Therefore, since $\{\|u_n\|_{\varepsilon_n}\}$ is bounded and, by \eqref{mo}, we also have that $\{\|u_n^1\|_{\varepsilon_n}\}$ is bounded, by a similar argument  to that used in \eqref{bddPS}, we obtain
		\begin{equation}\label{c01}
			c - J_{0,1}(U_{0,1}) + o_n(1) + o_n(1)\|u_n^1\|_{\varepsilon_n} \geq \left( \frac{1}{2} - \frac{1}{\Theta} - \frac{1}{2k} \right) \|u_n^1\|_{\varepsilon_n}^2,
		\end{equation}
		and so $c - J_{0,1}(U_{0,1}) \geq 0$.

		If $c - J_{0,1}(U_{0,1}) = 0$, by \eqref{c01} we obtain that $\|u_n^1\|_{\varepsilon_n}\to 0$ as $n\to +\infty$, and so the proof is complete.
		Otherwise, if $c - J_{0,1}(U_{0,1}) > 0$, starting from $u_n^1$, which is a $(PS)_{c-J_{0,1}(U_{0,1})}^*$ sequence with, by \eqref{J01U01}, $\displaystyle c-J_{0,1}(U_{0,1})\in \left(0,\frac{1}{8}\left(\frac{1}{2}-\frac{1}{\Theta}-\frac{1}{2k}\right)\right)$, we can repeat the above procedure
		obtaining a sequence $\{y_{n,2}\} \subset \mathbb{R}^2$ and constants $R_2, \beta_2 > 0$ such that
		\begin{equation}\label{yn2}
			\liminf_{n\rightarrow +\infty}\int_{B_{R_2}(y_{n,2})} |u_n^1|^2 dx > \beta_2.
		\end{equation}

		The sequence $\{y_{n,2}\}$ can be chosen such that $|y_{n,1}-y_{n,2}|\to +\infty$.
		Indeed, arguing by contradiction, we assume that $|y_{n,1}-y_{n,2}|$ is bounded for some subsequence. Then, by \eqref{yn2}, there exists $R'>0$ such that
		\[
		\beta_2<\int_{B_{R_2}(y_{n,2})} |u_n^1|^2 dx\leq \int_{B_{R'}(0)} |v_n-u_{0,1}\chi_{\varepsilon_n}|^2 dx
		=\int_{B_{R'}(0)} |w_n-u_{0,1}|^2 dx
		\]
		for $n\in\mathbb{N}$ large, which contradicts the convergence, up to a subsequence, of $\{w_n\}$ to $u_{0,1}$ in $L_{\rm loc}^2(\mathbb{R}^2,\mathbb{C})$. 
		
		Then, as before, considering $v_n^1:=u^1_{n}(\cdot+y_{n,2})$, we obtain that $\varepsilon_ny_{n,2}\to y_2\in \overline{\Lambda}_\Gamma$ and there is $U_{0,2}\in H^{1}(\mathbb{R}^2,\mathbb{C})\setminus\{0\}$ that solves \eqref{Ueqn} for $j=2$.
		Hence, defining \( u_n^2 := u_n^1 - \Psi_{n,2} \), where $\Psi_{n,2} := e^{iA(y_2) \cdot (\cdot - y_{n,2})}(U_{0,2} \chi_{\varepsilon_n})(\cdot - y_{n,2})=(u_{0,2} \chi_{\varepsilon_n})(\cdot - y_{n,2})$, and showing that, up to a subsequence,  $\{u_n^2\}$ is a $(PS)_{c-J_{0,1}(U_{0,1})-J_{0,2}(U_{0,2})}^*$ sequence with $c-J_{0,1}(U_{0,1})-J_{0,2}(U_{0,2})\geq 0$, we can proceed as before, considering the cases $c-J_{0,1}(U_{0,1})-J_{0,2}(U_{0,2})= 0$ (and obtaining that \[
		\|u_n^2\|_{\varepsilon_n}
		=\left\| u_n - \sum_{j=1}^{2} e^{iA(y_j)\cdot(\cdot - y_{n,j})} U_{0,j}(\cdot - y_{n,j}) \chi_{\varepsilon_n}(\cdot - y_{n,j}) \right\|_{\varepsilon_n} \to 0, \quad \text{as } n \to +\infty,
		\]
		concluding the proof) or $c-J_{0,1}(U_{0,1})-J_{0,2}(U_{0,2})>0$.
		
		This process stops in a finite number of steps since, due to the fact that each $y_j\in\overline{\Lambda}_\Gamma$, 
		\begin{equation}\label{finalsplitting}
			J_{0,j}(U_{0,j})
			=\max_{t\geq 0}J_{0,j}(tU_{0,j})\geq
			\max_{t\geq 0}I_{V(y_j)}^c(tU_{0,j})
			\geq \max_{t\geq 0}I_{V_0}^c(tU_{0,j})
			\geq c_{V_0}^c>0
		\end{equation}
		and this allows us to conclude.
	\end{proof}				
	\begin{remark}\label{eta-one}
		Under the assumptions of the previous Proposition, if we suppose, in addition, that $\varepsilon_n y_{n,j}\to y_j\in\Lambda_\Gamma$, then $\eta_j=1$ a.e. in $\mathbb R^2$.
		Indeed, 
		for every
		compact set \(Q\subset\mathbb R^2\), we have $\varepsilon_n(x+y_{n,j})\in\Lambda_\Gamma$ for every $x\in Q$ and 
		for \(n\) sufficiently large, so that $\chi_{\Lambda_\Gamma}\bigl(\varepsilon_n(x+y_{n,j})\bigr)=1$ on $Q$. Then  the weak convergence $\chi_{\Lambda_\Gamma}\bigl(\varepsilon_n(\cdot+y_{n,j})\bigr)
		\rightharpoonup\eta_j$ in $L^p_{\mathrm{loc}}(\mathbb R^2)$ for $1\leq p<+\infty$ and the arbitrariness of \(Q\) imply that $\eta_j=1$ a.e. in $\mathbb R^2$.
	\end{remark}

	\section{Multiple solutions for the modified problem}\label{Sec3}
	
	For $j=1,2,\dots,m$, set
	$\Lambda_{j,\varepsilon}:=\{x\in\mathbb{R}^2:\varepsilon x\in\Lambda_j\}$ and
	\[
	g_j(x, t) := \chi_{\Lambda_j}(x) f(t) + \left(1 - \chi_{\Lambda_j}(x)\right) \widehat{f}(t), \quad G_j(x, t) := \int_0^t g_j(x, s)ds.
	\]
	We now consider the modified problem corresponding to $\Lambda_j$, where $j = 1, 2, \dots, m$, that is,
	\begin{equation}\label{danfeng}
		\left(\frac{1}{i} \nabla - A_\varepsilon(x)\right)^2 u + V_\varepsilon(x) u = g_j(\varepsilon x, |u|^2)u \quad \text{in } \mathbb{R}^2,
	\end{equation}
	with the associated $C^1$ energy functional
	\[
	J_{\varepsilon,j}(u) := \frac{1}{2} \|u\|_{\varepsilon}^2 - \frac{1}{2} \int_{\mathbb{R}^2} G_j(\varepsilon x, |u|^2)dx, \quad u \in H_{\varepsilon}.
	\]
	Note that problem \eqref{danfeng} is a special case of problem \eqref{p3.2}, so all the properties satisfied by problem \eqref{p3.2} and its functional $J_\varepsilon$ are also satisfied by \eqref{danfeng} and $J_{\varepsilon,j}$. Then the functional $J_{\varepsilon,j}$ satisfies the Mountain Pass Geometry, and there exists a $(PS)_{c_{\varepsilon,j}}$ sequence for $J_{\varepsilon,j}$, where
	\begin{equation*}
		0 < c_{\varepsilon,j}
		:= \inf_{\gamma \in \Gamma_{\varepsilon,j}} \sup_{t \in [0,1]}
		J_{\varepsilon,j}(\gamma(t)) 
		= \inf_{u \in H_\varepsilon \setminus \{0\}}
		\sup_{t \geq 0} J_{\varepsilon,j}(tu) 
		=\inf_{\substack{u\in H_\varepsilon\setminus\{0\}\\
				|\operatorname{supp}u\cap\Lambda_{j,\varepsilon}|\neq 0}}
		\max_{t\geq0}J_{\varepsilon,j}(tu)
		= \inf_{u \in \mathcal{N}_{\varepsilon,j}} J_{\varepsilon,j}(u),
	\end{equation*}
	\[
	\mathcal{N}_{\varepsilon,j} := \left\{ u \in H_\varepsilon \setminus \{0\} : J_{\varepsilon,j}'(u)[u] = 0 \right\},\quad \Gamma_{\varepsilon,j} := \left\{ \gamma \in C([0,1], H_\varepsilon) : \gamma(0) = 0 \text{ and } J_{\varepsilon,j}(\gamma(1)) < 0 \right\}.
	\]

	The following lemma provides uniform lower bounds for suitable norms on the corresponding {\em Nehari manifolds}.
	
	\begin{lemma}\label{uniformNeharilowerbound}
		There exist constants \(\varepsilon_*>0\) and \(\kappa_*>0\) such that,
		for every \(\varepsilon\in(0,\varepsilon_*)\) and \(j\in\{1,\ldots,m\}\),
		\begin{equation*}
			\|v\|_\varepsilon\geq\kappa_*
			\qquad\text{for every }v\in\mathcal N_{\varepsilon,j}.
		\end{equation*}
		Moreover, for every \(j\in\{1,\ldots,m\}\) and $v\in H^1(\mathbb R^2,\mathbb C)\setminus\{0\}$ satisfying
		\[
		\int_{\mathbb R^2}
		\bigl(|\nabla v|^2+V_j|v|^2\bigr)\,dx
		=
		\int_{\mathbb R^2}f(|v|^2)|v|^2\,dx,
		\]
		we have
		\[
		\left(|\nabla v|_2^2+V_j|v|_2^2\right)^{1/2}\geq\kappa_*.
		\]
	\end{lemma}
	\begin{proof}
		Fix \(q>2\) and \(\alpha>4\pi\). By
		\(0\leq g_j(x,t)\leq f(t)\), \eqref{p2.4}, H\"older's
		inequality, \eqref{e}, the Sobolev embedding, the diamagnetic
		inequality, and Lemma \ref{2.1}, there exist constants
		\(\delta>0\) and \(C>0\), independent of sufficiently small
		\(\varepsilon>0\) and of \(j\in\{1,\ldots,m\}\), such that, whenever \(\|v\|_\varepsilon\leq\delta\),
		\[
		\int_{\mathbb R^2}
		g_j(\varepsilon x,|v|^2)|v|^2\,dx
		\leq
		\frac12\|v\|_\varepsilon^2+C\|v\|_\varepsilon^q.
		\]
		Then, if \(v\in\mathcal N_{\varepsilon,j}\) and \(\|v\|_\varepsilon\leq\delta\),
		\[
		\|v\|_\varepsilon^2
		=
		\int_{\mathbb R^2}
		g_j(\varepsilon x,|v|^2)|v|^2\,dx
		\leq \frac12\|v\|_\varepsilon^2+C\|v\|_\varepsilon^q
		\]
		and so
		\[
		\|v\|_\varepsilon
		\geq
		\min\left\{\delta,(2C)^{-1/(q-2)}\right\}.
		\]
		The same argument, with \(g_j(\varepsilon x,\cdot)\) replaced by \(f\) and \(\|\cdot\|_\varepsilon\) replaced by
		\(\left(|\nabla(\cdot)|_2^2+V_j|\,\cdot\,|_2^2\right)^{1/2}\),
		applies to the second part. Observing that, since \(\{1,\ldots,m\}\) is
		finite, the constants may be chosen uniformly with respect to
		\(j\in\{1,\ldots,m\}\), we can conclude. 
	\end{proof}

	Let us continue with the following preliminary result.
	\begin{lemma}\label{lempsiej}
		For every $j = 1,2,\dots,m$, let
		\begin{equation}\label{psi}
			\psi_{\varepsilon,j} := e^{iA(x_j)\cdot(\cdot - x_j/\varepsilon)}(w_j \chi_\varepsilon)(\cdot - x_j/\varepsilon),
		\end{equation}
		where
		$x_j \in K_j$, $\chi_\varepsilon$ is as in Proposition \ref{fengjie}, and $w_j:=w_{V_j} \in H^1(\mathbb{R}^2,\mathbb{R})$ is a radially symmetric ground state solution of problem  \eqref{fude} with $\mu=V_j$. 
		Then, for sufficiently small $\varepsilon > 0$, 
		\begin{equation}\label{supp}
			\operatorname{supp}\psi_{\varepsilon,j} \subset \Lambda_{j,\varepsilon},
		\end{equation}	
		so that
		\begin{equation}\label{empintj1j2e}
			\Lambda_{j_1,\varepsilon}\cap\Lambda_{j_2,\varepsilon}=\emptyset \text{ for every } j_1,j_2\in \Gamma \text{ with } j_1\neq j_2,
		\end{equation}
		and
		\begin{equation}\label{V_j}
			\lim_{\varepsilon\to 0^+}\|\psi_{\varepsilon,j}\|^2_{\varepsilon}
			= |\nabla w_j|_{2}^{2}+V_j|w_j|_2^2.
		\end{equation}
	\end{lemma}
	\begin{proof}
		Assumption \eqref{V2} ensures that, for all $j=1,2,\dots,m$, $\operatorname{dist}(x_j, \partial \Lambda_j)  > 0$
		and this implies that there exists $r_j>0$ such that $B_{r_j}(x_j)\subset\Lambda_j$. Hence, for sufficiently small $\varepsilon > 0$, we get \eqref{supp}.
		
		Then, \eqref{empintj1j2e} follows from (\ref{V2}).
		
		Let now $\varepsilon_n\to 0^+$ as $n\to +\infty$ be arbitrary.
		Since, for each $j=1,2,...,m$,
		\begin{equation*}
			\begin{split}					 \|\psi_{\varepsilon_n,j}\|^2_{\varepsilon_n}
				&= \int_{\mathbb{R}^2} |\chi_{\varepsilon_n} \nabla w_j|^2 dx +  \int_{\mathbb{R}^2}\left( |w_j \nabla \chi_{\varepsilon_n}|^2  + 2   \chi_{\varepsilon_n} w_j \nabla w_j \cdot \nabla \chi_{\varepsilon_n}\right) dx \\
				&\quad +\int_{\mathbb{R}^2}  |A(x_j) - A(\varepsilon_nx+x_j)|^2 |w_j \chi_{\varepsilon_n}|^2  dx +\int_{\mathbb{R}^2} V(\varepsilon_nx+x_j)|w_j\chi_{\varepsilon_n}|^2  dx.
			\end{split}
		\end{equation*}
		Then, by a similar argument used above  to obtain \eqref{mo}, we have that $\psi_{\varepsilon_n,j}\in H_{\varepsilon_n}$ and
		\begin{equation*}
			\|\psi_{\varepsilon_n,j}\|^2_{\varepsilon_n}
			= |\nabla w_j|_{2}^{2}+V_j|w_j|_2^2 + o_n(1),
		\end{equation*}
		completing the proof.
	\end{proof}
	
	Now we show the asymptotic behavior of the energy levels $c_{\varepsilon,j}$.

	\begin{lemma}\label{4.1}
		For each $j=1,2,...,m$, we have
		$c_{\varepsilon,j}\rightarrow c_{V_j}$ as $ \varepsilon\rightarrow 0^+$.
	\end{lemma}
	\begin{proof}
		Fix \(j\in\{1,\ldots,m\}\) and let
		\(\varepsilon_n\to0^+\) be arbitrary.  We first prove the upper bound.  Let
		\(\psi_{\varepsilon_n,j}\) be given by \eqref{psi}.
		Since, by definition, $|\operatorname{supp}\psi_{\varepsilon_n,j}|\neq 0$, then, for $n$ large enough, by \eqref{supp}, $|\operatorname{supp}\psi_{\varepsilon_n,j}\cap \Lambda_{j,\varepsilon_n}|\neq 0$, and so,
		arguing as in the first part of the proof of Lemma \ref{lem_same_levels}, there exists \(t_{n,j}>0\) such that
		\(t_{n,j}\psi_{\varepsilon_n,j}\in\mathcal N_{\varepsilon_n,j}\), namely
		\begin{equation}\label{N_ej}
			\|\psi_{\varepsilon_n,j}\|_{\varepsilon_n}^2
			=\int_{\mathbb R^2}
			f(t_{n,j}^2|w_j\chi_{\varepsilon_n}|^2)
			|w_j\chi_{\varepsilon_n}|^2dx.
		\end{equation}
		The sequence \(\{t_{n,j}\}\) is bounded and bounded away from zero.  Indeed,
		if \(t_{n,j}\to+\infty\), then, since
		\(\min_{\overline{B_1(0)}}w_j>0\), by (\ref{f4}), for $n$ large enough
		\[
		\int_{\mathbb R^2}
		f(t_{n,j}^2|w_j\chi_{\varepsilon_n}|^2)
		|w_j\chi_{\varepsilon_n}|^2dx
		\geq C t_{n,j}^{p-2}\int_{B_1(0)} |w_j|^p dx
		\geq C t_{n,j}^{p-2}
		\]
		so that, passing to the limit as $n\to +\infty$ and using \eqref{N_ej}, we get that $\|\psi_{\varepsilon_n,j}\|_{\varepsilon_n}^2\to +\infty$, which contradicts
		\eqref{V_j}.  If
		\(t_{n,j}\to0^+\), then \eqref{p2.4}, H\"older's inequality, and Lemma
		\ref{2.1} show that the right-hand side tends to zero, again contradicting
		\eqref{V_j}.  Hence, after passing to a subsequence,
		\(t_{n,j}\to t_j>0\).
		
		Thus, by \eqref{V_j}, the Dominated Convergence Theorem, and \eqref{N_ej},
		\[
		\int_{\mathbb R^2}(|\nabla w_j|^2+V_jw_j^2)dx
		=\int_{\mathbb R^2}f(t_j^2w_j^2)w_j^2dx.
		\]
		Since \(w_j\in\mathcal N_{V_j}\) and \(f\) is strictly increasing on
		\((0,+\infty)\) by \eqref{f4}, it follows that \(t_j=1\). 
		
		Moreover,
		\eqref{p2.5}, H\"older's inequality, and Lemma \ref{2.1} provide an
		integrable dominating function on \(\mathbb R^2\) for
		\(F(t_{n,j}^2|w_j\chi_{\varepsilon_n}|^2)\) independent of \(n\).  Therefore, using again the Dominated Convergence Theorem,
		\[
		J_{\varepsilon_n,j}(t_{n,j}\psi_{\varepsilon_n,j})
		\longrightarrow I_{V_j}(w_j)=c_{V_j},
		\]
		and hence
		\begin{equation}\label{locallevelupper}
			\limsup_{n\to+\infty}c_{\varepsilon_n,j}
			= \limsup_{n\to+\infty} \inf_{u \in \mathcal{N}_{\varepsilon_n,j}} J_{\varepsilon_n,j}(u)
			\leq \limsup_{n\to+\infty} J_{\varepsilon_n,j}(t_{n,j}\psi_{\varepsilon_n,j})
			=c_{V_j}.
		\end{equation}
		
		For \(n\) sufficiently large, by \eqref{locallevelupper} and \eqref{p3.1}, using Lemma \ref{3.1} for \(\Gamma=\{j\}\), we find
		\(u_{n,j}\in H_{\varepsilon_n}\setminus\{0\}\) such that
		\[
		J_{\varepsilon_n,j}(u_{n,j})=c_{\varepsilon_n,j},
		\qquad J_{\varepsilon_n,j}'(u_{n,j})=0.
		\]
		Then, Lemma \ref{uniformNeharilowerbound} gives that $\|u_{n,j}\|_{\varepsilon_n}\geq\kappa_*$ for $n$ sufficiently large and so, arguing as in \eqref{bddPS} and since
		\(J_{\varepsilon_n,j}'(u_{n,j})=0\),
		\[
		c_{\varepsilon_n,j}=J_{\varepsilon_n,j}(u_{n,j})
		-\frac1{\Theta} J_{\varepsilon_n,j}'(u_{n,j})[u_{n,j}]
		\geq\left(\frac12-\frac1{\Theta}-\frac1{2k}\right)
		\|u_{n,j}\|_{\varepsilon_n}^2
		\geq\left(\frac12-\frac1{\Theta}-\frac1{2k}\right)\kappa_*^2>0.
		\]
		Thus, by \eqref{locallevelupper}, we obtain that
		\(\{u_{n,j}\}\) is bounded in \(H_{\varepsilon_n}\).  
		
		Let
		\begin{equation}\label{dsingle}
			d:=\liminf\limits_{n\to+\infty}
			c_{\varepsilon_n,j}>0
		\end{equation}
		and consider a subsequence of \(\{u_{n,j}\}\), still denoted by \(\{u_{n,j}\}\), such that $c_{\varepsilon_n,j}= J_{\varepsilon_n,j}(u_{n,j})\to d$.
		By \eqref{locallevelupper} and
		\eqref{p3.1},
		\begin{equation}\label{uppbcVj}
			0<d\leq c_{V_j}<\frac{1}{8m
			}\left(\frac12-\frac1{\Theta}-\frac1{2k}\right)
		\end{equation}
		and, since \(J_{\varepsilon_n,j}'(u_{n,j})=0\), it follows that
		\(\{u_{n,j}\}\) is a \((PS)^*_d\) sequence for $J_{\varepsilon_n,j}$.  
		Applying Proposition \ref{fengjie} to \(\Gamma=\{j\}\), so that the penalized nonlinearity \(g\) and the functional
		\(J_\varepsilon\)  are  \(g_j\) and
		\(J_{\varepsilon,j}\), respectively, there exists a subsequence still denoted by \(\{u_{n,j}\}\),
		a positive integer \(l\),
		mutually diverging sequences \(\{y_{n,h}\}_{h=1}^l\subset\mathbb R^2\), $l$
		points \(y_h\in\overline\Lambda_j\) that are limits of $\{\varepsilon_n y_{n,h}\}$, and $l$ nontrivial functions
		\(U_{0,h}\in H^1(\mathbb R^2,\mathbb C)\) such that
		\[
		-\Delta U_{0,h}+V(y_h)U_{0,h}
		=g_{0,h}(x,|U_{0,h}|^2)U_{0,h}
		\quad
		\text{in }
		\mathbb R^2,
		\text{ for every }
		h\in\{1,\ldots,l\},
		\]
		\[
		\left\|u_{n,j}-\sum_{h=1}^l
		e^{iA(y_h)\cdot(\cdot-y_{n,h})}
		(U_{0,h}\chi_{\varepsilon_n})(\cdot-y_{n,h})
		\right\|_{\varepsilon_n}\longrightarrow0,
		\qquad
		d=\sum_{h=1}^lJ_{0,h}(U_{0,h}),
		\]
		\[
		J_{0,h}(U_{0,h})\geq c^c_{V(y_h)}>0,
		\qquad h=1,\ldots,l.
		\]
		where 
		\[
		J_{0,h}(v):=\frac12\int_{\mathbb R^2}
		\bigl(|\nabla v|^2+V(y_h)|v|^2\bigr)dx
		-\frac12\int_{\mathbb R^2}G_{0,h}(x,|v|^2)dx, 
		\]
		\[
		g_{0,h}(x,t):=\eta_h(x)f(t)+(1-\eta_h(x))\widehat f(t),
		\quad
		G_{0,h}(x,t):=\int_0^t g_{0,h}(x,s)\,ds,
		\]
		\[
		\chi_{\Lambda_j}(\varepsilon_n(\cdot+y_{n,h}))
		\rightharpoonup\eta_h
		\quad\hbox{in }L^p_{\mathrm{loc}}(\mathbb{R}^2)
		\quad\hbox{for every }1\leq p<+\infty,
		\qquad 0\leq\eta_h\leq1\quad\hbox{a.e. in }\mathbb{R}^2.
		\]
		Since \(y_h\in\overline\Lambda_j\), we have \(V(y_h)\geq V_j\).
		Hence, by the minimax characterization \eqref{jitaic} and Lemma
		\ref{existenceC},
		\[
		J_{0,h}(U_{0,h})
		\geq c^c_{V(y_h)}
		\geq c^c_{V_j}
		=c_{V_j}
		\]
		and hence \(d\geq c_{V_j}\). Combining this inequality with \eqref{uppbcVj}
		we obtain \(l=1\) and \(d=c_{V_j}\).
		Then, by \eqref{dsingle}, 
		\(\liminf_{n\to+\infty}c_{\varepsilon_n,j}=c_{V_j}\) and so, by \eqref{locallevelupper}, we obtain that
		\(c_{\varepsilon_n,j}\to c_{V_j}\). Hence, the arbitrariness of the sequence
		\(\varepsilon_n\to0^+\) allows us to conclude.
		
	\end{proof}

	We now introduce a smooth localization of the different potential wells.
	
	\begin{lemma}\label{smoothlocalization}
		There exist real-valued functions
		$\zeta_0,\zeta_1,\ldots,\zeta_s\in C^\infty(\mathbb R^2,[0,1])$ such that for every $r=0,1,\ldots,s$, $|\nabla\zeta_r|\in L^\infty(\mathbb{R}^2,\mathbb{R})$, for every $j\in\Gamma$, $\zeta_j=1$ in a neighborhood of $\overline\Lambda_j$, $\zeta_0=0$ in a neighborhood of $\overline\Lambda_\Gamma$,
		\[
		\sum_{r=0}^s\zeta_r^2=1,\qquad
		\operatorname{supp}\zeta_i\cap\operatorname{supp}\zeta_j=\emptyset
		\quad(i\ne j),\qquad i,j\in\Gamma.
		\]
		Moreover, there exists $C_\zeta>0$, such that, for every $\varepsilon>0$ small enough and $u\in H_\varepsilon$, if $\zeta_{r,\varepsilon}:=\zeta_r(\varepsilon \cdot)$,
		\begin{equation}\label{multiplierbound}
			\|\zeta_{r,\varepsilon}u\|_\varepsilon\leq C_\zeta\|u\|_\varepsilon,
			\qquad r=0,1,\ldots,s
		\end{equation}
		and
		\begin{equation}\label{magneticIMS}
			\sum_{r=0}^s\|\zeta_{r,\varepsilon}u\|_\varepsilon^2
			=\|u\|_\varepsilon^2+
			\int_{\mathbb R^2}\sum_{r=0}^s|\nabla\zeta_{r,\varepsilon}|^2|u|^2dx.
		\end{equation}
	\end{lemma}
	\begin{proof}
		Since the compact sets $\overline\Lambda_j$, $j\in\Gamma$, are mutually
		disjoint, choose $\varrho_j\in C_c^\infty(\mathbb R^2,[0,\pi/2])$ such that
		$\varrho_j=\pi/2$ in a neighborhood of $\overline\Lambda_j$ and the supports
		of $\varrho_1,\ldots,\varrho_s$ are mutually disjoint. 
		Define
		\(\zeta_j:=\sin\varrho_j\) for \(j\in\Gamma\), and
		\(\zeta_0:=\cos(\sum_{j=1}^s\varrho_j)\).
		If $\varrho_j(x)\ne0$, then $\varrho_i(x)=0$ for $i\ne j$, and hence
		$\zeta_0(x)^2+\zeta_j(x)^2=1$.  Outside all these supports,
		$\zeta_0=1$.  This proves the asserted properties of the partition.
		By construction, each \(\zeta_j\), \(j\in\Gamma\), has compact support,
		while \(\zeta_0=1\) outside a compact set. Hence all their gradients are
		bounded. 
		
		For a real-valued smooth function $\zeta$, the product rule gives
		\(\nabla_{A_\varepsilon}(\zeta u)=\zeta\nabla_{A_\varepsilon}u
		+\frac1i u\nabla\zeta\).
		Since $|\nabla\zeta_{r,\varepsilon}|\leq |\nabla\zeta_r|_\infty\varepsilon$ and, by
		\eqref{V1}, $|u|_2\leq V_0^{-1/2}\|u\|_\varepsilon$, we obtain
		\[
		\|\zeta_{r,\varepsilon}u\|_\varepsilon^2
		\leq
		2\left(1+
		\frac{\varepsilon^2|\nabla\zeta_r|_\infty^2}{V_0}\right)\|u\|_\varepsilon^2,
		\]
		which proves \eqref{multiplierbound}.

		Finally, since, for each $r=0,1,\ldots,s$, the
		product rule gives, a.e. in $\mathbb R^2$,
		\[
		|\nabla_{A_\varepsilon}(\zeta_{r,\varepsilon}u)|^2
		=\zeta_{r,\varepsilon}^2|\nabla_{A_\varepsilon}u|^2
		+|\nabla\zeta_{r,\varepsilon}|^2|u|^2+2\operatorname{Re}\left(
		\zeta_{r,\varepsilon}\nabla_{A_\varepsilon}u\cdot
		\overline{\frac1i u\nabla\zeta_{r,\varepsilon}}\right)
		\]
		and 
		\[
		\sum_{r=0}^s\zeta_{r,\varepsilon}^2=1 
		\implies
		\sum_{r=0}^s\zeta_{r,\varepsilon}
		\nabla\zeta_{r,\varepsilon}
		=\frac12\nabla\left(\sum_{r=0}^s
		\zeta_{r,\varepsilon}^2\right)=0
		\]
		then
		\begin{equation*}
			2\operatorname{Re}\sum_{r=0}^s
			\left(\zeta_{r,\varepsilon}\nabla_{A_\varepsilon}u\cdot
			\overline{\frac1i u\nabla\zeta_{r,\varepsilon}}\right)
			=2\operatorname{Re}\left(\nabla_{A_\varepsilon}u\cdot
			\overline{\frac1i u\sum_{r=0}^s\zeta_{r,\varepsilon}\nabla\zeta_{r,\varepsilon}}\right)=0,    
		\end{equation*}
		and hence
		\[
		\sum_{r=0}^s|\nabla_{A_\varepsilon}
		(\zeta_{r,\varepsilon}u)|^2
		=|\nabla_{A_\varepsilon}u|^2
		+\sum_{r=0}^s|\nabla\zeta_{r,\varepsilon}|^2|u|^2.
		\]
		Moreover,
		\[
		\sum_{r=0}^sV_\varepsilon(x)|\zeta_{r,\varepsilon}u|^2
		=V_\varepsilon(x)|u|^2\sum_{r=0}^s\zeta_{r,\varepsilon}^2
		=V_\varepsilon(x)|u|^2
		\]
		and so, integrating the last two identities over $\mathbb R^2$, we obtain
		\eqref{magneticIMS}.
	\end{proof}
	
	We now establish a global energy decomposition associated with
	the smooth localization above.
	\begin{lemma}\label{globalenergysplitting}
		There exists $C>0$,  such that for every $\varepsilon>0$ small enough and $u\in H_\varepsilon$
		\begin{equation}\label{energydecomp}
			J_\varepsilon(u)\geq
			\sum_{j=1}^sJ_{\varepsilon,j}(\zeta_{j,\varepsilon}u)
			+\frac12\int_{\mathbb R^2}|\nabla_{A_\varepsilon}(\zeta_{0,\varepsilon}u)|^2dx+\frac{k-1}{2k}\int_{\mathbb R^2}V_\varepsilon(x)|\zeta_{0,\varepsilon}u|^2dx
			-C\varepsilon^2\|u\|_\varepsilon^2.
		\end{equation}
		If, in addition,
		$\zeta_{j,\varepsilon}u\in\mathcal N_{\varepsilon,j}$ for all
		$j\in\Gamma$, then, after increasing $C$, for all sufficiently small $\varepsilon>0$,
		\begin{equation}\label{energydecompN}
			J_\varepsilon(u)\geq
			(1-C\varepsilon^2)\sum_{j=1}^s
			J_{\varepsilon,j}(\zeta_{j,\varepsilon}u).
		\end{equation}
	\end{lemma}
	\begin{proof}
		We first compare the nonlinear terms pointwise.  If
		$\varepsilon x\in\Lambda_j$ for some $j\in\Gamma$, then
		$\zeta_{j,\varepsilon}(x)=1$, whereas
		$\zeta_{i,\varepsilon}(x)=0$ for $i\ne j$ and
		$\zeta_{0,\varepsilon}(x)=0$.  By the definitions of $G$ and $G_j$,
		\[
		G(\varepsilon x,|u|^2)=F(|u|^2)
		=G_j(\varepsilon x,|\zeta_{j,\varepsilon}u|^2)
		=\sum_{i=1}^sG_i(\varepsilon x,|\zeta_{i,\varepsilon}u|^2).
		\]
		Suppose now that $\varepsilon x\notin\Lambda_\Gamma$.  Then
		$G(\varepsilon x,t)=G_i(\varepsilon x,t)=\widehat F(t)$ for every
		$i=1,\ldots,s$ and $t\geq0$.  Moreover, at most one of
		$\zeta_{h,\varepsilon}$'s ($h\in\Gamma$) is nonzero.
		If the nonzero member is $\zeta_{j,\varepsilon}$, then
		$\zeta_{0,\varepsilon}^2(x)=1-\zeta_{j,\varepsilon}^2(x)$, and
		\begin{align*}
			G(\varepsilon x,|u(x)
			|^2)-\sum_{i=1}^sG_i(\varepsilon x,|\zeta_{i,\varepsilon}(x)
			u(x)
			|^2)
			&=\int_{\zeta_{j,\varepsilon}^2(x)
				|u(x)
				|^2}^{|u(x)
				|^2}\widehat f(t)dt\leq\frac{V_0}{k}(1-\zeta_{j,\varepsilon}^2(x)
			)|u(x)
			|^2\\
			&
			\leq\frac{V_\varepsilon(x)}{k}|\zeta_{0,\varepsilon}(x)
			u(x)
			|^2.
		\end{align*}
		If, instead, all the functions $\zeta_{h,\varepsilon}$
		vanish, then $\zeta_{0,\varepsilon}(x)=1$ and
		\[
		G(\varepsilon x,|u(x)
		|^2)
		=\widehat F(|u(x)
		|^2)=\int_0^{|u(x)
			|^2}\widehat f(t)dt
		\leq\frac{V_0}{k}|u(x)
		|^2
		\leq\frac{V_\varepsilon(x)}{k}|\zeta_{0,\varepsilon}(x)
		u(x)
		|^2.
		\]
		Combining these cases, for a.e. $x\in\mathbb R^2$ we obtain
		\begin{equation}\label{nonlinearpartition}
			G(\varepsilon x,|u|^2)\leq
			\sum_{j=1}^sG_j(\varepsilon x,|\zeta_{j,\varepsilon}u|^2)
			+\frac{V_\varepsilon(x)}{k}|\zeta_{0,\varepsilon}u|^2.
		\end{equation}
		Then, by \eqref{magneticIMS} and  \eqref{nonlinearpartition}, we obtain
		\[
		\begin{aligned}
			J_\varepsilon(u)
			\geq{}&\frac12\sum_{r=0}^s
			\|\zeta_{r,\varepsilon}u\|_\varepsilon^2
			-\frac12\sum_{j=1}^s\int_{\mathbb R^2}
			G_j(\varepsilon x,|\zeta_{j,\varepsilon}u|^2)dx-\frac1{2k}\int_{\mathbb R^2}V_\varepsilon(x)
			|\zeta_{0,\varepsilon}u|^2dx
			-\frac12\int_{\mathbb R^2}\sum_{r=0}^s
			|\nabla\zeta_{r,\varepsilon}|^2|u|^2dx\\
			={}&\sum_{j=1}^sJ_{\varepsilon,j}(\zeta_{j,\varepsilon}u)
			+\frac12\int_{\mathbb R^2}
			|\nabla_{A_\varepsilon}(\zeta_{0,\varepsilon}u)|^2dx+\frac{k-1}{2k}\int_{\mathbb R^2}V_\varepsilon(x)
			|\zeta_{0,\varepsilon}u|^2dx
			-\frac12\int_{\mathbb R^2}\sum_{r=0}^s
			|\nabla\zeta_{r,\varepsilon}|^2|u|^2dx.
		\end{aligned}
		\]
		Since
		$|\nabla\zeta_{r,\varepsilon}(x)|\leq
		\varepsilon|\nabla\zeta_r|_\infty$ and $V_\varepsilon\geq V_0$, it follows
		that
		\[
		\int_{\mathbb R^2}\sum_{r=0}^s
		|\nabla\zeta_{r,\varepsilon}|^2|u|^2dx
		\leq\varepsilon^2\sum_{r=0}^s|\nabla\zeta_r|_\infty^2|u|_2^2
		\leq\frac{\varepsilon^2}{V_0}\left(\sum_{r=0}^s|\nabla\zeta_r|_\infty^2\right)\|u\|_\varepsilon^2.
		\]
		Thus, combining the last two estimates, we get
		\eqref{energydecomp}.
		
		Let now $\zeta_{j,\varepsilon}u\in\mathcal N_{\varepsilon,j}$ for all $j\in\Gamma$.   Arguing as in \eqref{bddPS},
		\[
		J_{\varepsilon,j}(\zeta_{j,\varepsilon}u)
		=J_{\varepsilon,j}(\zeta_{j,\varepsilon}u)-\frac1{\Theta} J_{\varepsilon,j}'(\zeta_{j,\varepsilon}u)[\zeta_{j,\varepsilon}u]
		\geq\left(\frac12-\frac1{\Theta}-\frac1{2k}\right)\|\zeta_{j,\varepsilon}u\|_\varepsilon^2
		\]
		and so, by \eqref{p3.1},
		\[
		\|\zeta_{j,\varepsilon}u\|_\varepsilon^2
		\leq \left(\frac12-\frac1{\Theta}-\frac1{2k}\right)^{-1}J_{\varepsilon,j}(\zeta_{j,\varepsilon}u)
		\] for $j=1,\ldots,s$.
		Moreover, \eqref{magneticIMS} yields
		\[
		\|u\|_\varepsilon^2\leq
		\sum_{j=1}^s\|\zeta_{j,\varepsilon}u\|_\varepsilon^2
		+\|\zeta_{0,\varepsilon}u\|_\varepsilon^2.\]
		In addition, since $(k-1)/(2k)<1/2$,
		\[
		\frac12\int_{\mathbb R^2}|\nabla_{A_\varepsilon}(\zeta_{0,\varepsilon}u)|^2dx
		+\frac{k-1}{2k}\int_{\mathbb R^2}V_\varepsilon(x)|\zeta_{0,\varepsilon}u|^2dx
		\geq\frac{k-1}{2k}\|\zeta_{0,\varepsilon}u\|_\varepsilon^2.
		\]
		Thus, using the last three estimates in \eqref{energydecomp}, we obtain
		\begin{align*}
			J_\varepsilon(u)&\geq
			\sum_{j=1}^sJ_{\varepsilon,j}(\zeta_{j,\varepsilon}u)
			+\frac{k-1}{2k}\|\zeta_{0,\varepsilon}u\|_\varepsilon^2-C\varepsilon^2\left(
			\sum_{j=1}^s\|\zeta_{j,\varepsilon}u\|_\varepsilon^2
			+\|\zeta_{0,\varepsilon}u\|_\varepsilon^2\right)\\
			&\geq
			\left[1-C\varepsilon^2\left(
			\frac12-\frac1{\Theta}-\frac1{2k}\right)^{-1}\right]
			\sum_{j=1}^sJ_{\varepsilon,j}(\zeta_{j,\varepsilon}u)+\left(\frac{k-1}{2k}-C\varepsilon^2\right)
			\|\zeta_{0,\varepsilon}u\|_\varepsilon^2
		\end{align*}
		and so we can conclude.
	\end{proof}
	
	Due to the geometry of $t\mapsto I_{V_j}(tw_j)$, we may choose $T>1$ so
	large that, for every $j=1,\ldots,m$,
	\begin{equation*}
		0<I_{V_j}(T^{-1}w_j)<\frac{c_{V_j}}2,\qquad
		I_{V_j}'(T^{-1}w_j)[w_j]>0,\qquad
		I_{V_j}'(Tw_j)[w_j]<0,\qquad 
		I_{V_j}(Tw_j)<0.
	\end{equation*}
	Observe that
	\begin{equation}\label{carcVj}
		c_{V_j}=\max_{\theta_j\in[T^{-2},1]}I_{V_j}(\theta_jTw_j),
	\end{equation}
	and the maximum is attained only at $\theta_j=1/T$.
	
	For $\theta=(\theta_1,\ldots,\theta_s)\in[T^{-2},1]^s$, set
	\[
	\widetilde E_\varepsilon(\theta):=
	\sum_{j=1}^s\theta_jT\psi_{\varepsilon,j}
	\]
	and consider the class
	\begin{equation}\label{pathclass}
		\Sigma_\varepsilon
		:=\left\{E\in C([T^{-2},1]^s,H_\varepsilon):
		\begin{array}{l}
			E=\widetilde E_\varepsilon
			\text{ on }\partial([T^{-2},1]^s),\\
			\|\zeta_{j,\varepsilon}E(\theta)\|_\varepsilon>0
			\text{ for every }j\in\Gamma
			\text{ and }\theta\in[T^{-2},1]^s
		\end{array}
		\right\}.
	\end{equation}
	By \eqref{supp},
	$\zeta_{j,\varepsilon}\psi_{\varepsilon,j}=\psi_{\varepsilon,j}$ for small
	$\varepsilon$, and hence $\widetilde E_\varepsilon\in\Sigma_\varepsilon$.
	
	The next lemma shows that every admissible map
	$E\in\Sigma_\varepsilon$ admits a parameter for which all its localized
	components belong simultaneously to the corresponding Nehari manifolds.
	
	\begin{lemma}\label{4.2}
		For $\varepsilon>0$ sufficiently small and every $E\in\Sigma_\varepsilon$,
		there exists $\theta^*\in[T^{-2},1]^s$ such that, for every
		\(j\in\Gamma\),
		\(\zeta_{j,\varepsilon}E(\theta^*)\in\mathcal N_{\varepsilon,j}\).
	\end{lemma}
	\begin{proof}
		For $\varepsilon>0$, $E\in\Sigma_\varepsilon$, $j\in\Gamma$, and $\theta\in[T^{-2},1]^s$, define
		\[
		\Phi_{\varepsilon,j}(\theta):=J_{\varepsilon,j}'(\zeta_{j,\varepsilon}E(\theta))
		[\zeta_{j,\varepsilon}E(\theta)],\qquad
		\text{and}
		\qquad
		\Phi_\varepsilon:=(\Phi_{\varepsilon,1},\ldots,\Phi_{\varepsilon,s}).
		\]
		The map $\Phi_\varepsilon:\theta\in [T^{-2},1]^s\to\mathbb{R}^s$ is continuous. If $\theta_j=T^{-2}$, then
		$\theta\in\partial([T^{-2},1]^s)$ and hence
		$E(\theta)=\widetilde E_\varepsilon(\theta)$. Moreover, for $\varepsilon$ small enough, due to \eqref{supp} and to the properties of $\zeta_{j,\varepsilon}$ in
		Lemma \ref{smoothlocalization},
		\[
		\zeta_{j,\varepsilon}\psi_{\varepsilon,j}=\psi_{\varepsilon,j},
		\qquad
		\zeta_{j,\varepsilon}\psi_{\varepsilon,i}=0\quad (i\ne j).
		\]
		Therefore
		\[
		\zeta_{j,\varepsilon}E(\theta)
		=\sum_{i=1}^s\theta_iT\zeta_{j,\varepsilon}\psi_{\varepsilon,i}
		=\theta_jT\psi_{\varepsilon,j}=T^{-1}\psi_{\varepsilon,j}.
		\]
		Analogously, if $\theta_j=1$, \(
		\zeta_{j,\varepsilon}E(\theta)=T\psi_{\varepsilon,j}
		\).
		
		Moreover, for \(a\in\{T^{-1},T\}\), observed that
		\[
		\int_{\mathbb{R}^2} f(a^2 |\psi_{\varepsilon,j}|^2)a^2 |\psi_{\varepsilon,j}|^2 dx
		= \int_{\mathbb{R}^2} f(a^2 |w_j\chi_\varepsilon|^2)a^2 |w_j\chi_\varepsilon|^2 dx,
		\]
		since
		\(0\leq\chi_\varepsilon\leq1\), $\chi_\varepsilon\to 1$ a.e. in $\mathbb{R}^2$, and \(f\) is nondecreasing, 
		\[
		f\bigl(a^2|w_j\chi_\varepsilon|^2\bigr)
		a^2|w_j\chi_\varepsilon|^2
		\leq f(T^2w_j^2)T^2w_j^2,
		\]
		with the left hand side that tends to $f\bigl(a^2w_j^2\bigr)a^2w_j^2$ a.e. in $\mathbb{R}^2$ and  the right hand side that belongs to \(L^1(\mathbb R^2)\) by
		\eqref{p2.4}, H\"older's inequality, and Lemma \ref{2.1}.
		Then, the Dominated Convergence Theorem and \eqref{V_j} give
		\[
		\Phi_{\varepsilon,j}(\theta)=
		\begin{cases}
			T^{-1}I_{V_j}'(T^{-1}w_j)[w_j]+o_\varepsilon(1)>0 & \text{for }\theta_j=T^{-2},\\
			TI_{V_j}'(Tw_j)[w_j]+o_\varepsilon(1)<0 & \text{for }\theta_j=1.
		\end{cases}
		\]
		Hence, by the Poincar\'e--Miranda theorem,
		there exists $\theta^*\in[T^{-2},1]^s$ such that
		$\Phi_\varepsilon(\theta^*)=0$. Therefore, since the second
		condition in \eqref{pathclass} gives
		$\zeta_{j,\varepsilon}E(\theta^*)\ne0$, we conclude.
	\end{proof}
	
	Define
	\[
	S_\varepsilon:=\inf_{E\in\Sigma_\varepsilon}
	\max_{\theta\in[T^{-2},1]^s}J_\varepsilon(E(\theta)).
	\]
	We have
	\begin{proposition}\label{4.3}
		\[
		\lim_{\varepsilon\to0^+}S_\varepsilon=\sum_{j=1}^sc_{V_j}.
		\]
	\end{proposition}
	\begin{proof}
		Let $E\in\Sigma_\varepsilon$.  By Lemma \ref{4.2}, for $\varepsilon>0$ small enough, there exists
		$\theta^*\in[T^{-2},1]^s$ such that, for every $j\in\Gamma$, $\zeta_{j,\varepsilon}E(\theta^*)\in\mathcal N_{\varepsilon,j}$.  Hence \eqref{energydecompN} and the definition of $c_{\varepsilon,j}$ give
		\[
		\max_{\theta\in[T^{-2},1]^s}J_\varepsilon(E(\theta))
		\geq J_\varepsilon(E(\theta^*))
		\geq(1-C\varepsilon^2)\sum_{j=1}^sJ_{\varepsilon,j}(\zeta_{j,\varepsilon}E(\theta^*))
		\geq(1-C\varepsilon^2)\sum_{j=1}^sc_{\varepsilon,j}.
		\]
		Then, taking the infimum over
		$E\in\Sigma_\varepsilon$  we obtain
		\[
		S_\varepsilon\geq(1-C\varepsilon^2)
		\sum_{j=1}^sc_{\varepsilon,j}.
		\]
		Finally, by Lemma \ref{4.1}, it follows that
		\[
		\liminf_{\varepsilon\to0^+}S_\varepsilon
		\geq\sum_{j=1}^sc_{V_j}.
		\]
		
		On the other hand, \eqref{supp} and \eqref{empintj1j2e} give
		\begin{equation}\label{Eesum}
			J_\varepsilon(\widetilde E_\varepsilon(\theta))
			=\sum_{j=1}^sJ_{\varepsilon,j}(\theta_jT\psi_{\varepsilon,j})
		\end{equation}
		and
		\[
		\int_{\mathbb R^2}G_j(\varepsilon x,|T\theta_j\psi_{\varepsilon,j}|^2)dx
		= \int_{\Lambda_{j,\varepsilon}}G_j(\varepsilon x,|T\theta_j\psi_{\varepsilon,j}|^2)dx
		= \int_{\Lambda_{j,\varepsilon}}F(|T\theta_j\psi_{\varepsilon,j}|^2)dx
		= \int_{\mathbb R^2}F(|T\theta_j\psi_{\varepsilon,j}|^2)dx.
		\]
		Thus, by \eqref{V_j}, 
		\[
		\left|J_{\varepsilon,j}(\theta_jT\psi_{\varepsilon,j})-I_{V_j}(\theta_jTw_j)\right|
		\leq\frac{T^2}{2}o_\varepsilon(1)
		+\frac12\left|\int_{\mathbb R^2}\left[F(|T\theta_jw_j|^2)
		-F(|T\theta_jw_j\chi_\varepsilon|^2)\right]dx\right|,
		\]
		where $o_\varepsilon(1)$ is independent of $\theta$.
		Moreover, by the Mean Value Theorem,
		for a.e. $x\in\mathbb R^2$ there exists
		\(
		\xi(x)\in[|T\theta_jw_j\chi_\varepsilon|,|T\theta_jw_j|]
		\)
		such that
		\[
		\left|F(|T\theta_jw_j|^2)-F(|T\theta_jw_j\chi_\varepsilon|^2)\right|
		=T^2\theta_j^2w_j^2(1-\chi_\varepsilon^2)f((\xi(x))^2)
		\leq T^2w_j^2f(T^2w_j^2)(1-\chi_\varepsilon^2)
		\leq T^2w_j^2f(T^2w_j^2).
		\]
		The last function belongs to $L^1(\mathbb R^2)$ by \eqref{p2.4},
		H\"older's inequality, and Lemma \ref{2.1}
		and so, since $\chi_\varepsilon\to 1$ a.e. in $\mathbb{R}^2$, the
		Lebesgue Dominated Convergence Theorem gives
		\begin{equation}\label{uniftheta}
			\int_{\mathbb R^2} T^2w_j^2f(T^2w_j^2)(1-\chi_\varepsilon^2)dx \to 0
		\end{equation}
		and so
		\[
		\int_{\mathbb R^2}\left|F(|T\theta_jw_j|^2)
		-F(|T\theta_jw_j\chi_\varepsilon|^2)\right|dx\longrightarrow0
		\]
		uniformly with respect to $\theta\in[T^{-2},1]^s$, since \eqref{uniftheta} is independent of $\theta$.  Hence
		\begin{equation}\label{JejtoIVj}
			J_{\varepsilon,j}(\theta_jT\psi_{\varepsilon,j})\longrightarrow I_{V_j}(\theta_jTw_j)
		\end{equation}
		uniformly with respect to $\theta\in[T^{-2},1]^s$.

		Thus, since $\widetilde E_\varepsilon\in\Sigma_\varepsilon$ for
		$\varepsilon>0$ small enough, by \eqref{Eesum}, \eqref{JejtoIVj}, and \eqref{carcVj}, we obtain
		\[
		\begin{aligned}
			S_\varepsilon
			&\leq\max_{\theta\in[T^{-2},1]^s}J_\varepsilon(\widetilde E_\varepsilon(\theta))
			=\max_{\theta\in[T^{-2},1]^s}\sum_{j=1}^sJ_{\varepsilon,j}(\theta_jT\psi_{\varepsilon,j})
			=\max_{\theta\in[T^{-2},1]^s}
			\sum_{j=1}^sI_{V_j}(\theta_jTw_j)+o_\varepsilon(1)\\
			&\leq\sum_{j=1}^s\max_{\theta_j\in[T^{-2},1]}
			I_{V_j}(\theta_jTw_j)+o_\varepsilon(1)
			=\sum_{j=1}^sc_{V_j}+o_\varepsilon(1)
		\end{aligned}
		\]
		completing the proof.
	\end{proof}
	
	\begin{remark}\label{4.4}
		Set
		\(
		C_{\varepsilon,\Gamma}:=
		\max\limits_{\theta\in[T^{-2},1]^s}
		J_\varepsilon(\widetilde E_\varepsilon(\theta)).
		\)
		Then{, using the last series of inequalities in the proof of the previous Proposition and its final result,}
		\begin{equation}\label{Cepslimit}
			\lim_{\varepsilon\to0^+}C_{\varepsilon,\Gamma}
			=\sum_{j=1}^sc_{V_j}.
		\end{equation}
	\end{remark}

	Now, let
	\[
	\|u\|_{\varepsilon,j}:=\left(
	\int_{\Lambda_{j,\varepsilon}}
	\bigl(|\nabla_{A_\varepsilon}u|^2+V_\varepsilon(x)|u|^2\bigr)dx
	\right)^{1/2},
	\qquad
	j\in\Gamma,
	\]
	and
	\begin{equation}\label{rhochoice}
		0<\rho<\frac{1}{4}\min\left\{\kappa_*,
		T^{-1}\min_{j\in\Gamma}
		\bigl(|\nabla w_j|_2^2+V_j|w_j|_2^2\bigr)^{1/2}\right\}.
	\end{equation}
	
	By Lemma \ref{lempsiej}, uniformly for $\theta\in[T^{-2},1]^s$ and for $\varepsilon>0$ sufficiently small,
	\begin{equation}\label{pathlocalmass}
		\|\widetilde E_\varepsilon(\theta)\|_{\varepsilon,j}
		=\theta_jT\|\psi_{\varepsilon,j}\|_\varepsilon
		\geq2\rho,\qquad j\in\Gamma.
	\end{equation}
	In addition, Proposition
	\ref{4.3} and \eqref{JejtoIVj} give
	\begin{equation}\label{pathcenter}
		J_\varepsilon\bigl(\widetilde E_\varepsilon(1/T,\ldots,1/T)\bigr)
		=S_\varepsilon+o_\varepsilon(1).
	\end{equation}
	
	Define now
	\begin{equation*}
		m_\varepsilon(u):=\min_{j\in\Gamma}\|u\|_{\varepsilon,j},
		\qquad
		M_\varepsilon(a):=\{u\in H_\varepsilon:m_\varepsilon(u)\geq a\},
		\quad a>0,
	\end{equation*}
	For $u,v\in H_\varepsilon$,
	\begin{equation}\label{mlipschitz}
		|m_\varepsilon(u)-m_\varepsilon(v)|
		\leq\max_{j\in\Gamma}
		\bigl|\|u\|_{\varepsilon,j}-\|v\|_{\varepsilon,j}\bigr|
		\leq\|u-v\|_\varepsilon.
	\end{equation}
	Thus $m_\varepsilon$ is $1$-Lipschitz and so every $M_\varepsilon(a)$ is
	closed. 
	
	Notice that \eqref{pathlocalmass} places the whole reference path
	strictly inside $M_\varepsilon(\rho/2)$.
	
	Let
	\[
	J_\varepsilon^{C_{\varepsilon,\Gamma}}
	:=\{u\in H_\varepsilon:J_\varepsilon(u)\leq
	C_{\varepsilon,\Gamma}\},
	\qquad
	B_{\varepsilon,\nu}:=
	\{u\in H_\varepsilon:|J_\varepsilon(u)-S_\varepsilon|\leq\nu\}.
	\]
	By Remark \ref{4.4}, using \eqref{Eesum}, \eqref{JejtoIVj} and \eqref{carcVj}, for sufficiently small \(\varepsilon>0\),
	both \(\widetilde E_\varepsilon(T^{-2},\ldots,T^{-2})\) and
	\(\widetilde E_\varepsilon(1,\ldots,1)\) belong to
	\(J_\varepsilon^{C_{\varepsilon,\Gamma}}\). Hence
	\(J_\varepsilon^{C_{\varepsilon,\Gamma}}\neq\emptyset\).
	
	Since the maximum in \eqref{carcVj} is unique, each summand in the continuous
	function
	\(
	\theta\longmapsto
	\sum_{j=1}^s\bigl(c_{V_j}-I_{V_j}(\theta_jTw_j)\bigr)
	\)
	is nonnegative and vanishes only when $\theta_j=1/T$.  Hence the sum vanishes only at
	$(1/T,\ldots,1/T)$.  Since this point lies in the interior of
	$[T^{-2},1]^s$, the function is positive on the compact set
	$\partial([T^{-2},1]^s)$.  We can therefore choose $\nu^*>0$ such that
	\begin{equation}\label{p4.7}
		\sum_{j=1}^s\bigl(c_{V_j}-I_{V_j}(\theta_jTw_j)\bigr)
		\geq4\nu^*,
		\qquad \theta\in\partial([T^{-2},1]^s).
	\end{equation}
	From \eqref{pathlocalmass} and \eqref{pathcenter} it follows that, for every $\nu>0$ and small
	$\varepsilon>0$,
	\begin{equation}\label{pathinB}
		\widetilde E_\varepsilon(1/T,\ldots,1/T)
		\in B_{\varepsilon,\nu}\cap M_\varepsilon(\rho/2).
	\end{equation}
	On the other hand, by \eqref{Eesum}, \eqref{JejtoIVj}, Proposition
	\ref{4.3}, and \eqref{p4.7}, uniformly for
	$\theta\in\partial([T^{-2},1]^s)$,
	\[
	S_\varepsilon-J_\varepsilon(\widetilde E_\varepsilon(\theta))
	=\sum_{j=1}^s\bigl(c_{V_j}-I_{V_j}(\theta_jTw_j)\bigr)
	+o_\varepsilon(1)
	\geq4\nu^*+o_\varepsilon(1).
	\]
	Consequently, for every $\nu\in(0,\nu^*)$ and all sufficiently small
	$\varepsilon>0$,
	\begin{equation}\label{boundaryoutsideB}
		\widetilde E_\varepsilon(\theta)\notin B_{\varepsilon,2\nu},
		\qquad \theta\in\partial([T^{-2},1]^s).
	\end{equation}

	The following proposition rules out almost critical points in the transition
	regions between the energy neighborhoods
	$B_{\varepsilon,\nu}\subset B_{\varepsilon,2\nu}$ and the localization sets
	$M_\varepsilon(\rho/2)\subset M_\varepsilon(\rho/4)$. This uniform 
	estimate is the key ingredient in the localized deformation argument used
	below.
	\begin{proposition}\label{4.5}
		For every $\nu\in(0,\nu^*)$, there exist $\sigma_0=\sigma_0(\nu)>0$ and
		$\varepsilon^*=\varepsilon^*(\nu)>0$ such that, for
		$0<\varepsilon<\varepsilon^*$,
		\(\|J_\varepsilon'(u)\|_\varepsilon\geq\sigma_0\) for every
		$u\in J_\varepsilon^{C_{\varepsilon,\Gamma}}$ belonging to
		\begin{equation}\label{gradientregion}
			\bigl(\left(B_{\varepsilon,2\nu}\setminus
			B_{\varepsilon,\nu}\right)\cap M_\varepsilon(\rho/4)\bigr)\cup
			\bigl(B_{\varepsilon,2\nu}\cap
			\left(M_\varepsilon(\rho/4)\setminus
			\operatorname{int}M_\varepsilon(\rho/2)\right)\bigr).
		\end{equation}
	\end{proposition}
	\begin{proof}
		Observe preliminarily that, if $\nu\in(0,\nu^*)$ and $\varepsilon$ is sufficiently small, by \eqref{pathinB} and \eqref{boundaryoutsideB}, the continuous map
		\(\theta\mapsto |J_\varepsilon(\widetilde E_\varepsilon(\theta))
		-S_\varepsilon|\) is at most \(\nu\) at
		\((1/T,\ldots,1/T)\) and greater than \(2\nu\) at \((1,\ldots,1)\).
		Thus, by \eqref{pathlocalmass}, the first set in \eqref{gradientregion} is nonempty.

		Suppose, by contradiction, that there exist $\varepsilon_n\to0^+$ and
		$u_n$ in the indicated set such that
		$\|J_{\varepsilon_n}'(u_n)\|_{\varepsilon_n}\to0$.
		Since $u_n\in J_{\varepsilon_n}^{C_{\varepsilon_n,\Gamma}}$,
		Remark \ref{4.4} gives a uniform upper bound for
		$J_{\varepsilon_n}(u_n)$. Thus, arguing as in \eqref{bddPS}, we first
		obtain the boundedness of $\{u_n\}$ in $H_{\varepsilon_n}$ and then
		\begin{equation}\label{boundJen}
			J_{\varepsilon_n}(u_n)
			\geq\left(\frac12-\frac1{\Theta}-\frac1{2k}\right)
			\|u_n\|_{\varepsilon_n}^2+o_n(1)\geq o_n(1).
		\end{equation}
		Hence, after passing to a subsequence,
		$J_{\varepsilon_n}(u_n)\to c$ for some $c\geq0$, and Remark \ref{4.4}
		also gives
		\begin{equation}\label{goodcbound}
			c\leq\sum_{j=1}^sc_{V_j}.
		\end{equation}
		We have that $c>0$. Indeed, if $c=0$,
		by \eqref{boundJen},
		$\|u_n\|_{\varepsilon_n}\to0$, whereas
		\(\rho/4\leq m_{\varepsilon_n}(u_n)\leq\|u_n\|_{\varepsilon_n}\).
		Moreover, by \eqref{p3.1}, \eqref{goodcbound}, and $s\leq m$,
		\[
		0<c\leq\sum_{j=1}^s c_{V_j}
		\leq s c_{V^*}\leq m c_{V^*}
		<\frac18\left(\frac12-\frac1{\Theta}-\frac1{2k}\right).
		\]
		Therefore, up to a subsequence, Proposition \ref{fengjie} provides a positive integer $l$,
		mutually diverging sequences of points $y_{n,r}$, and nontrivial limit
		profiles $U_{0,r}$ such that
		\[
		\left\|u_n-
		\sum_{r=1}^l e^{iA(y_r)\cdot(\cdot-y_{n,r})}
		(U_{0,r}\chi_{\varepsilon_n})(\cdot-y_{n,r})
		\right\|_{\varepsilon_n}\longrightarrow0,
		\qquad
		c=\sum_{r=1}^lJ_{0,r}(U_{0,r}).
		\]
		
		We have that \(l\geq s\). Indeed, if \(l<s\), then, since the closures of
		the wells are mutually disjoint, there exists \(i\in\Gamma\) such that
		\(y_r\notin\overline\Lambda_i\) for every \(r=1,\ldots,l\). By the
		separation of the wells and the convergences
		\(\varepsilon_ny_{n,r}\to y_r\) and
		\(\varepsilon_n|\log\varepsilon_n|\to0\), for $n$ large enough,
		\[
		\frac\rho4\leq\|u_n\|_{\varepsilon_n,i}
		\leq\left\|u_n-\sum_{r=1}^l e^{iA(y_r)\cdot(\cdot-y_{n,r})}
		(U_{0,r}\chi_{\varepsilon_n})(\cdot-y_{n,r})\right\|_{\varepsilon_n}
		\]
		and so, the strong convergence above gives a contradiction.
		
		The same argument shows that each \(\overline\Lambda_i\),
		\(i\in\Gamma\), contains at least one of the points \(y_1,\ldots,y_l\).
		After relabelling, we may assume that \(y_j\in\overline\Lambda_j\) for
		\(j\in\Gamma\). Then, by \eqref{djl616:profile-energy-lower}, the
		monotonicity of \(c_\mu^c\) with respect to \(\mu\), which follows from
		\eqref{jitaic}, and Lemma \ref{existenceC},
		\[
		\sum_{j=1}^sc_{V_j}\geq c
		=\sum_{r=1}^lJ_{0,r}(U_{0,r})
		\geq\sum_{j=1}^sJ_{0,j}(U_{0,j})
		\geq\sum_{j=1}^sc^c_{V(y_j)}
		\geq\sum_{j=1}^sc^c_{V_j}
		=\sum_{j=1}^sc_{V_j}.
		\]
		Hence equality holds throughout the preceding chain and, since every nontrivial
		profile has strictly positive energy by \eqref{djl616:profile-energy-lower},
		\(l=s\). 
		
		We next show that the corresponding concentration
		points belong to the minimum sets. Equality in the preceding chain gives
		\(J_{0,j}(U_{0,j})=c_{V_j}\) for every \(j\in\Gamma\). Moreover, \eqref{finalsplitting} gives
		\begin{equation}\label{fromfinalsplitting}
			c_{V_j}=J_{0,j}(U_{0,j})
			\geq\max_{t\geq0}I^c_{V(y_j)}(tU_{0,j})
			\geq\max_{t\geq0}I^c_{V_j}(tU_{0,j})
			\geq c^c_{V_j}=c_{V_j}.
		\end{equation}
		Fix \(j\in\Gamma\), and let \(t_j>0\)
		be such that
		\(
		I^c_{V_j}(t_jU_{0,j})
		=\max_{t\geq0}I^c_{V_j}(tU_{0,j}),
		\)
		whose existence follows from \eqref{IcNeh}. If \(V(y_j)>V_j\), then
		\[
		\begin{aligned}
			\max_{t\geq0}I^c_{V(y_j)}(tU_{0,j})
			&\geq I^c_{V(y_j)}(t_jU_{0,j})=I^c_{V_j}(t_jU_{0,j})
			+\frac{V(y_j)-V_j}{2}t_j^2|U_{0,j}|_2^2>\max_{t\geq0}I^c_{V_j}(tU_{0,j}),
		\end{aligned}
		\]
		which contradicts \eqref{fromfinalsplitting}. Therefore
		\(V(y_j)=V_j\) and, by (\ref{V2}) and \eqref{Kj},
		\(y_j\in K_j\).
		
		Moreover, since
		\(y_j\in K_j\), Remark
		\ref{eta-one} gives \(\eta_j=1\) a.e. in \(\mathbb R^2\) and so
		\[
		g_{0,j}(x,t)=f(t),\qquad G_{0,j}(x,t)=F(t)
		\quad\hbox{for a.e. }x\in\mathbb R^2\hbox{ and every }t\geq0,
		\quad
		J_{0,j}=I^c_{V_j},
		\]
		by \eqref{fromfinalsplitting}, each \(U_{0,j}\) is a ground state
		solution of \eqref{fude} with \(\mu=V_j\), and
		\begin{equation*}
			J_{\varepsilon_n}(u_n)\to\sum_{j=1}^sc_{V_j}.
		\end{equation*}
		Proposition \ref{4.3} rules out the annular part of
		\eqref{gradientregion}. Consequently,
		\[
		u_n\in B_{\varepsilon_n,2\nu}
		\cap\left(M_{\varepsilon_n}(\rho/4)
		\setminus\operatorname{int}M_{\varepsilon_n}(\rho/2)\right).
		\]
		Of course, if, for $n$ large enough, $B_{\varepsilon_n,2\nu}
		\cap\left(M_{\varepsilon_n}(\rho/4)
		\setminus\operatorname{int}M_{\varepsilon_n}(\rho/2)\right)=\emptyset$, we get a contradiction.
		
		For \(j\in\Gamma\), the support of \(\Psi_{n,j}:=e^{iA(y_j)\cdot(\cdot-y_{n,j})}
		(U_{0,j}\chi_{\varepsilon_n})(\cdot-y_{n,j})\) is contained in \(\overline{B}_{2|\log\varepsilon_n|}(y_{n,j})\).
		Since \(\varepsilon_n y_{n,j}\to y_j\in K_j\) and
		\(\varepsilon_n|\log\varepsilon_n|\to0\), \(\overline{B}_{2|\log\varepsilon_n|}(y_{n,j})\subset\Lambda_{j,\varepsilon_n}\) for large \(n\).
		The disjointness of the wells then shows that all terms with \(i\ne j\) vanish on \(\Lambda_{j,\varepsilon_n}\).
		Then, restricting the profile decomposition \eqref{decomposition} from Proposition \ref{fengjie} to
		\(\Lambda_{j,\varepsilon_n}\), we get
		\[\left|\|u_n\|_{\varepsilon_n,j}
		- \|\Psi_{n,j}\|_{\varepsilon_n,j}
		\right|\to0.
		\]
		By the support inclusion above and by repeating the argument leading to \eqref{mo}
		we obtain
		\[
		\|\Psi_{n,j}\|_{\varepsilon_n}^2=\|\Psi_{n,j}\|_{\varepsilon_n,j}^2
		=|\nabla U_{0,j}|_2^2+V_j|U_{0,j}|_2^2+o_n(1)
		\]
		and so
		\begin{equation}\label{localprofilenorm}
			\|u_n\|_{\varepsilon_n,j}\to
			\left(|\nabla U_{0,j}|_2^2+V_j|U_{0,j}|_2^2\right)^{1/2},
			\qquad j\in\Gamma.
		\end{equation}
		Moreover, since \(u_n\in M_{\varepsilon_n}(\rho/4)\) and 
		$m_{\varepsilon_n}$ is continuous,
		\[
		\{u\in H_{\varepsilon_n}:m_{\varepsilon_n}(u)>\rho/2\}
		\subset\operatorname{int}M_{\varepsilon_n}(\rho/2).
		\]
		Thus $u_n\notin\operatorname{int}M_{\varepsilon_n}(\rho/2)$ implies that
		$m_{\varepsilon_n}(u_n)\leq\rho/2$. Since \(\Gamma\) is finite, after taking
		a subsequence, there is a \(j\in\Gamma\) such that
		\begin{equation}\label{r4r2}
			\frac{\rho}{4}\leq\|u_n\|_{\varepsilon_n,j}\leq\frac{\rho}{2}.
		\end{equation}
		Since \(U_{0,j}\) is a
		ground state solution of \eqref{fude} with \(\mu=V_j\), by Lemma \ref{uniformNeharilowerbound} and \eqref{rhochoice}
		\[
		\left(|\nabla U_{0,j}|_2^2+V_j|U_{0,j}|_2^2\right)^{1/2}\geq\kappa_*>4\rho,
		\]
		which is in contradiction with \eqref{localprofilenorm} and \eqref{r4r2}.
		This completes the proof.
	\end{proof}
	
	Now we show a pseudo-gradient lemma
	in the \(C^1\) setting. For a related construction based on a partition
	of unity, see \cite[Lemma 2.7]{SquassinaSzulkinErratum}.

	\begin{lemma}
		\label{pseudogradient}
		Let \(E\) be a separable Hilbert space, with scalar product $\langle\cdot,\cdot\rangle$, of functions on
		\(\mathbb R^2\) such that
		\(C_c^\infty(\mathbb R^2,\mathbb C)\) is dense in \(E\), and let
		\(J\in C^1(E,\mathbb R)\). If \(E\) is complex, it is regarded as a
		real Hilbert space. Denote the set of critical points of \(J\) by \(K\). Then there exists a locally Lipschitz vector field $H:E\setminus K\to E$ such that, for every \(u\in E\setminus K\),
		\begin{equation*}
			\|H(u)\|\leq1,
			\qquad
			J'(u)[H(u)]
			>
			\frac12\|J'(u)\|.
		\end{equation*}
	\end{lemma}
	
	\begin{proof}
		First observe that, for each \(u\in E\setminus K\), there exists $v_u\in C_c^\infty(\mathbb R^2,\mathbb C)$ such that $\|v_u\|=1$ and $J'(u)[v_u] > \|J'(u)\|/2$.
		Indeed, by the Riesz representation theorem, there exists a unique
		$\mathfrak{j}_u\in E$ such that \(J'(u)=\langle \mathfrak{j}_u,\cdot\rangle\)
		and $\|J'(u)\|=\|\mathfrak{j}_u\|$. Since $J'(u)\neq 0$, then $\mathfrak{j}_u\neq 0$ and
		so we can consider $w_u:=\mathfrak{j}_u/\|\mathfrak{j}_u\|$.
		Thus $\|w_u\|=1$ and $J'(u)[w_u]=\|\mathfrak{j}_u\|=\|J'(u)\|$.
		By the density of $C_c^\infty(\mathbb{R}^2,\mathbb{C})$ in $E$, we can take $\varphi_u\in C_c^\infty(\mathbb{R}^2,\mathbb{C})$ such that $\|\varphi_u-w_u\|<1/3$.
		Then $\varphi_u\neq 0$ and
		\[
		J'(u)[\varphi_u]
		=J'(u)[w_u]+J'(u)[\varphi_u-w_u]
		\geq \|J'(u)\| \bigl(1-
		\|\varphi_u-w_u\|\bigr)
		>\frac{2}{3}\|J'(u)\|,
		\]
		and $\|\varphi_u\| \leq \|w_u\|+\|\varphi_u-w_u\| <4/3$.
		Set $v_u:=\varphi_u/\|\varphi_u\|$. We have that $v_u\in C_c^\infty(\mathbb{R}^2,\mathbb{C})$, $\|v_u\|=1$, and
		\[
		J'(u)[v_u]
		=\frac{J'(u)[\varphi_u]}
		{\|\varphi_u\|}
		>\frac12\|J'(u)\|.
		\]
		
		Set $\delta_u:= J'(u)[v_u] -\|J'(u)\|/2>0$.
		
		Since \(J'\) is continuous and \(E\setminus K\) is open, there exists
		a neighborhood \(W(u)\subset E\setminus K\) of \(u\) such that, for every $w\in W(u)$, \(\|J'(w)-J'(u)\|<\delta_u/2\).
		Then, for every \(w\in W(u)\), since \(\|v_u\|=1\),
		\[
		J'(w)[v_u]
		=J'(u)[v_u]+(J'(w)-J'(u))[v_u]
		>J'(u)[v_u]-\frac{\delta_u}{2}
		=\frac{1}{2}(\|J'(u)\|+\delta_u)
		\]
		and so
		\begin{equation}\label{localpseudogradientestimate}
			\|J'(w)\|
			\leq \|J'(u)\|
			+\|J'(w)-J'(u)\|
			<\|J'(u)\|+\frac{\delta_u}{2}
			<\|J'(u)\|+\delta_u
			<2J'(w)[v_u].
		\end{equation}
		
		Since \(E\setminus K\) is a separable metric space, the covering
		\(\{W(u):u\in E\setminus K\}\)
		has a countable subcover \(\{W(u_k)\}_{k\geq1}\). 
		Since
		\(E\setminus K\) is metric and hence paracompact, there exists a
		locally finite open refinement \(\{W_j\}_{j\geq1}\) of this
		subcover. For every \(j\), choose \(k_j\) such that \(W_j\subset W(u_{k_j})\) and set \(v_j:=v_{u_{k_j}}\).
		Thus, by \eqref{localpseudogradientestimate},
		for every $w\in W_j$, $2J'(w)[v_j]> \|J'(w)\|$.
		
		For $u\in E\setminus K$, let $\widetilde\rho_j(u):=
		\operatorname{dist}(u,E\setminus W_j)$,
		\[
		\rho_j(u):=
		\frac{\widetilde\rho_j(u)}
		{\displaystyle\sum_{k\geq1}\widetilde\rho_k(u)},
		\quad
		\text{and}
		\quad
		H(u):=
		\sum_{j\geq1}\rho_j(u)v_j.
		\]
		Observe that, for every \(u\in E\setminus K\), by the local
		finiteness of the covering, there exist a neighborhood \(U\) of \(u\) and
		a finite set \(I\subset\mathbb N\) such that
		\(U\cap W_j=\emptyset\) for every \(j\notin I\). Hence
		\(\widetilde\rho_j=0\) on \(U\) for \(j\notin I\), so both sums in the
		definitions of \(\rho_j\) and \(H\) reduce to finite sums on \(U\). Moreover,
		since \(\{W_j\}_{j\geq1}\) covers \(E\setminus K\), there exists
		\(j_u\in I\) such that \(u\in W_{j_u}\). Since \(W_{j_u}\) is open, $\widetilde\rho_{j_u}(u)>0$
		and therefore the denominator
		\[
		S(u):=\sum_{k\geq1}\widetilde\rho_k(u)
		\]
		is positive.
		
		Moreover, each function \(\widetilde\rho_j\) is
		\(1\)-Lipschitz, because it is the distance from a fixed closed set. Thus
		\(S\) is Lipschitz on \(U\). After shrinking \(U\), we may assume that
		\(S\geq c_0>0\) there. It follows that every quotient
		\(\rho_j\), \(j\in I\), is Lipschitz on \(U\), while
		\(\rho_j=0\) on \(U\) for \(j\notin I\). Hence, \(H\) is a finite
		sum of locally Lipschitz maps on \(U\), and hence is locally Lipschitz on
		\(E\setminus K\).
		
		Finally,
		\[
		\|H(u)\|
		\leq\sum_{j\geq1}\rho_j(u)\|v_j\|
		=\sum_{j\geq1}\rho_j(u)=1
		\]
		and
		\[
		J'(u)[H(u)]
		=\sum_{j\geq1}\rho_j(u)J'(u)[v_j]
		>\frac12\|J'(u)\|\sum_{j\geq1}\rho_j(u)
		=\frac12\|J'(u)\|.
		\]

	\end{proof}

	We next combine the previous results with a deformation argument to obtain a critical point.
	\begin{proposition}\label{4.6}
		For every $\nu\in(0,\nu^*/2)$ there exists
		$\varepsilon^*=\varepsilon^*(\nu)>0$ such that, for every
		$0<\varepsilon<\varepsilon^*$, the functional $J_\varepsilon$ has a
		critical point in
		\(
		B_{\varepsilon,\nu}\cap
		J_\varepsilon^{C_{\varepsilon,\Gamma}}\cap M_\varepsilon(\rho/2).
		\)
	\end{proposition}
	\begin{proof}
		Suppose, by contradiction, that there are $\nu\in(0,\nu^*/2)$ and
		$\varepsilon_n\to0^+$ such that $J_{\varepsilon_n}$ has no critical point
		in
		\(
		B_{\varepsilon_n,\nu}\cap
		J_{\varepsilon_n}^{C_{\varepsilon_n,\Gamma}}
		\cap M_{\varepsilon_n}(\rho/2).
		\)
		By Proposition \ref{4.5}, for $n$ large there is no critical point in
		\(
		B_{\varepsilon_n,2\nu}\cap
		J_{\varepsilon_n}^{C_{\varepsilon_n,\Gamma}}\cap
		\bigl(M_{\varepsilon_n}(\rho/4)\setminus
		\operatorname{int}M_{\varepsilon_n}(\rho/2)\bigr)
		\)
		and so
		there is no critical point in
		$B_{\varepsilon_n,\nu}\cap
		J_{\varepsilon_n}^{C_{\varepsilon_n,\Gamma}}\cap
		M_{\varepsilon_n}(\rho/4)$.
		Then, for each fixed sufficiently large $n$, there exists $b_n>0$ such that for every $u\in B_{\varepsilon_n,\nu}\cap
		J_{\varepsilon_n}^{C_{\varepsilon_n,\Gamma}}\cap
		M_{\varepsilon_n}(\rho/4)$,
		\begin{equation}\label{fixedgradientbound}
			\|J_{\varepsilon_n}'(u)\|_{\varepsilon_n}\geq b_n.
		\end{equation}
		Indeed, otherwise there would be a sequence $u_m$ in the set
		such that $\|J_{\varepsilon_n}'(u_m)\|_{\varepsilon_n}\to0$. Since
		$J_{\varepsilon_n}(u_m)\leq C_{\varepsilon_n,\Gamma}$, arguing as in
		\eqref{bddPS} we first obtain that $\{u_m\}$ is bounded in
		$H_{\varepsilon_n}$ and then,
		after passing to a subsequence,
		$J_{\varepsilon_n}(u_m)\to d$ for some
		$0\leq d\leq C_{\varepsilon_n,\Gamma}$. Since
		$m_{\varepsilon_n}(u_m)\geq\rho/4$, we further obtain
		\[
		d\geq\left(\frac12-\frac1{\Theta}-\frac1{2k}\right)
		\liminf_{m\to+\infty}\|u_m\|_{\varepsilon_n}^2
		\geq\left(\frac12-\frac1{\Theta}-\frac1{2k}\right)\frac{\rho^2}{16}>0.
		\]
		Furthermore,
		\eqref{Cepslimit} together with \eqref{p3.1} and $s\leq m$ gives
		\[C_{\varepsilon_n,\Gamma}<\frac18\left(\frac12-\frac1{\Theta}-\frac1{2k}\right)
		<\frac12-\frac1{\Theta}-\frac1{2k}\] for $n$ large. Thus by Lemma
		\ref{3.1}, up to a subsequence,
		$u_m\to u$ in $H_{\varepsilon_n}$.  Hence, since the three sets in the intersection are
		closed, $u$ belongs to the same intersection and
		$J_{\varepsilon_n}'(u)=0$, a contradiction.
		
		Let \(\sigma_0>0\) be given by Proposition \ref{4.5}, and set
		\(\widetilde b_n:=\min\{b_n,\sigma_0\}\).  We now construct the deformation.
		Define the two cut-offs
		\[
		\vartheta_n(u):=
		\begin{cases}
			1,&|J_{\varepsilon_n}(u)-S_{\varepsilon_n}|\leq3\nu/2,\\
			\displaystyle\frac{2\nu-|J_{\varepsilon_n}(u)-S_{\varepsilon_n}|}{\nu/2},
			&3\nu/2<|J_{\varepsilon_n}(u)-S_{\varepsilon_n}|<2\nu,\\
			0,&|J_{\varepsilon_n}(u)-S_{\varepsilon_n}|\geq2\nu,
		\end{cases}
		\]
		and
		\[
		\varphi_n(u):=
		\begin{cases}
			0,&m_{\varepsilon_n}(u)\leq\rho/4,\\
			\displaystyle\frac{4m_{\varepsilon_n}(u)}\rho-1,
			&\rho/4<m_{\varepsilon_n}(u)<\rho/2,\\
			1,&m_{\varepsilon_n}(u)\geq\rho/2.
		\end{cases}
		\]
		By \eqref{mlipschitz}, \(\varphi_n\) is Lipschitz. Since
		\(J_{\varepsilon_n}\in C^1(H_{\varepsilon_n},\mathbb R)\), the map
		\(u\mapsto |J_{\varepsilon_n}(u)-S_{\varepsilon_n}|\) is locally Lipschitz,
		and hence so is \(\vartheta_n\). Put
		\begin{equation*}
			\Psi_n(u):=\vartheta_n(u)\varphi_n(u).
		\end{equation*}

		By Lemma \ref{pseudogradient} applied to $E=H_{\varepsilon_n}$, $J=J_{\varepsilon_n}$, and $K_n=\{u\in H_{\varepsilon_n}:J_{\varepsilon_n}'(u)=0\}$, there exists a locally Lipschitz vector field $Y_n:H_{\varepsilon_n}\setminus K_n\to H_{\varepsilon_n}$ such that, for every \(u\in H_{\varepsilon_n}\setminus K_n\),
		\[
		\|Y_n(u)\|_{\varepsilon_n}\leq1,
		\qquad
		J_{\varepsilon_n}'(u)[Y_n(u)]
		>
		\frac12
		\|J_{\varepsilon_n}'(u)\|_{\varepsilon_n}.
		\]

		Furthermore, since the closed set
		\(J_{\varepsilon_n}^{C_{\varepsilon_n,\Gamma}}\cap
		B_{\varepsilon_n,2\nu}\cap M_{\varepsilon_n}(\rho/4)\) is contained in the
		open set \(\{u\in H_{\varepsilon_n}:
		\|J_{\varepsilon_n}'(u)\|_{\varepsilon_n}>\widetilde b_n/2\}\),
		using the distance functions of these two sets, choose a locally Lipschitz
		function \(\omega_n:H_{\varepsilon_n}\to[0,1]\) which is equal to \(1\) on
		the former set and vanishes outside the latter one. 
		
		Then, define
		\[
		H_n(u):=
		\begin{cases}
			-\omega_n(u)\Psi_n(u)Y_n(u),&J_{\varepsilon_n}'(u)\ne0,\\
			0,&J_{\varepsilon_n}'(u)=0.
		\end{cases}
		\]
		To verify the regularity across the critical set, let
		\(u_*\in H_{\varepsilon_n}\) satisfy \(J_{\varepsilon_n}'(u_*)=0\).
		By the continuity of \(J_{\varepsilon_n}'\), there exists \(\delta_*>0\)
		such that \(\|J_{\varepsilon_n}'(v)\|_{\varepsilon_n}<\widetilde b_n/2\)
		whenever \(\|v-u_*\|_{\varepsilon_n}<\delta_*\).
		The definition of \(\omega_n\) then gives \(\omega_n(v)=0\), and consequently
		\(H_n(v)=0\), throughout this neighborhood.  Hence \(H_n\) is locally
		Lipschitz near every critical point.
		At every noncritical point this follows
		from the local Lipschitz continuity of \(\omega_n\), \(\Psi_n\), and \(Y_n\).
		Therefore \(H_n:H_{\varepsilon_n}\to H_{\varepsilon_n}\) is locally
		Lipschitz.  Moreover,
		\begin{equation}\label{localizedpseudogradient}
			\|H_n(u)\|_{\varepsilon_n}\leq1,
			\qquad H_n(u)=0\quad\hbox{if }\Psi_n(u)=0.
		\end{equation}
		Observe that, for every $u\in H_{\varepsilon_n}$,
		\begin{equation}\label{globaldescent}
			J_{\varepsilon_n}'(u)[H_n(u)]\leq 0.
		\end{equation}
		Indeed, if \(J_{\varepsilon_n}'(u)=0\),  the estimate is immediate.  If
		\(J_{\varepsilon_n}'(u)\ne0\), then, by the definition of \(H_n\) and the
		defining property of \(Y_n\),
		\begin{equation}\label{previousineq}
			J_{\varepsilon_n}'(u)[H_n(u)]
			=-\omega_n(u)\Psi_n(u)J_{\varepsilon_n}'(u)[Y_n(u)]
			\leq-\frac12\omega_n(u)\Psi_n(u)\|J_{\varepsilon_n}'(u)\|_{\varepsilon_n}\leq0.
		\end{equation}
		In particular, for every  $u\in J_{\varepsilon_n}^{C_{\varepsilon_n,\Gamma}}$, 
		\begin{equation}\label{localizeddescent}
			J_{\varepsilon_n}'(u)[H_n(u)]
			\leq-\frac12\Psi_n(u)
			\|J_{\varepsilon_n}'(u)\|_{\varepsilon_n}.
		\end{equation}
		To verify this, if \(\Psi_n(u)>0\) and
		\(u\in J_{\varepsilon_n}^{C_{\varepsilon_n,\Gamma}}\), then
		\(u\in B_{\varepsilon_n,2\nu}\cap M_{\varepsilon_n}(\rho/4)\), and hence
		\(\omega_n(u)=1\). Thus, \eqref{previousineq} yields \eqref{localizeddescent}.  If \(\Psi_n(u)=0\), then \(H_n(u)=0\), and both sides
		of \eqref{localizeddescent} are zero.

		For every $u\in H_{\varepsilon_n}$, consider the flow
		\begin{equation*}
			\dot\eta_n(t,u)=H_n(\eta_n(t,u)),
			\qquad \eta_n(0,u)=u.
		\end{equation*}
		By \cite[\S1.1]{DeimlingODE}\footnote{If \(X\) is a Banach
			space, \(D\subset X\) is open, and \(F:D\to X\) is locally Lipschitz, then for
			every \(u_0\in D\) the problem
			\[
			\dot u(t)=F(u(t)),\qquad u(0)=u_0,
			\]
			has a unique solution on a maximal interval.}
		and since \(H_n\) is locally
		Lipschitz, the preceding fact gives a unique solution on a maximal interval.
		
		Let us show that it is globally defined.
		
		Fix \(n\) and \(u\in H_{\varepsilon_n}\). By \eqref{localizedpseudogradient}, for \(0\leq s\leq t\) in its interval
		of existence,
		\begin{equation}\label{ineqflow}
			\|\eta_n(t,u)-\eta_n(s,u)\|_{\varepsilon_n}
			\leq\int_s^t\|H_n(\eta_n(\tau,u))\|_{\varepsilon_n}d\tau
			\leq t-s.
		\end{equation}
		Suppose, by contradiction, that
		the solution is defined on \([0,T)\) with \(T<+\infty\).
		Then, there exists $\lim\limits_{t\to T^{-}}\eta_n(t,u)$. Indeed,
		if
		\(t_k\to T^{-}\), then, by \eqref{ineqflow}, \(\{\eta_n(t_k,u)\}_{k}\) is a Cauchy sequence in the
		Hilbert space \(H_{\varepsilon_n}\).  Since \(H_{\varepsilon_n}\) is complete,
		this sequence converges in \(H_{\varepsilon_n}\), and 
		\eqref{ineqflow} also
		shows that the limit is independent of the sequence \(t_k\).  Denote this
		limit by \(u_T\). Since \(H_n\) is locally
		Lipschitz in a neighborhood of \(u_T\), the local existence theorem applied at
		the initial point \(u_T\) gives a solution \(\zeta\) of
		\[
		\dot\zeta(t)=H_n(\zeta(t)),\qquad \zeta(0)=u_T,
		\]
		on some interval \([0,\delta)\). Concatenating \(\eta_n(\cdot,u)\)
		with \(t\mapsto\zeta(t-T)\) extends the solution beyond \(T\), contradicting
		the maximality of \(T\).

		We also have the continuity with respect to the initial point. Indeed,
		fixed \(t>0\), since the curve \(\{\eta_n(s,u):0\leq s\leq t\}\) is compact and
		\(H_n\) is locally Lipschitz, we may divide \([0,t]\) into finitely many
		subintervals \(0=t_0<t_1<\cdots<t_m=t\) such that, on a neighborhood of
		\(\{\eta_n(s,u):t_{r-1}\leq s\leq t_r\}\), the vector field \(H_n\) is
		Lipschitz with constant \(L_r\).  If \(v\) is sufficiently close to \(u\),
		then, by induction on \(r\), both trajectories stay in this neighborhood on
		\([t_{r-1},t_r]\).  Hence, for \(t_{r-1}\leq s\leq t_r\),
		\[
		\begin{aligned}
			\|\eta_n(s,v)-\eta_n(s,u)\|_{\varepsilon_n}
			&\leq \|\eta_n(t_{r-1},v)-\eta_n(t_{r-1},u)\|_{\varepsilon_n}
			{}+\int_{t_{r-1}}^s\|H_n(\eta_n(\tau,v))-H_n(\eta_n(\tau,u))\|_{\varepsilon_n}\,d\tau\\
			&\leq \|\eta_n(t_{r-1},v)-\eta_n(t_{r-1},u)\|_{\varepsilon_n}
			{}+L_r\int_{t_{r-1}}^s\|\eta_n(\tau,v)-\eta_n(\tau,u)\|_{\varepsilon_n}\,d\tau .
		\end{aligned}
		\]
		By Gronwall's inequality, it follows that, for \(t_{r-1}\leq s\leq t_r\),
		\[
		\|\eta_n(s,v)-\eta_n(s,u)\|_{\varepsilon_n}
		\leq e^{L_r(s-t_{r-1})}
		\|\eta_n(t_{r-1},v)-\eta_n(t_{r-1},u)\|_{\varepsilon_n}.
		\]
		Taking \(s=t_r\) and applying the last estimate successively for
		\(r=1,\ldots,m\), we get
		\[
		\|\eta_n(t,v)-\eta_n(t,u)\|_{\varepsilon_n}
		\leq
		e^{\sum_{r=1}^mL_r(t_r-t_{r-1})}
		\|v-u\|_{\varepsilon_n}.
		\]
		The constant on the right is independent of \(v\).  Therefore
		\(\eta_n(t,v)\to\eta_n(t,u)\) in \(H_{\varepsilon_n}\) as \(v\to u\).

		By the chain rule and \eqref{globaldescent},
		\[
		\frac d{dt}J_{\varepsilon_n}(\eta_n(t,u))
		=J_{\varepsilon_n}'(\eta_n(t,u))
		[H_n(\eta_n(t,u))]\leq0.
		\]
		Consequently, every trajectory starting in
		\(J_{\varepsilon_n}^{C_{\varepsilon_n,\Gamma}}\) remains in the same
		sublevel.  Along such a trajectory, \eqref{localizedpseudogradient} and
		\eqref{localizeddescent} yield
		\begin{equation}\label{flowproperties}
			\|\dot\eta_n(t,u)\|_{\varepsilon_n}\leq1,\qquad
			\frac d{dt}J_{\varepsilon_n}(\eta_n(t,u))
			\leq-\frac12\Psi_n(\eta_n(t,u))
			\|J_{\varepsilon_n}'(\eta_n(t,u))\|_{\varepsilon_n}\leq0.
		\end{equation}
		Moreover,
		\begin{equation}\label{etanb}
			u\in M_{\varepsilon_n}(\rho/4)\quad\Longrightarrow\quad
			\eta_n(t,u)\in M_{\varepsilon_n}(\rho/4)\quad\text{for every }t\geq0.
		\end{equation}
		Indeed, if \(u\in M_{\varepsilon_n}(\rho/4)\), suppose, by
		contradiction, that
		\(m_{\varepsilon_n}(\eta_n(t_1,u))<\rho/4\) for some \(t_1>0\).  Since
		\(m_{\varepsilon_n}\) and \(t\mapsto\eta_n(t,u)\) are continuous, there is
		\(t_*\in[0,t_1)\) such that
		\(m_{\varepsilon_n}(\eta_n(t_*,u))=\rho/4\).
		By the definition of \(\varphi_n\), we have
		\(\varphi_n(\eta_n(t_*,u))=0\), and therefore
		\(\Psi_n(\eta_n(t_*,u))=0\) and
		\(H_n(\eta_n(t_*,u))=0\).  Thus the constant curve
		\(t\in [t_*,\infty)\mapsto\eta_n(t_*,u)\) is a solution on \([t_*,\infty)\) of 
		\[
		\dot\zeta(t)=H_n(\zeta(t)),
		\qquad \zeta(t_*)=\eta_n(t_*,u).
		\]
		Uniqueness gives \(\eta_n(t,u)=\eta_n(t_*,u)\) for \(t\geq t_*\), which contradicts
		\(m_{\varepsilon_n}(\eta_n(t_1,u))<\rho/4\).  
		
		Let
		\[
		d_0:=\min\left\{\frac\nu2,\frac{3\sigma_0\rho}{64}\right\}>0,
		\qquad
		T_n:=\frac{8d_0}{\widetilde b_n}.
		\]
		
		We claim that there exists $\xi>0$, independent of $n$, such that
		\begin{equation}\label{claim1}
			\limsup_{n\to+\infty}
			\max_{\theta\in[T^{-2},1]^s}
			J_{\varepsilon_n}
			(\eta_n(T_n,\widetilde E_{\varepsilon_n}(\theta)))
			\leq\sum_{j=1}^sc_{V_j}-\xi.
		\end{equation}
		Set
		\[
		k_0^n:=\sup\left\{J_{\varepsilon_n}(u):
		u\in\widetilde E_{\varepsilon_n}([T^{-2},1]^s)
		\setminus B_{\varepsilon_n,\nu}\right\}.
		\]
		Observe that, by \eqref{boundaryoutsideB}, the set
		$\widetilde E_{\varepsilon_n}([T^{-2},1]^s)\setminus B_{\varepsilon_n,\nu}$
		is nonempty for all sufficiently large $n$.
		
		Remark \ref{4.4} implies that
		\(
		\limsup_{n\to+\infty}k_0^n\leq\sum_{j=1}^sc_{V_j}.
		\)
		We claim that
		\begin{equation}\label{strict}
			\limsup_{n\to+\infty}k_0^n<\sum_{j=1}^sc_{V_j}.
		\end{equation}
		Otherwise, after passing to a
		subsequence, still denoted by $\{k_0^n\}$, we would have
		\(
		k_0^n\to\sum_{j=1}^sc_{V_j}.
		\)
		For every $n$, choose
		\(
		v_n\in\widetilde E_{\varepsilon_n}([T^{-2},1]^s)
		\setminus B_{\varepsilon_n,\nu}
		\)
		such that
		\(
		J_{\varepsilon_n}(v_n)\geq k_0^n-1/n
		\).
		Then, by Proposition \ref{4.3},
		\(
		J_{\varepsilon_n}(v_n)-S_{\varepsilon_n}\to0
		\)
		and so $v_n\in B_{\varepsilon_n,\nu}$ for all sufficiently large
		$n$, which is a contradiction.
		
		Fix $n$ sufficiently large and $\theta\in[T^{-2},1]^s$.  If
		$\widetilde E_{\varepsilon_n}(\theta)\notin B_{\varepsilon_n,\nu}$,
		then the monotonicity of the energy along the flow and the definition of
		\(k_0^n\) give
		\begin{equation}\label{strict1}
			J_{\varepsilon_n}
			(\eta_n(T_n,\widetilde E_{\varepsilon_n}(\theta)))
			\leq
			J_{\varepsilon_n}(\widetilde E_{\varepsilon_n}(\theta))
			\leq k_0^n.
		\end{equation}
		If, instead, $\widetilde E_{\varepsilon_n}(\theta)\in B_{\varepsilon_n,\nu}$, by the definition of $C_{\varepsilon_n,\Gamma}$, and
		\eqref{pathlocalmass}, for all sufficiently large $n$,
		\[
		\eta_n(0,\widetilde E_{\varepsilon_n}(\theta))
		=\widetilde E_{\varepsilon_n}(\theta)
		\in J_{\varepsilon_n}^{C_{\varepsilon_n,\Gamma}}
		\cap B_{\varepsilon_n,\nu}\cap M_{\varepsilon_n}(\rho/2).
		\]
		In particular, the monotonicity of the energy along the flow implies
		$\eta_n(t,\widetilde E_{\varepsilon_n}(\theta))
		\in J_{\varepsilon_n}^{C_{\varepsilon_n,\Gamma}}$ for every $t\geq0$.
		
		If the trajectory $t\mapsto
		\eta_n(t,\widetilde E_{\varepsilon_n}(\theta))$ leaves
		$B_{\varepsilon_n,3\nu/2}$ before $T_n$, then, since its energy is
		non-increasing and its initial point belongs to $B_{\varepsilon_n,\nu}$,
		it can leave only through the lower boundary
		so that
		\begin{equation}\label{strict2}
			J_{\varepsilon_n}
			(\eta_n(T_n,\widetilde E_{\varepsilon_n}(\theta)))
			\leq S_{\varepsilon_n}-\frac{3\nu}{2}.
		\end{equation}
		If
		$\eta_n(t,\widetilde E_{\varepsilon_n}(\theta))
		\in B_{\varepsilon_n,3\nu/2}$ for $0\leq t\leq T_n$, let $\tau_n\in(0,T_n]$ be the first time at which
		$m_{\varepsilon_n}(\eta_n(t,\widetilde E_{\varepsilon_n}(\theta)))
		=\rho/2$ (if this level is reached), or $\tau_n=T_n$ if this level is not reached before time
		$T_n$.  By \eqref{pathlocalmass} and the definition of $\tau_n$,
		\(m_{\varepsilon_n}(\eta_n(t,\widetilde E_{\varepsilon_n}(\theta)))
		\geq\rho/2\) for \(0\leq t\leq\tau_n\).
		
		Then, on $[0,\tau_n]$, we have
		\[\vartheta_n(\eta_n(t,\widetilde E_{\varepsilon_n}(\theta)))
		=\varphi_n(\eta_n(t,\widetilde E_{\varepsilon_n}(\theta)))
		=\Psi_n(\eta_n(t,\widetilde E_{\varepsilon_n}(\theta)))=1
		\]
		and
		\[
		\eta_n(t,\widetilde E_{\varepsilon_n}(\theta))\in
		J_{\varepsilon_n}^{C_{\varepsilon_n,\Gamma}}
		\cap B_{\varepsilon_n,2\nu}\cap M_{\varepsilon_n}(\rho/4)
		\]
		so that
		$\|J_{\varepsilon_n}'
		(\eta_n(t,\widetilde E_{\varepsilon_n}(\theta)))\|_{\varepsilon_n}
		\geq\widetilde b_n$. 
		
		If $\tau_n\geq T_n/2$, then,
		by \eqref{flowproperties},
		\begin{equation}\label{strict3}
			\begin{split}
				J_{\varepsilon_n}(\eta_n(T_n,\widetilde E_{\varepsilon_n}(\theta)))
				&\leq J_{\varepsilon_n}(\eta_n(T_n/2,\widetilde E_{\varepsilon_n}(\theta)))
				\leq J_{\varepsilon_n}(\widetilde E_{\varepsilon_n}(\theta))
				-\frac12\int_0^{T_n/2}\|J_{\varepsilon_n}'(\eta_n(t,\widetilde E_{\varepsilon_n}(\theta)))\|_{\varepsilon_n}dt\\
				&\leq C_{\varepsilon_n,\Gamma}-\frac{\widetilde b_nT_n}{4}
				=C_{\varepsilon_n,\Gamma}-2d_0.
			\end{split}
		\end{equation}
		
		It remains to consider $\tau_n<T_n/2$.  By \eqref{mlipschitz} and
		\eqref{flowproperties}, if $0\leq s<t\leq T_n$,
		\begin{align*}
			\left|m_{\varepsilon_n}(\eta_n(t,\widetilde E_{\varepsilon_n}(\theta)))
			-m_{\varepsilon_n}(\eta_n(s,\widetilde E_{\varepsilon_n}(\theta)))\right|
			&\leq
			\|\eta_n(t,\widetilde E_{\varepsilon_n}(\theta))
			-\eta_n(s,\widetilde E_{\varepsilon_n}(\theta))\|_{\varepsilon_n}
			\leq\int_s^t\|\dot\eta_n(r,\widetilde E_{\varepsilon_n}(\theta))\|_{\varepsilon_n}dr\\
			&\leq t-s.
		\end{align*}
		Then the function
		$t\mapsto m_{\varepsilon_n}
		(\eta_n(t,\widetilde E_{\varepsilon_n}(\theta)))$ is Lipschitz and so, for a.e. $t\in(0,T_n)$,
		\begin{equation}\label{bdddermen}
			\left|\frac d{dt}m_{\varepsilon_n}
			(\eta_n(t,\widetilde E_{\varepsilon_n}(\theta)))\right|\leq1.
		\end{equation}
		If $m_{\varepsilon_n}(\eta_n(t,\widetilde E_{\varepsilon_n}(\theta)))$ 
		reaches $3\rho/8$ at some time in $(\tau_n,T_n]$, denote its first
		hitting time after $\tau_n$ by $t_n^*$.  Then
		\(
		\widehat\tau_n:=\sup\{t\in[\tau_n,t_n^*]:
		m_{\varepsilon_n}(\eta_n(t,\widetilde E_{\varepsilon_n}(\theta)))
		=\rho/2\}
		\)
		is well defined and, for every $t\in(\widehat\tau_n,t_n^*)$,
		\[
		\frac{3\rho}{8}<m_{\varepsilon_n}
		(\eta_n(t,\widetilde E_{\varepsilon_n}(\theta)))<\frac\rho2.
		\]
		Thus, the definition of $\varphi_n$ and \eqref{bdddermen} give
		\begin{equation}\label{phi-integral-lower}
			\begin{aligned}
				\int_{\widehat\tau_n}^{t_n^*}\varphi_n(\eta_n(t,\widetilde E_{\varepsilon_n}(\theta)))dt
				&\geq-\int_{\widehat\tau_n}^{t_n^*}
				\left(\frac{4m_{\varepsilon_n}(\eta_n(t,\widetilde E_{\varepsilon_n}(\theta)))}{\rho}-1\right)
				\frac d{dt}m_{\varepsilon_n}(\eta_n(t,\widetilde E_{\varepsilon_n}(\theta)))dt\\
				&=\int_{3\rho/8}^{\rho/2}\left(\frac{4s}{\rho}-1\right)ds=\frac{3\rho}{32}.
			\end{aligned}
		\end{equation}
		Moreover, for a.e. \(t\in(\widehat\tau_n,t_n^*)\), Proposition
		\ref{4.5} gives
		$\|J_{\varepsilon_n}'
		(\eta_n(t,\widetilde E_{\varepsilon_n}(\theta)))\|_{\varepsilon_n}
		\geq\sigma_0$ and, since we are in the case in which the trajectory remains in
		$B_{\varepsilon_n,3\nu/2}$, 
		$\vartheta_n(\eta_n(t,\widetilde E_{\varepsilon_n}(\theta)))=1$ and
		$\Psi_n(\eta_n(t,\widetilde E_{\varepsilon_n}(\theta)))
		=\varphi_n(\eta_n(t,\widetilde E_{\varepsilon_n}(\theta)))$.
		Therefore, using the monotonicity of the energy along the flow, integrating
		\eqref{flowproperties} over \([\widehat\tau_n,t_n^*]\), and using
		\eqref{phi-integral-lower}, we obtain
		\begin{equation}\label{strict4}
			\begin{split}
				J_{\varepsilon_n}(\eta_n(T_n,\widetilde E_{\varepsilon_n}(\theta)))
				&\leq J_{\varepsilon_n}(\eta_n(t_n^*,\widetilde E_{\varepsilon_n}(\theta)))\\
				&\leq J_{\varepsilon_n}(\eta_n(\widehat\tau_n,\widetilde E_{\varepsilon_n}(\theta)))
				-\frac12\int_{\widehat\tau_n}^{t_n^*}
				\Psi_n(\eta_n(t,\widetilde E_{\varepsilon_n}(\theta)))
				\|J_{\varepsilon_n}'(\eta_n(t,\widetilde E_{\varepsilon_n}(\theta)))\|_{\varepsilon_n}dt\\
				&\leq J_{\varepsilon_n}(\eta_n(\widehat\tau_n,\widetilde E_{\varepsilon_n}(\theta)))
				-\frac{\sigma_0}{2}\int_{\widehat\tau_n}^{t_n^*}
				\varphi_n(\eta_n(t,\widetilde E_{\varepsilon_n}(\theta)))dt\\
				&\leq J_{\varepsilon_n}(\widetilde E_{\varepsilon_n}(\theta))
				-\frac{3\sigma_0\rho}{64}
				\leq C_{\varepsilon_n,\Gamma}-d_0.
			\end{split}
		\end{equation}
		If $m_{\varepsilon_n}(\eta_n(t,\widetilde E_{\varepsilon_n}(\theta)))>
		3\rho/8$ for every $t\in[\tau_n,T_n]$, then
		$\varphi_n(\eta_n(t,\widetilde E_{\varepsilon_n}(\theta)))\geq1/2$  and, since the trajectory remains in
		$B_{\varepsilon_n,3\nu/2}$, $\vartheta_n(\eta_n(t,\widetilde E_{\varepsilon_n}(\theta)))=1$, so that \(\Psi_n(\eta_n(t,\widetilde E_{\varepsilon_n}(\theta)))
		=\varphi_n(\eta_n(t,\widetilde E_{\varepsilon_n}(\theta)))\geq1/2\) for
		$\tau_n\leq t\leq T_n$.   Moreover \eqref{fixedgradientbound}
		and Proposition \ref{4.5} give
		$\|J_{\varepsilon_n}'
		(\eta_n(t,\widetilde E_{\varepsilon_n}(\theta)))\|_{\varepsilon_n}
		\geq\widetilde b_n$.
		Therefore \eqref{flowproperties} gives, for a.e.
		$t\in[\tau_n,T_n]$,
		\[
		\frac d{dt}J_{\varepsilon_n}(\eta_n(t,\widetilde E_{\varepsilon_n}(\theta)))
		\leq-\frac12\Psi_n(\eta_n(t,\widetilde E_{\varepsilon_n}(\theta)))
		\|J_{\varepsilon_n}'(\eta_n(t,\widetilde E_{\varepsilon_n}(\theta)))\|_{\varepsilon_n}
		\leq-\frac{\widetilde b_n}{4}
		\]
		and so, integrating this inequality over $[\tau_n,T_n]$ and since $\tau_n<T_n/2$, we obtain
		\begin{equation}\label{strict5}
			J_{\varepsilon_n}(\eta_n(T_n,\widetilde E_{\varepsilon_n}(\theta)))
			\leq J_{\varepsilon_n}(\eta_n(\tau_n,\widetilde E_{\varepsilon_n}(\theta)))
			-\frac{\widetilde b_n}{4}(T_n-\tau_n)
			\leq J_{\varepsilon_n}(\widetilde E_{\varepsilon_n}(\theta))
			-
			\frac{\widetilde b_n}{8}T_n
			\leq C_{\varepsilon_n,\Gamma}-d_0
			.
		\end{equation}
		
		Then, combining \eqref{strict}, \eqref{strict1}, \eqref{strict2}, \eqref{strict3}, \eqref{strict4}, \eqref{strict5}, we obtain \eqref{claim1}.
		
		Finally, for \(\theta\in\partial([T^{-2},1]^s)\),
		\eqref{boundaryoutsideB} gives
		\(\Psi_n(\widetilde E_{\varepsilon_n}(\theta))=0\). Hence
		\(H_n(\widetilde E_{\varepsilon_n}(\theta))=0\) and so the constant curve
		\(t\mapsto\widetilde E_{\varepsilon_n}(\theta)\) solves the flow equation
		with this initial datum.  By uniqueness,
		\[
		\eta_n(T_n,\widetilde E_{\varepsilon_n}(\theta))
		=\widetilde E_{\varepsilon_n}(\theta),
		\qquad \theta\in\partial([T^{-2},1]^s).
		\]
		Moreover, \eqref{pathlocalmass} and \eqref{etanb} imply that, for every \(\theta\in[T^{-2},1]^s\),
		\(m_{\varepsilon_n}(\eta_n(T_n,\widetilde E_{\varepsilon_n}(\theta)))
		\geq\rho/4\) and, since $\zeta_{j,\varepsilon_n}=1$ on
		$\Lambda_{j,\varepsilon_n}$, for every \(j\in\Gamma\),
		\(\|\zeta_{j,\varepsilon_n}\eta_n(T_n,\widetilde E_{\varepsilon_n}(\theta))
		\|_{\varepsilon_n}>0\).
		Thus the map
		\(\theta\in[T^{-2},1]^s\mapsto\eta_n(T_n,\widetilde E_{\varepsilon_n}(\theta))\) belongs to
		\(\Sigma_{\varepsilon_n}\). Hence Proposition \ref{4.3}, the definition of
		$S_{\varepsilon_n}$, and \eqref{claim1} now give
		\[
		\sum_{j=1}^sc_{V_j}
		=\lim_{n\to+\infty}S_{\varepsilon_n}
		\leq\limsup_{n\to+\infty}
		\max_{\theta\in[T^{-2},1]^s}
		J_{\varepsilon_n}(\eta_n(T_n,\widetilde E_{\varepsilon_n}(\theta)))
		\leq\sum_{j=1}^sc_{V_j}-\xi
		\]
		and this contradiction allows us to complete the proof.
	\end{proof}

	\section{Proof of Theorem \ref{1.1}}\label{Sec4}
	In this section, we prove Theorem \ref{1.1} by showing that the solutions
	of the modified problem obtained in Proposition \ref{4.6} solve the
	original problem and possess the asserted concentration and decay properties.
	
	Fix \(\nu\in(0,\nu^*/2)\) and, for every sufficiently small
	\(\varepsilon>0\), 
	let \(u_\varepsilon\) be a critical point provided by
	Proposition \ref{4.6}. Let us work along an arbitrary sequence
	\(\varepsilon_n\to0^+\) and write
	\(u_n=u_{\varepsilon_n}\). Then
	\[
	u_n\in B_{\varepsilon_n,\nu}\cap
	J_{\varepsilon_n}^{C_{\varepsilon_n,\Gamma}}\cap M_{\varepsilon_n}(\rho/2),
	\qquad J_{\varepsilon_n}'(u_n)=0.
	\]
	Passing to a subsequence and arguing as in the proof of Proposition \ref{4.5}, we may assume that
	\[
	J_{\varepsilon_n}(u_n)\to c
	\in
	\left(0,\frac18\left(\frac12-\frac1{\Theta}-\frac1{2k}\right)
	\right).
	\]
	Indeed, the arguments in \eqref{bddPS} show
	that \(\{\|u_n\|_{\varepsilon_n}\}\) is bounded and that \(c\geq0\).
	If \(c=0\), then \eqref{bddPS} and
	\(J_{\varepsilon_n}'(u_n)=0\) imply
	\(\|u_n\|_{\varepsilon_n}\to0\), whereas
	\(
	\rho/2\leq m_{\varepsilon_n}(u_n)
	\leq\|u_n\|_{\varepsilon_n}
	\).
	Then, Remark \ref{4.4} and \eqref{p3.1} allow us to conclude.
	Therefore we can apply Proposition
	\ref{fengjie} and so, up to a subsequence, there exist a positive integer
	\(l\), sequences \(\{y_{n,h}\}_{h=1}^l\subset\mathbb R^2\), and nontrivial
	solutions \(U_{0,h}\) of the corresponding limit problems such that
	\(
	\varepsilon_ny_{n,h}\to y_h\in\overline\Lambda_\Gamma,\) 
	\(|y_{n,h}-y_{n,k}|\to+\infty\,\ (h\ne k),
	\)
	\begin{equation}\label{profileconv}
		\left\|u_n-\sum_{h=1}^l e^{iA(y_h)\cdot(\cdot-y_{n,h})}
		(U_{0,h}\chi_{\varepsilon_n})(\cdot-y_{n,h})
		\right\|_{\varepsilon_n}\to0,
	\end{equation}
	and
	\begin{equation*}
		c=\sum_{h=1}^lJ_{0,h}(U_{0,h}).
	\end{equation*}
	
	Repeating the argument in the proof of Proposition \ref{4.5} and using
	\(m_{\varepsilon_n}(u_n)\geq\rho/2\), every prescribed well contains at least
	one profile,  \(l= s\), each \(y_j\in\overline\Lambda_j\) for \(j=1,\ldots,s\),
	\begin{equation}\label{exactprofilelevel}
		c=\sum_{j=1}^sc_{V_j},
	\end{equation}
	and, for every \(j\in\Gamma\), \(y_j\in K_j\) and \(U_{0,j}\) is a ground
	state solution of \eqref{fude} with \(\mu=V_j\).
	
	Let us prove now \(L^\infty\) and decay estimates. 
	
	Let $\eta\in C^\infty(\mathbb{R}^2,\mathbb{R})$ be such that
	$0\leq\eta\leq1$ and \(|\nabla\eta|_\infty<+\infty\). Moreover, for any
	\(n\in\mathbb{N}\) and \(L>0\), let
	\[
	u_{L,n}(x):=
	\begin{cases}
		|u_n(x)|, & \text{if } |u_n(x)|\le L,\\
		L, & \text{if } |u_n(x)|>L,
	\end{cases}
	\qquad
	z_{L,n}:=\eta^{2}u_{L,n}^{2(\beta-1)}u_n,
	\quad \text{and}\quad
	u^*_{L,n}:=\eta u_{L,n}^{\beta-1}|u_n|,
	\]
	where $\beta\geq 1$ will be determined later.\footnote{For
		\(\beta=1\), we consider \(u_{L,n}^{2(\beta-1)}\equiv1\).}
	Observe that \(u^*_{L,n}\in H^1(\mathbb{R}^2,\mathbb{R})\). Indeed, since
	\(|u^*_{L,n}|\leq L^{\beta-1}|u_n|\in
	L^2(\mathbb{R}^2,\mathbb{R})\), we have
	\(u^*_{L,n}\in L^2(\mathbb{R}^2,\mathbb{R})\). Moreover,
	\[
	|\nabla(u_{L,n}^{\beta-1}|u_n|)|(x)=
	\begin{cases}
		\beta |u_n|^{\beta-1}|\nabla |u_n||(x),
		&\text{if } |u_n(x)|\le L,\\
		L^{\beta-1}|\nabla |u_n||(x),
		&\text{if } |u_n(x)|>L,
	\end{cases}
	\leq \beta L^{\beta-1}|\nabla |u_n||(x)\in L^2(\mathbb{R}^2, \mathbb{R}),
	\]
	we also get that
	\(
	|\nabla u^*_{L,n}|
	=|\eta\nabla(u_{L,n}^{\beta-1}|u_n|)
	+u_{L,n}^{\beta-1}|u_n|\nabla\eta|
	\in L^2(\mathbb{R}^2,\mathbb{R}).
	\)
	Thus \(u^*_{L,n}\in H^1(\mathbb{R}^2,\mathbb{R})\).
	Furthermore \(z_{L,n}\in H_{\varepsilon_n}\). Indeed, since \(|z_{L,n}|\leq L^{2(\beta-1)}|u_n|\), we have
	\[
	\int_{\mathbb R^2}V_{\varepsilon_n}(x)|z_{L,n}|^2dx
	\leq
	L^{4(\beta-1)}
	\int_{\mathbb R^2}V_{\varepsilon_n}(x)|u_n|^2dx
	<+\infty.
	\]
	In particular, since \(V_{\varepsilon_n}\geq V_0>0\),
	\(z_{L,n}\in L^2(\mathbb R^2,\mathbb C)\).
	Moreover, the chain rule and the diamagnetic inequality give
	\[
	|\nabla_{A_{\varepsilon_n}}z_{L,n}|
	\leq
	L^{2(\beta-1)}|\nabla_{A_{\varepsilon_n}}u_n|
	+2L^{2(\beta-1)}|\nabla\eta|\,|u_n|
	+2(\beta-1)L^{2(\beta-1)}|\nabla|u_n||
	\in L^2(\mathbb R^2).
	\]
	In addition, by the diamagnetic inequality \eqref{p2.1} and
	\(\nabla |u_n|\cdot\nabla u_{L,n}\geq0\), we have
	\begin{equation}\label{testzLn}
		\begin{split}
			\operatorname{Re}\left(
			\nabla_{A_{\varepsilon_n}}u_n\cdot
			\overline{\nabla_{A_{\varepsilon_n}}z_{L,n}}
			\right)
			&=
			\eta^2u_{L,n}^{2(\beta-1)}
			|\nabla_{A_{\varepsilon_n}}u_n|^2
			+\operatorname{Re}(\overline{u_n}\nabla u_n)
			\nabla(\eta^2u_{L,n}^{2(\beta-1)})\\
			&=
			\eta^2u_{L,n}^{2(\beta-1)}
			|\nabla_{A_{\varepsilon_n}}u_n|^2
			+|u_n|\nabla |u_n|
			\nabla(\eta^2u_{L,n}^{2(\beta-1)})\\
			&\geq
			\eta^2u_{L,n}^{2(\beta-1)}|\nabla |u_n||^2
			+2\eta\nabla\eta\,u_{L,n}^{2(\beta-1)}
			|u_n|\nabla |u_n|
		\end{split}
	\end{equation}
	and, by Young's inequality,
	for every \(\delta>0\), there exists $C_\delta>0$ such that
	\begin{equation*}
		2\eta|\nabla\eta|u_{L,n}^{2(\beta-1)}|u_n||\nabla |u_n||
		\leq\delta\eta^2u_{L,n}^{2(\beta-1)}|\nabla |u_n||^2
		+C_\delta|\nabla\eta|^2u_{L,n}^{2(\beta-1)}|u_n|^2.
	\end{equation*}
	
	Since \(u_n\) is a solution of \eqref{p3.2}, taking
	\(z_{L,n}\) as a test function, 
	by \eqref{testzLn}, \eqref{g3}, and \eqref{p2.4}, we get that for a fixed $q>2$, for any \(0<\xi<V_0\) and $\alpha>4\pi$, there exists a constant \(C>0\),
	independent of \(L,\beta\), and \(n\), such that
	\begin{align*}
		\int_{\mathbb{R}^2}\eta^2u_{L,n}^{2(\beta-1)}
		|\nabla |u_n||^2dx
		&\leq
		\int_{\mathbb{R}^2}\eta^2u_{L,n}^{2(\beta-1)}
		|\nabla |u_n||^2dx
		+2\int_{\mathbb{R}^2}
		\eta\nabla\eta\,u_{L,n}^{2(\beta-1)}
		|u_n|\nabla |u_n|dx\\
		&\quad
		+2\int_{\mathbb{R}^2}
		\eta|\nabla\eta|\,u_{L,n}^{2(\beta-1)}
		|u_n||\nabla |u_n||dx\\
		&\quad
		+\int_{\mathbb{R}^2}
		V_{\varepsilon_n}(x)\eta^2u_{L,n}^{2(\beta-1)}|u_n|^2dx
		-\xi\int_{\mathbb{R}^2}
		\eta^2u_{L,n}^{2(\beta-1)}|u_n|^2dx\\
		&\leq
		\operatorname{Re}\int_{\mathbb{R}^2}
		\nabla_{A_{\varepsilon_n}}u_n\cdot
		\overline{\nabla_{A_{\varepsilon_n}}z_{L,n}}\,dx
		+\delta\int_{\mathbb{R}^2}
		\eta^2u_{L,n}^{2(\beta-1)}|\nabla |u_n||^2dx\\
		&\quad
		+C_\delta\int_{\mathbb{R}^2}
		|\nabla\eta|^2u_{L,n}^{2(\beta-1)}|u_n|^2dx
		+\operatorname{Re}\int_{\mathbb{R}^2}
		V_{\varepsilon_n}(x)u_n\overline{z_{L,n}}dx\\
		&\quad
		-\xi\int_{\mathbb{R}^2}
		\eta^2u_{L,n}^{2(\beta-1)}|u_n|^2dx\\
		&=
		\int_{\mathbb{R}^2}
		g(\varepsilon_nx,|u_n|^2)\eta^2u_{L,n}^{2(\beta-1)}
		|u_n|^2dx
		+\delta\int_{\mathbb{R}^2}
		\eta^2u_{L,n}^{2(\beta-1)}|\nabla |u_n||^2dx\\
		&\quad
		+C_\delta\int_{\mathbb{R}^2}
		|\nabla\eta|^2u_{L,n}^{2(\beta-1)}|u_n|^2dx
		-\xi\int_{\mathbb{R}^2}
		\eta^2u_{L,n}^{2(\beta-1)}|u_n|^2dx\\
		&\leq
		C\int_{\mathbb{R}^2}
		|u_n|^q\left(e^{\alpha|u_n|^2}-1\right)
		\eta^2u_{L,n}^{2(\beta-1)}dx
		+\delta\int_{\mathbb{R}^2}
		\eta^2u_{L,n}^{2(\beta-1)}|\nabla |u_n||^2dx\\
		&\quad
		+C_\delta\int_{\mathbb{R}^2}|\nabla\eta|^2|u_n|^{2\beta}dx .
	\end{align*}
	Hence, for \(\delta>0\) small enough,
	\begin{equation}\label{moser-grad}
		\int_{\mathbb{R}^2}\eta^2u_{L,n}^{2(\beta-1)}|\nabla |u_n||^2dx
		\leq C\left[\int_{\operatorname{supp}\eta}|u_n|^{q+2(\beta-1)}
		\left(e^{\alpha |u_n|^2}-1\right)dx+\int_{\mathbb{R}^2}|\nabla\eta|^2|u_n|^{2\beta}dx\right]
	\end{equation}
	and, since \(0<\xi<V_0\), we also get
	\begin{equation}\label{moser-12}
		\begin{split}
			\int_{\mathbb{R}^2}\eta^2u_{L,n}^{2(\beta-1)}|u_n|^2dx
			&\leq
			\frac{1}{V_0-\xi}
			\left[\int_{\mathbb{R}^2}V_{\varepsilon_n}(x)\eta^2u_{L,n}^{2(\beta-1)}|u_n|^2dx
			-\xi\int_{\mathbb{R}^2}\eta^2u_{L,n}^{2(\beta-1)}|u_n|^2dx\right]\\
			&\leq C\left[\int_{\operatorname{supp}\eta}|u_n|^{q+2(\beta-1)}
			\left(e^{\alpha |u_n|^2}-1\right)dx+\int_{\mathbb{R}^2}|\nabla\eta|^2|u_n|^{2\beta}dx\right].
		\end{split}
	\end{equation}
	
	Moreover, arguing as in \eqref{bddPS} and using the
	diamagnetic inequality, \eqref{exactprofilelevel}, and
	\eqref{p3.1},
	\[
	\limsup_{n\to+\infty}
	\bigl(|\nabla|u_n||_2^2+V_0|u_n|_2^2\bigr)
	\leq
	\limsup_{n\to+\infty}\|u_n\|_{\varepsilon_n}^2
	\leq
	\left(\frac12-\frac1{\Theta}-\frac1{2k}\right)^{-1}
	\sum_{j=1}^s c_{V_j}
	<\frac18.
	\]
	Then, if $\sigma\in (1/(2\sqrt{2}),1)$, without loss of generality we may assume that for every $n\in\mathbb{N}$,
	\begin{equation}\label{sigma}
		|\nabla|u_n||_2^2+V_0|u_n|_2^2\leq \sigma^2  <1.
	\end{equation}
	
	Let $q>2$ and $\alpha\in(4\pi,4\pi/\sigma^2)$. Take $\tau_0\in(1,4\pi/(\sigma^2\alpha))$
	and $\tau_1,\tau_2>1$ such that $1/\tau_0+1/\tau_1+1/\tau_2=1$ and $\tau_1(q-2)> 2$.
	Using the Sobolev embedding, \eqref{moser-grad}, \eqref{moser-12}, the
	H\"older inequality and \eqref{e} for such values, we have
	\begin{align*}
		\left|u^*_{L,n}\right|_q^2
		&\leq C\int_{\mathbb{R}^2}\left(|\nabla u^*_{L,n}|^2+|u^*_{L,n}|^2\right)dx\\
		&\leq C\left[\int_{\mathbb{R}^2}|\nabla\eta|^2|u_n|^{2\beta}dx
		+\beta^2\int_{\mathbb{R}^2}\eta^2u_{L,n}^{2(\beta-1)}|\nabla |u_n||^2dx
		+\int_{\mathbb{R}^2}\eta^2u_{L,n}^{2(\beta-1)}|u_n|^2dx\right]\\
		&\leq C\beta^2\left[\int_{\mathbb{R}^2}|\nabla\eta|^2|u_n|^{2\beta}dx
		+\int_{\operatorname{supp}\eta}|u_n|^{q+2(\beta-1)}
		\left(e^{\alpha |u_n|^2}-1\right)dx\right]\\
		&\leq C\beta^2\biggl[\int_{\mathbb{R}^2}|\nabla\eta|^2|u_n|^{2\beta}dx\\
		&\qquad\qquad
		+\left(\int_{\operatorname{supp}\eta}|u_n|^{\tau_1(q-2)}dx\right)^{\frac{1}{\tau_1}}
		\left(\int_{\mathbb{R}^2}\left(e^{\tau_0\sigma^2\alpha (|u_n|/\sigma)^2}-1\right)dx\right)^{\frac{1}{\tau_0}}
		\left(\int_{\operatorname{supp}\eta}|u_n|^{2\beta \tau_2}dx\right)^{\frac{1}{\tau_2}}\biggr].
	\end{align*}
	Then, \eqref{sigma} and Lemma \ref{2.1} give the uniform estimate
	\begin{equation}\label{exp-unif}
		\sup_n\int_{\mathbb R^2}
		\left(e^{\tau_0\sigma^2\alpha(|u_n|/\sigma)^2}-1\right)dx
		\leq C.
	\end{equation}
	Hence, by the Sobolev embedding,
	\begin{equation}\label{moser-star}
		\left|u^*_{L,n}\right|_q^2
		\leq C\beta^2\left[\int_{\mathbb{R}^2}|\nabla\eta|^2|u_n|^{2\beta}dx
		+\left(\int_{\operatorname{supp}\eta}|u_n|^{2\beta \tau_2} dx\right)^{\frac{1}{\tau_2}}\right].
	\end{equation}
	
	Let us show the uniform \(L^\infty\) bound.\\
	Let $\zeta:=q/(2\tau_2)$ and $\beta:=\zeta^h$ with $h\in\mathbb{N}$.
	To have $\zeta> 1$, let us take $q>8\pi/(4\pi-\sigma^2\alpha)$. Indeed, with this choice, take
	\[
	1<\frac{q}{q-2}<\tau_0<\frac{4\pi}{\sigma^2\alpha}
	\text{ and }
	\tau_1>\frac{\tau_0 q}{\tau_0(q-2)-q}
	\]
	and so
	\[
	\frac{1}{\tau_2}
	=
	1-\frac{1}{\tau_0}-\frac{1}{\tau_1}
	>1-\frac{1}{\tau_0}-\frac{q-2}{q}+\frac{1}{\tau_0}
	=\frac{2}{q}.
	\]
	Fix \(x_0\in\mathbb R^2\), \(N\in\mathbb N\), and, for each
	\(h=0,\ldots,N\), choose a cut-off function
	\(\eta_h\in C_c^\infty(B_{1+2^{-h}}(x_0))\) satisfying
	\(0\leq\eta_h\leq1\), \(\eta_h=1\) in
	\(B_{1+2^{-(h+1)}}(x_0)\), and \(|\nabla\eta_h|\leq C2^h\) where $C>0$ is independent of $h$. Since
	\(\operatorname{supp}\nabla\eta_h
	\subset
	\overline{B_{1+2^{-h}}(x_0)\setminus
		B_{1+2^{-(h+1)}}(x_0)}\), by H\"older's
	inequality,
	\begin{equation}\label{Ch}
		\begin{split}
			\int_{\mathbb{R}^2}|\nabla\eta_h|^2|u_n|^{2\zeta^h}dx
			&\leq C\bigg(\frac{3+4\cdot 2^h}{2^{2h}}\bigg)^{1-\frac{1}{\tau_2}}2^{2h}\left(\int_{B_{1+2^{-h}}(x_0)}|u_n|^{2\tau_2\zeta^h}dx\right)^{1/\tau_2}\\
			&\leq C2^{2h}\left(\int_{B_{1+2^{-h}}(x_0)}|u_n|^{2\tau_2\zeta^h}dx\right)^{1/\tau_2},
		\end{split}
	\end{equation}
	being the sequence $\{(3+4\cdot 2^h)2^{-2h}\}$ decreasing.
	Moreover, since \(\eta_h=1\) in \(B_{1+2^{-(h+1)}}(x_0)\) and
	\(u_{L,n}^{\beta-1}|u_n|\geq u_{L,n}^{\beta}\), by \eqref{moser-star} and \eqref{Ch},
	\[
	\left(\int_{B_{1+2^{-(h+1)}}(x_0)}u_{L,n}^{q\zeta^h}dx\right)^{2/q}
	\leq \left|u^*_{L,n}\right|_q^2
	\leq C\zeta^{2h}\left(1+2^{2h}\right)
	\left(\int_{B_{1+2^{-h}}(x_0)}|u_n|^{2\tau_2\zeta^h}dx\right)^{1/\tau_2}.
	\]
	Then, letting \(L\to+\infty\) and applying Fatou's lemma, we obtain
	\[
	\left(\int_{B_{1+2^{-(h+1)}}(x_0)}|u_n|^{q\zeta^h}dx\right)^{2/q}
	\leq C\zeta^{2h}\left(1+2^{2h}\right)
	\left(\int_{B_{1+2^{-h}}(x_0)}|u_n|^{2\tau_2\zeta^h}dx\right)^{1/\tau_2}.
	\]
	Since \(q\zeta^h=2\tau_2\zeta^{h+1}\), we get
	\[
	\begin{aligned}
		\left(\int_{B_{1+2^{-(h+1)}}(x_0)}
		|u_n|^{2\tau_2\zeta^{h+1}}dx\right)^{\frac{1}{2\tau_2\zeta^{h+1}}}
		&\leq C^{\zeta^{-h}}\zeta^{h\zeta^{-h}}
		\left(1+2^{2h}\right)^{\frac{1}{2\zeta^h}}
		\biggl(\int_{B_{1+2^{-h}}(x_0)}|u_n|^{2\tau_2\zeta^h}dx\biggr)^{\frac{1}{2\tau_2\zeta^h}} .
	\end{aligned}
	\]
	Iterating the previous inequality, we obtain
	\[
	\left(\int_{B_{1+2^{-(N+1)}}(x_0)}
	|u_n|^{2\tau_2\zeta^{N+1}}dx\right)^{\frac{1}{2\tau_2\zeta^{N+1}}}
	\leq C^{\sum_{h=0}^N\zeta^{-h}}
	\zeta^{\sum_{h=0}^N h\zeta^{-h}}
	e^{\sum_{h=0}^N\frac{\log(1+2^{2h})}{2\zeta^h}}\left(\int_{B_2(x_0)}|u_n|^{2\tau_2} dx\right)^{\frac{1}{2\tau_2}}
	\]
	and, passing to the limit as \(N\to+\infty\) and using that all the series in the
	last line are convergent, we conclude that
	\begin{equation*}
		|u_n|_{L^\infty(B_1(x_0))}
		\leq C\left(\int_{B_2(x_0)}|u_n|^{2\tau_2} dx\right)^{\frac{1}{2\tau_2}}.
	\end{equation*}
	Thus, by \eqref{p2.1}, \(V_0>0\), the boundedness of
	\(\|u_n\|_{\varepsilon_n}\), and
	the Sobolev embedding we get
	\[
	|u_n|_{L^\infty(B_1(x_0))}
	\leq C|u_n|_{L^{2\tau_2}(B_2(x_0))}
	\leq C|u_n|_{2\tau_2}\leq C,
	\]
	uniformly in \(x_0\) and \(n\) and so, covering \(\mathbb R^2\) by countably many
	unit balls, we obtain a constant \(C_\infty>0\) such that
	\begin{equation}\label{uniformLinfty}
		|u_n|_\infty\leq C_\infty .
	\end{equation}
	
	Let us prove the decay away from the concentration points.

	For every \(R>0\), set
	\(
	\Omega_{n,R}:=\mathbb{R}^2\setminus
	\bigcup_{j\in\Gamma}B_R(y_{n,j})
	\)
	and
	\(\Psi_{n,j}:=e^{iA(y_j)\cdot(\cdot-y_{n,j})}
	(U_{0,j}\chi_{\varepsilon_n})(\cdot-y_{n,j})\).
	We have
	\[
	\int_{\Omega_{n,R}}\left(|\nabla_{A_{\varepsilon_n}}u_n|^2
	+V_{\varepsilon_n}(x)|u_n|^2\right)dx
	\leq2\left\|u_n-\sum_{j\in\Gamma}\Psi_{n,j}\right\|_{\varepsilon_n}^2
	+C\sum_{j\in\Gamma}\int_{\Omega_{n,R}}
	\left(|\nabla_{A_{\varepsilon_n}}\Psi_{n,j}|^2+V_{\varepsilon_n}(x)|\Psi_{n,j}|^2\right)dx 
	\]
	where \(C\) is independent of \(n\) and \(R\).
	For \(j\in\Gamma\), set
	\(z_j=x-y_{n,j}\), so that \(|z_j|\geq R\) when \(x\in\Omega_{n,R}\).
	Moreover, if \(z_j\in\operatorname{supp}\chi_{\varepsilon_n}\), then
	\(|z_j|\leq2|\log\varepsilon_n|\) and hence
	\(\varepsilon_n(z_j+y_{n,j})\to y_j\) uniformly on this support.
	Therefore \(A(\varepsilon_n(z_j+y_{n,j}))\) and
	\(V(\varepsilon_n(z_j+y_{n,j}))\) are uniformly bounded there.
	Using also
	\(|\nabla\chi_{\varepsilon_n}|\leq C/|\log\varepsilon_n|\), we obtain
	\[
	\begin{aligned}
		&\int_{\Omega_{n,R}}\left(|\nabla_{A_{\varepsilon_n}}\Psi_{n,j}|^2+V_{\varepsilon_n}(x)|\Psi_{n,j}|^2\right)dx\\
		&=\int_{\Omega_{n,R}-y_{n,j}}\left[
		\left|\left(\frac1i\nabla-A(\varepsilon_n(z_j+y_{n,j}))+A(y_j)\right)
		(U_{0,j}\chi_{\varepsilon_n})(z_j)\right|^2 
		+V(\varepsilon_n(z_j+y_{n,j}))
		|(U_{0,j}\chi_{\varepsilon_n})(z_j)|^2\right]dz_j\\
		&\leq C\|U_{0,j}\chi_{\varepsilon_n}\|_{H^1(B_R^c(0),\mathbb C)}^2
		\leq C\|U_{0,j}\|_{H^1(B_R^c(0),\mathbb C)}^2+o_n(1).
	\end{aligned}
	\]
	Thus, since \(U_{0,j}\in H^1(\mathbb{R}^2,\mathbb{C})\) and \(\Gamma\) is finite, using \eqref{profileconv}, we
	deduce that
	\begin{equation}\label{tail-H}
		\lim_{R\to+\infty}\limsup_{n\to+\infty}
		\int_{\Omega_{n,R}}\left(|\nabla_{A_{\varepsilon_n}}u_n|^2
		+V_{\varepsilon_n}(x)|u_n|^2\right)dx=0.
	\end{equation}

	For any \(R>0\) and \(0<r\leq R/2\), take a cut-off function
	\(\eta\in C^\infty(\mathbb{R}^2,\mathbb{R})\) such that
	\(0\leq\eta\leq1\), \(\eta=1\) in \(\Omega_{n,R}\),
	\(\eta=0\) in \(\bigcup_{j\in\Gamma}B_{R-r}(y_{n,j})\), and
	\(|\nabla\eta|\leq2/r\). Then
	\[
	\operatorname{supp}\nabla\eta\subset
	\bigcup_{j\in\Gamma}\left(B_R(y_{n,j})\setminus B_{R-r}(y_{n,j})\right),
	\]
	and so, since the number of points \(y_{n,j}\) is fixed, $|\operatorname{supp}\nabla\eta|
	\leq |\Gamma|\pi\bigl(R^2-(R-r)^2\bigr)\leq CRr\leq CR^2$.
	Hence, by H\"older's
	inequality,
	\[
	\int_{\mathbb{R}^2}|\nabla\eta|^2|u_n|^{2\beta}dx
	\leq \frac{C}{r^2}\left|\operatorname{supp}\nabla\eta\right|^{1-\frac{1}{\tau_2}}
	\left(\int_{\Omega_{n,R-r}}|u_n|^{2\beta \tau_2}dx\right)^{\frac{1}{\tau_2}}
	\leq C\frac{R^\frac{2(\tau_2-1)}{\tau_2}}{r^2}
	\left(\int_{\Omega_{n,R-r}}|u_n|^{2\beta \tau_2}dx\right)^{\frac{1}{\tau_2}}.
	\]
	Moreover, since $\tau_1(q-2)>2$ and \(\{\|u_n\|_{\varepsilon_n}\}\) is bounded, using the diamagnetic inequality and the Sobolev embeddings, there exists \(C\) independent of \(n,R,r\) such that
	\begin{equation*}
		\left(\int_{\operatorname{supp}\eta}|u_n|^{\tau_1(q-2)}\,dx\right)^{1/\tau_1}
		\leq |u_n|_{\tau_1(q-2)}^{q-2}
		\leq C.
	\end{equation*}
	Hence, since \(\eta=1\) in \(\Omega_{n,R}\), by \eqref{moser-star},
	\begin{align*}
		\left(\int_{\Omega_{n,R}}u_{L,n}^{q\beta}dx\right)^{2/q}
		&\leq \left|u^*_{L,n}\right|_q^2
		\leq C\beta^2\left[
		\int_{\mathbb R^2}|\nabla\eta|^2|u_n|^{2\beta}\,dx
		+\left(\int_{\Omega_{n,R-r}}|u_n|^{2\beta\tau_2}\,dx\right)^{1/\tau_2}
		\right]\\
		&\leq
		C\beta^2\left(1+\frac{R^\frac{2(\tau_2-1)}{\tau_2}}{r^2}\right)
		\left(\int_{\Omega_{n,R-r}}|u_n|^{2\beta \tau_2}dx\right)^{\frac{1}{\tau_2}}.
	\end{align*}
	Then, applying Fatou's Lemma as \(L\to+\infty\), we obtain
	\[
	\left(\int_{\Omega_{n,R}}|u_n|^{q\beta}dx\right)^{2/q}
	\leq
	C\beta^2\left(1+\frac{R^\frac{2(\tau_2-1)}{\tau_2}}{r^2}\right)
	\left(\int_{\Omega_{n,R-r}}|u_n|^{2\beta \tau_2}dx\right)^{\frac{1}{\tau_2}}
	\]
	and, taking \(\beta=\zeta^h\) as before (with $\zeta=q/(2\tau_2)$), we obtain
	\[
	\left(\int_{\Omega_{n,R}}|u_n|^{2\tau_2\zeta^{h+1}}dx\right)^{\frac{1}{2\tau_2\zeta^{h+1}}}
	\leq C^{\zeta^{-h}}\zeta^{h\zeta^{-h}}
	\left(1+\frac{R^\frac{2(\tau_2-1)}{\tau_2}}{r^2}\right)^{\frac{1}{2\zeta^h}}
	\left(\int_{\Omega_{n,R-r}}|u_n|^{2\tau_2\zeta^h}dx\right)^{\frac{1}{2\tau_2\zeta^h}}.
	\]
	Fix \(R\geq1\) and \(N\in\mathbb N^*\), and apply this estimate for
	\(h=1,\ldots,N\), with left-hand radius \(R-R/2^{h+1}\) and
	\(r=R/2^{h+1}\), so that, at the \(h\)-th step the right-hand
	radius is $R-R/2^h$. Moreover, since \(R\geq1\),
	\(1+(R-R/2^{h+1})^{2(\tau_2-1)/\tau_2}/(R/2^{h+1})^2
	\leq1+2^{2(h+1)}R^{-2/\tau_2}\leq1+2^{2(h+1)}\).
	Hence
	\[
	\left(\int_{\Omega_{n,R-R/2^{N+1}}}
	|u_n|^{2\tau_2\zeta^{N+1}}dx\right)^{\frac{1}{2\tau_2\zeta^{N+1}}}
	\leq C^{\sum_{h=1}^N\zeta^{-h}}
	\zeta^{\sum_{h=1}^N h\zeta^{-h}}
	e^{\sum_{h=1}^N\frac{\log(1+2^{2(h+1)})}{2\zeta^h}}
	\left(\int_{\Omega_{n,R/2}}|u_n|^{2\tau_2\zeta}dx\right)^{\frac{1}{2\tau_2\zeta}}.
	\]
	Passing to the limit as \(N\to+\infty\), we conclude that
	\[
	|u_n|_{L^\infty(\Omega_{n,R})}
	\leq C\left(\int_{\Omega_{n,R/2}}|u_n|^{q}dx\right)^{\frac{1}{q}}
	\]
	and, since by \eqref{uniformLinfty}, \(q>2\), and
	\(V_{\varepsilon_n}\geq V_0\),
	\[
	\int_{\Omega_{n,R/2}}|u_n|^{q}dx
	\leq |u_n|_\infty^{q-2}\int_{\Omega_{n,R/2}}|u_n|^2dx
	\leq \frac{C}{V_0}\int_{\Omega_{n,R/2}}V_{\varepsilon_n}(x)|u_n|^2dx,
	\]
	we obtain that there exists \(C>0\) independent of \(R\) and \(n\) such that
	\[
	|u_n|_{L^\infty(\Omega_{n,R})}
	\leq C\left(\int_{\Omega_{n,R/2}}V_{\varepsilon_n}(x)|u_n|^2dx\right)^{\frac{1}{2\tau_2\zeta}}.
	\]
	Then, by \eqref{tail-H}, we conclude that
	\begin{equation}\label{uniformDecayOutside}
		\lim_{R\to+\infty}\limsup_{n\to+\infty}
		|u_n|_{L^\infty(\mathbb{R}^2\setminus\bigcup_{j\in\Gamma}B_R(y_{n,j}))}=0.
	\end{equation}
	
	Since \(y_j\in K_j\), for every fixed \(R>0\) and \(n\)
	sufficiently large we have \(B_R(y_{n,j})\subset\Lambda_{j,\varepsilon_n}\).
	Taking \(R\) large in \eqref{uniformDecayOutside}, we get \(|u_n|^2\leq a\) a.e. in \(\Lambda_{\Gamma,\varepsilon_n}^c\) for \(n\) large. Thus, by Remark \ref{remark}, \(u_n\) is also a
	solution of the original problem \eqref{origi} with \(\varepsilon=\varepsilon_n\). Then, due to the arbitrariness of the sequence \(\varepsilon_n\to0^+\), we get that, for every \(\varepsilon>0\) sufficiently small, \(u_\varepsilon\) solves \eqref{origi}.

	Let \(k>1\) be the number chosen in \eqref{p3.1} and fix \(\gamma\in(0,a)\). The properties of $f$ and the definition of $a$ in Section \ref{Secthree} imply that
	\begin{equation}\label{gg}
		f(t^2)\leq \frac{V_{0}}{k}, \quad \text{for }|t|\leq\sqrt{\gamma}.
	\end{equation}
	
	We claim that there exist \(R>0\) and \(\varepsilon_*>0\) such that,
	for every \(0<\varepsilon<\varepsilon_*\), one can choose centers
	\(q_{\varepsilon,j}\in\mathbb{R}^2\) with \(\varepsilon q_{\varepsilon,j}\in\Lambda_j\)
	for which
	\begin{equation}\label{p5.7}
		\left|u_\varepsilon\right|_{L^\infty\left(\mathbb R^2\setminus
			\bigcup_{j\in\Gamma}B_R(q_{\varepsilon,j})\right)}^2
		\leq\frac\gamma2<\gamma.
	\end{equation}
	Indeed, if we suppose that this were false, for every positive integer \(n\), there would exist \(\varepsilon_n\in(0,1/n)\) such that, for every choice of
	centers \(q_j\) satisfying \(\varepsilon_nq_j\in\Lambda_j\),
	\[
	\left|u_{\varepsilon_n}\right|_{L^\infty\left(\mathbb R^2\setminus
		\bigcup_{j\in\Gamma}B_n(q_j)\right)}^2>\frac\gamma2.
	\]
	After passing to a subsequence, the argument leading to
	\eqref{uniformDecayOutside} for $R=n$ provides profile centers \(y_{n,j}\) with
	\(\varepsilon_ny_{n,j}\to y_j\in K_j\), and so \(\varepsilon_ny_{n,j}\in\Lambda_j\) for large \(n\), such that
	\[
	\limsup_{n\to+\infty}
	|u_{\varepsilon_n}|_{L^\infty(\mathbb{R}^2\setminus\bigcup_{j\in\Gamma}B_n(y_{n,j}))}=0,
	\]
	obtaining a contradiction.
	
	By \eqref{empintj1j2e}, for $\varepsilon$ small enough the balls
	\(B_R(q_{\varepsilon,j})\) are mutually disjoint.
	Moreover, Lemma \ref{smoothlocalization} gives, after decreasing
	\(\varepsilon_*\),
	\[
	\zeta_{j,\varepsilon}:=\zeta_j(\varepsilon \cdot)=1\quad\text{on }B_R(q_{\varepsilon,j}),
	\qquad
	B_R(q_{\varepsilon,j})\cap\operatorname{supp}\zeta_{i,\varepsilon}
	=\emptyset\quad(i\ne j).
	\]

	We also have that, for \(\varepsilon>0\) sufficiently small, 
	\begin{equation}\label{p5.8}
		|u_\varepsilon|_{L^\infty(B_R(q_{\varepsilon,j}))}^2\geq\gamma,
		\qquad j\in\Gamma.
	\end{equation}
	Otherwise, since \(\Gamma\) is finite, there would be an index
	\(\iota\in\Gamma\) and a sequence \(\varepsilon_n\to0^+\) such that
	\(|u_n|_{L^\infty(B_R(q_{\varepsilon_n,\iota}))}^2<\gamma\).
	Since
	\(\operatorname{supp}\zeta_{\iota,\varepsilon_n}\cap B_R(q_{\varepsilon_n,j})=\emptyset\)
	for \(j\ne\iota\), the contradiction assumption and \eqref{p5.7} give $|u_n(x)|^2<\gamma\quad\text{for a.e. }x\in\operatorname{supp}\zeta_{\iota,\varepsilon_n}$.\\
	Testing \eqref{origi} with \(\zeta_{\iota,\varepsilon_n}^2u_n\), and using Young's
	inequality and \eqref{gg}, we obtain, for every \(0<\delta<1\),
	\begin{align*}
		(1-\delta)\int_{\mathbb R^2}\zeta_{\iota,\varepsilon_n}^2
		|\nabla_{A_{\varepsilon_n}}u_n|^2dx
		+\int_{\mathbb R^2}V_{\varepsilon_n}(x)\zeta_{\iota,\varepsilon_n}^2|u_n|^2dx
		&\leq\int_{\mathbb R^2}f(|u_n|^2)\zeta_{\iota,\varepsilon_n}^2|u_n|^2dx
		+C_\delta\int_{\mathbb R^2}|\nabla\zeta_{\iota,\varepsilon_n}|^2|u_n|^2dx\\
		&\leq\frac1k\int_{\mathbb R^2}V_{\varepsilon_n}(x)\zeta_{\iota,\varepsilon_n}^2|u_n|^2dx
		+C_\delta\int_{\mathbb R^2}|\nabla\zeta_{\iota,\varepsilon_n}|^2|u_n|^2dx.
	\end{align*}
	Since by Lemma \ref{smoothlocalization} \(|\nabla\zeta_{\iota,\varepsilon_n}|\leq C\varepsilon_n\) and \(\{u_n\}\) is bounded
	in \(H_{\varepsilon_n}\), the last term is \(o_n(1)\).
	As \(k>1\) and
	\(\zeta_{\iota,\varepsilon_n}=1\) on \(\Lambda_{\iota,\varepsilon_n}\), it follows that $\|u_n\|_{\varepsilon_n,\iota}^2=o_n(1)$, 
	contrary to
	\(\|u_n\|_{\varepsilon_n,\iota}\geq
	m_{\varepsilon_n}(u_n)\geq\rho/2\).

	Assume now that \(A\in C^1(\mathbb{R}^2,\mathbb{R}^2)\), and let
	\(\Omega\Subset\mathbb R^2\) be a domain with \(C^1\) boundary.\\
	By \eqref{uniformLinfty} and the continuity of \(f\),
	\[
	|f(|u_n|^2)u_n|\leq
	C_\infty\max_{0\leq s\leq C_\infty^2}|f(s)|,
	\qquad
	f(|u_n|^2)u_n\in L^\tau(\Omega,\mathbb C)
	\quad(1\leq\tau<+\infty).
	\]
	The equation \eqref{origi} reads, in the sense of distributions,
	\[
	-\Delta u_n
	=f(|u_n|^{2})u_n
	+\dfrac{2}{i}A_{\varepsilon_n}(x)\cdot\nabla u_n
	-i\nabla \cdot A_{\varepsilon_n}(x)u_n-|A_{\varepsilon_n}(x)|^{2}u_n
	-V_{\varepsilon_n}(x)u_n .
	\]
	By Lemma \ref{coro}, \(u_n\in H^1(\Omega,\mathbb{C})\) and, since \(|A|,|\nabla A|,V\in
	L^\infty_{\rm loc}(\mathbb{R}^2,\mathbb{R})\), the
	coefficients \(A_{\varepsilon_n}\), \(\nabla\cdot A_{\varepsilon_n}\) and
	\(V_{\varepsilon_n}\) are bounded on \(\Omega\). Hence all terms in the
	right-hand side except \(A_{\varepsilon_n}\cdot\nabla u_n\) belong to
	\(L^\tau(\Omega,\mathbb{C})\) for every \(1\leq\tau<+\infty\), 
	while
	\(A_{\varepsilon_n}\cdot\nabla u_n\in L^2(\Omega,\mathbb{C})\). Hence, denoting
	the right-hand side by \(F_n\), we have \(F_n\in L^2(\Omega,\mathbb{C})\).
	
	We
	first prove the \(W^{2,2}\)-regularity.
	Take a bounded \(C^1\) domain \(D\) such
	that \(\overline\Omega\subset D\), and choose \(\eta\in C_c^\infty(D)\) with
	\(\eta=1\) in a neighbourhood of \(\overline\Omega\).
	Since
	\(u_n\in H^1(D,\mathbb C)\cap L^\infty(D,\mathbb C)\) and the coefficients are
	bounded on \(D\), the same term-by-term estimates used on \(\Omega\) give
	\(F_n\in L^2(D,\mathbb{C})\). Moreover, in the sense of
	distributions,
	\[
	-\Delta(\eta u_n)=\eta F_n-2\nabla\eta\cdot\nabla u_n-u_n\Delta\eta=:G_n
	\in L^2(D,\mathbb{C}).
	\]
	Let \(w_n\) be the Newtonian potential of \(-G_n\) on \(D\).
	By \cite[Theorem 9.9]{Gilbarg}, \(w_n\in W^{2,2}(D,\mathbb C)\) and
	\(\Delta w_n=-G_n\) a.e. in \(D\), so that, in the sense of distributions,
	\[
	\Delta(\eta u_n-w_n)=0 \quad \text{in }D.
	\]
	
	Thus \(\eta u_n-w_n\) is weakly
	harmonic in \(D\).
	Applying
	\cite[Chapter 5, Theorem 4.3]{SteinShakarchiReal}
	to its real and imaginary parts, \(\eta u_n-w_n\) coincides a.e. in \(D\)
	with a harmonic function. Replacing it by this representative, we obtain
	\(\eta u_n-w_n\in C^\infty(D,\mathbb{C})\), and in particular
	\(\eta u_n-w_n\in W^{2,2}(\Omega,\mathbb{C})\). Therefore
	\(u_n=(\eta u_n-w_n)+w_n\in W^{2,2}(\Omega,\mathbb{C})\).\\
	For \(1<\tau<+\infty\), choose \(p\in(1,2)\) with
	\(\tau\leq2p/(2-p)\). By \cite[Section 5.6.1, Theorem 2]{Evans},
	\[
	\nabla u_n\in W^{1,p}(\Omega,\mathbb C)
	\hookrightarrow L^{2p/(2-p)}(\Omega,\mathbb C)
	\hookrightarrow L^{\tau}(\Omega,\mathbb C),
	\]
	and, since \(\Omega\) is arbitrary, \(\nabla u_n\in L^{\tau}_{\rm loc}(\mathbb R^2,\mathbb C)\).

	Then, taking again a bounded \(C^1\) domain \(D\)
	such that \(\overline\Omega\subset D\),  \(\nabla u_n\in L^\tau(D,\mathbb{C})\). 
	Thus \(F_n\) belongs to
	\(L^\tau(D,\mathbb{C})\) and, repeating the localized Newtonian-potential
	argument for $L^\tau$ and applying
	\cite[Theorem 9.9]{Gilbarg} with \(p=\tau\), we obtain
	\(u_n\in W^{2,\tau}(\Omega,\mathbb{C})\). Hence, since \(\Omega\) is arbitrary, taking
	\(\tau>2\), the Sobolev embedding theorem \cite[Section 5.6, Theorem 6]{Evans} yields
	\(u_n\in C^{1,1-2/\tau}_{\rm loc}(\mathbb{R}^2,\mathbb{C})\) and so, due to the arbitrariness of the sequence \(\varepsilon_n\to0^+\), for
	every sufficiently small
	\(\varepsilon\),  \(u_\varepsilon\in C^1_{\rm loc}(\mathbb R^2,\mathbb C)\).\\
	For each such \(\varepsilon\) and \(j\in\Gamma\), let
	\(p_{\varepsilon,j}\) be a maximum point of \(|u_\varepsilon|\)  on
	\(\overline{B_R(q_{\varepsilon,j})}\),  and keep this choice fixed.
	By \eqref{p5.7}, the disjointness of the balls, and continuity,  we have
	the squared modulus on their boundaries is at most \(\gamma/2\),
	whereas \eqref{p5.8} gives $p_{\varepsilon,j}\in B_R(q_{\varepsilon,j})$ and $|u_\varepsilon(p_{\varepsilon,j})|^2\geq\gamma$.
	Thus these points are mutually distinct local maximum points. Moreover, if along $\varepsilon_n\to 0^+$ we write \(p_{n,j}:=p_{\varepsilon_n,j}\),
	since \(B_R(q_{\varepsilon_n,j})\subset B_{2R}(p_{n,j})\), we obtain
	\begin{equation}\label{p5.10}
		|u_n|_{L^\infty(\mathbb R^2\setminus\bigcup_{j\in\Gamma}B_{2R}(p_{n,j}))}^2
		<\gamma,\qquad
		|u_n|_{L^\infty(B_{2R}(p_{n,j}))}^2\geq\gamma,\qquad j\in\Gamma,
	\end{equation}
	and \eqref{gg} yields
	\begin{equation}\label{p5.11}
		f(|u_n|^2)|u_n|\leq\frac{V_0}{k}|u_n|
		\quad\text{in }\mathbb R^2\setminus\bigcup_{j\in\Gamma}B_{2R}(p_{n,j}).
	\end{equation}
	
	We next compare the maximum points with the profile centers.
	By \eqref{uniformDecayOutside}, there exists
	\(R_1>0\) such that, for all sufficiently large \(n\),
	\[
	|u_n|_{L^\infty(\mathbb R^2\setminus\bigcup_{h\in\Gamma}B_{R_1}(y_{n,h}))}^2
	<\gamma.
	\]
	Since \(|u_n(p_{n,j})|^2\geq\gamma\), each \(p_{n,j}\) belongs to one
	of these profile balls. Its index must be \(j\) for large \(n\). Indeed, since for $n$ large $\varepsilon_n q_{\varepsilon_n,j} \in \Lambda_j$ and $p_{n,j}\in B_R(q_{\varepsilon_n,j})$, then
	\(\operatorname{dist}(\varepsilon_np_{n,j},\overline\Lambda_j)
	\leq\varepsilon_nR\), whereas
	\(\varepsilon_ny_{n,h}\to y_h\in K_h\subset\Lambda_h\), and the
	closures of different wells are mutually disjoint.
	Hence
	\[
	|p_{n,j}-y_{n,j}|<R_1,\qquad
	\operatorname{dist}(\varepsilon_np_{n,j},K_j)
	\leq |\varepsilon_np_{n,j}-y_j|
	\leq\varepsilon_nR_1+
	|\varepsilon_ny_{n,j}-y_j|\longrightarrow0.
	\]
	Then
	\(V(\varepsilon_np_{n,j})\to V_j\) and
	\(\varepsilon_np_{n,j}\in\Lambda_j\) for large \(n\).
	Therefore, for $\varepsilon>0$ small enough, \(\varepsilon p_{\varepsilon,j}\)'s are local maximum points of \(x\mapsto|u_\varepsilon(x/\varepsilon)|\) and the preceding ball inclusions and \eqref{p5.10} give
	\eqref{zhendang} (after replacing \(R\) therein by \(2R\)).
	
	Since $A\in C^1(\mathbb{R}^2,\mathbb{R}^2)$ and
	$u_n\in W^{2,2}_{\rm loc}(\mathbb{R}^2,\mathbb{C})$, we have
	\[\left(\frac{1}{i}\nabla-A_{\varepsilon_n}(x)\right)^{2}u_n
	\in L^{2}_{\rm loc}(\mathbb{R}^2,\mathbb{C})\]
	and Kato's inequality \cite[Lemma A]{Kato} gives
	\[\Delta |u_n|\geq -\operatorname{Re}\left[\operatorname{sgn} \overline{u}_n
	\left(\frac{1}{i}\nabla-A_{\varepsilon_n}(x)\right)^{2}u_n\right],
	\quad
	\text{where }
	\operatorname{sgn} \overline{u}(x) =
	\begin{cases}
		0 & \text{if } u(x) = 0, \\
		\overline{u(x)}/|u(x)| & \text{if } u(x) \neq 0,
	\end{cases} 
	\]
	in the sense of distributions.
	Then, by \eqref{origi} and
	\(V_{\varepsilon_n}\geq V_0\) we obtain
	\[
	\Delta |u_n| \geq V_0 |u_n| - f(|u_n|^2)|u_n|\quad \text{in } \mathbb{R}^2
	\]
	and, using \eqref{p5.11}, 
	\begin{equation}\label{comparison}
		\Delta |u_n|\geq  V_0|u_n|-f(|u_n|^{2})|u_n|\geq \frac{k-1}{k}V_0|u_n|,\quad \text{in } \mathbb{R}^{2}\setminus\bigcup_{j\in \Gamma}B_{2R}(p_{n,j}).
	\end{equation}
	
	Consider the comparison function
	\[
	W(x):=C\sum_{j\in\Gamma}e^{-\sigma_{\mathrm d}|x-p_{n,j}|},
	\qquad x\in\mathbb R^2,
	\]
	where 
	\begin{equation}\label{neibu}
		\sigma_{\mathrm d}\in \left(0,
		\sqrt{\frac{k-1}{k}V_0}\right),
		\qquad
		C\geq\max\{\sqrt{\gamma},C_\infty\}e^{2R\sigma_{\mathrm d}}.
	\end{equation}
	
	By \eqref{comparison}, we have, in the sense of distributions in
	\(\displaystyle \mathbb R^2\setminus\bigcup_{j\in\Gamma}\overline{B_{2R}(p_{n,j})}\),
	\begin{equation}\label{comparisonDifference}
		\begin{split}
			-\Delta (W-|u_n|)(x)+\frac{k-1}{k}V_0(W-|u_n|)(x)
			&\geq C\sum_{j\in\Gamma}
			\left(\frac{k-1}{k}V_0-\sigma_{\mathrm d}^2+\frac{\sigma_{\mathrm d}}{|x-p_{n,j}|}\right)
			e^{-\sigma_{\mathrm d}|x-p_{n,j}|}\\
			&\geq C\left(\frac{k-1}{k}V_0-\sigma_{\mathrm d}^2\right)
			\sum_{j\in\Gamma}e^{-\sigma_{\mathrm d}|x-p_{n,j}|}\geq0
		\end{split}
	\end{equation}
	and, if $\displaystyle x\in\partial\Big(\bigcup_{j\in \Gamma}B_{2R}(p_{n,j})\Big)$, for $\iota\in\Gamma$ such that $|x-p_{n,\iota}|=2R$, combining \eqref{p5.10} and \eqref{neibu}, we have
	\begin{equation*}
		|u_n(x)|
		\leq \sqrt{\gamma}
		\leq C e^{-2\sigma_{\mathrm d} R} 
		=Ce^{-\sigma_{\mathrm d}|x-p_{n,\iota}|}\leq W(x).
	\end{equation*}
	Set
	\[
	D_n:=\mathbb R^2\setminus\bigcup_{j\in\Gamma}
	\overline{B_{2R}(p_{n,j})},\qquad
	\lambda:=\frac{k-1}{k}V_0,\qquad
	v_n:=(|u_n|-W)^+.
	\]
	For large \(n\), the balls defining \(D_n\) are mutually disjoint.
	Since \(|u_n|,W\in H^1(\mathbb R^2,\mathbb R)\) and
	\(|u_n|\leq W\) on \(\partial D_n\), we have
	\(v_n\in H_0^1(D_n,\mathbb R)\).
	Let $M>0$ and  \(\chi_M\in C_c^\infty(\mathbb R^2,[0,1])\) satisfy
	\(\chi_M=1\) on \(B_M(0)\), \(\chi_M=0\) outside
	\(B_{2M}(0)\), and \(|\nabla\chi_M|\leq C/M\).
	By approximation with nonnegative smooth test functions, we may test
	\eqref{comparisonDifference} with \(\chi_M^2v_n\) obtaining
	\[
	\int_{D_n}\chi_M^2|\nabla v_n|^2dx
	+2\int_{D_n}\chi_Mv_n\nabla v_n\cdot\nabla\chi_Mdx
	+\lambda\int_{D_n}\chi_M^2v_n^2dx \leq 0.
	\]
	Then, applying Young's inequality,  we get
	\[
	\int_{D_n}\chi_M^2|\nabla v_n|^2dx
	+\lambda\int_{D_n}\chi_M^2v_n^2dx
	\leq\frac12\int_{D_n}\chi_M^2|\nabla v_n|^2dx
	+2\int_{D_n}|\nabla\chi_M|^2v_n^2dx
	\]
	and so
	\[
	\frac12\int_{D_n}\chi_M^2|\nabla v_n|^2dx
	+\lambda\int_{D_n}\chi_M^2v_n^2dx
	\leq 2\int_{D_n}|\nabla\chi_M|^2v_n^2dx
	\leq\frac{C}{M^2}\int_{D_n\cap(B_{2M}(0)\setminus B_M(0))}v_n^2dx.
	\]
	Since \(v_n\in L^2(D_n)\), letting \(M\to+\infty\) gives
	\[
	\frac12\int_{D_n}|\nabla v_n|^2dx
	+\lambda\int_{D_n}v_n^2dx\leq0.
	\]
	Thus \(v_n=0\) a.e. in \(D_n\). By continuity up to its boundary, we conclude that, for every $\displaystyle x\in\mathbb{R}^{2}\setminus\bigcup_{j\in \Gamma}B_{2R}(p_{n,j})$, $|u_n(x)|\leq W(x)$.
	For \(\displaystyle x\in\bigcup_{j\in\Gamma}B_{2R}(p_{n,j})\), choose
	\(\iota\in\Gamma\) with \(|x-p_{n,\iota}|<2R\). By
	\eqref{uniformLinfty} and \eqref{neibu},
	\begin{equation*}
		|u_n(x)|
		\leq |u_n|_{\infty}
		\leq C_\infty \leq C e^{-2\sigma_{\mathrm d} R}
		< C e^{-\sigma_{\mathrm d} |x-p_{n,\iota}|}
		\leq W(x).
	\end{equation*}
	Thus, taking \(\,y_{\varepsilon,j}:=\varepsilon p_{\varepsilon,j}\,\) for
	\(j\in\Gamma\) and rescaling, we obtain \eqref{decay} for
	the function $u_n(\cdot/\varepsilon_n)$.
	
	It remains to make the estimates uniform in \(\varepsilon\).
	Observe that the constant \(\gamma\) in \eqref{gg} depends
	only on \(f\), \(k\), and \(V_0\).
	Moreover, arguing as in \eqref{bddPS}, 
	\[
	\|u_\varepsilon\|_\varepsilon^2\leq
	\Big(\frac12-\frac1\Theta-\frac1{2k}\Big)^{-1}C_{\varepsilon,\Gamma}
	\]
	that, combined with \eqref{Cepslimit} and \eqref{p3.1}, gives
	\[
	\limsup_{\varepsilon\to0^+}\|u_\varepsilon\|_\varepsilon^2
	\leq\left(\frac12-\frac1\Theta-\frac1{2k}\right)^{-1}
	\sum_{j=1}^s c_{V_j}<\frac18.
	\]
	Hence, if $\varepsilon_*>0$ is small enough, for every $\varepsilon\in(0,\varepsilon_*)$, 
	\[
	|\nabla|u_\varepsilon||_2^2+V_0|u_\varepsilon|_2^2
	\leq\|u_\varepsilon\|_\varepsilon^2\leq\frac18.
	\]
	Furthermore, the normalization parameter in \eqref{sigma} and the exponents
	\(\alpha,q,\tau_0,\tau_1,\tau_2\) can therefore be fixed independently
	of \(\varepsilon\)  small. In particular,  as for \eqref{exp-unif},
	\[
	\sup_{0<\varepsilon<\varepsilon_*}
	\int_{\mathbb R^2}\left(e^{\tau_0\alpha|u_\varepsilon|^2}-1\right)\,dx
	\leq C.
	\]
	Thus the Moser iteration above gives a common constant \(C_\infty\)
	in \eqref{uniformLinfty} for the whole family when $\varepsilon\in(0,\varepsilon_*)$.
	Finally, the radius \(R\) in \eqref{p5.7} is independent of \(\varepsilon\) and  the centers \(q_{\varepsilon,j}\) and the
	maximum points \(p_{\varepsilon,j}\) have already been fixed.  Then, arguing as before, for every \(j\in\Gamma\),
	\[
	\operatorname{dist}(\varepsilon p_{\varepsilon,j},K_j)\to0,
	\qquad V(\varepsilon p_{\varepsilon,j})\to V_j,
	\quad\text{as }\varepsilon\to 0^+.
	\]
	
	Hence, choosing once and for all
	\(0<\sigma_{\mathrm d}<\sqrt{(k-1)V_0/k}\) and then \(C\) as in
	\eqref{neibu}, independently of \(\varepsilon\), we obtain \eqref{decay}
	for every \(\varepsilon\in(0,\varepsilon_*)\), after decreasing
	\(\varepsilon_*\) if necessary.

	\appendix

	\section{Proof of Lemma \ref{existenceC}}\label{proofexistenceC}
	
	Fix $\mu\in(0,V^*]$ and let us consider
	\[
	A_\mu^c:=
	\inf _{u \in \mathcal{P}_\mu^c}  |\nabla u|_2^{2},
	\]
	where
	\[
	\mathcal{P}_\mu^c:=\left\{u\in H^1(\mathbb{R}^2,\mathbb{C})\setminus\{0\}: \int_{\mathbb{R}^{2}} H(|u|^2)dx=0 \right\}
	\text{ with }
	H(s):=F(s)-\mu s.
	\]

	Now we list some results we will use to get our goal.
	
	\begin{lemma}\label{lem5.3}
		$A_\mu^c>0$.
	\end{lemma}
	
	\begin{proof}
		Clearly $A_\mu^c \geq0$. Assume by contradiction that $A_\mu^c=0$ and let $\{u_{n}\}$ be a minimizing sequence in $\mathcal{P}_\mu^c$.
		Without loss of generality we can take $\{u_{n}\}\subset \mathcal{P}_\mu^c\cap H_{\text{rad}}^{1}(\mathbb{R}^{2},\mathbb{R})$. Indeed, if we denote with $v^*$ the symmetric-decreasing rearrangement of $v\in H^{1}(\mathbb{R}^{2},\mathbb{C})$,
		using also \cite[Theorem 6.17]{Lieb}, we have that $v^*$ is a nonnegative, radially symmetric, and nonincreasing function such that
		\[
		|\nabla v^*|_{2}\leq|\nabla|v||_2\leq |\nabla v|_{2}
		\quad
		\text{and}
		\quad
		\int_{\mathbb{R}^2}H((v^*)^2)dx=\int_{\mathbb{R}^2}H(|v|^2)dx
		\] 
		(see \cite[Sections 3.3 and 7.17]{Lieb}).
		Moreover, for each $\lambda_{n}>0$, the function $v_{n}:=u_{n}\left(\cdot/\lambda_{n}\right)\in \mathcal{P}_\mu^c$ and satisfies $|\nabla v_{n}|_2=|\nabla u_{n}|_2$.
		In addition, taking $\lambda_n=|u_n|_2^{-1}$ we also have that $|v_n|_2=1$.
		Thus, $v_n$ is bounded in $H^{1}_{\text{rad}}(\mathbb{R}^{2},\mathbb{R})$ and up to a subsequence, there exists $v \in H^{1}_{\text{rad}}(\mathbb{R}^{2},\mathbb{R})$ such that $v_{n} \rightharpoonup v$ in $H^{1}_{\text{rad}}(\mathbb{R}^{2},\mathbb{R})$ and  $v_n \rightarrow v$ in $L^r_{\rm loc}(\mathbb{R}^{2},\mathbb{R})$ for $1\leq r<+\infty$ and a.e. in $\mathbb R^2$. 
		Furthermore, by \cite[Lemma A.II]{Berestycki}, $\lim _{|x| \to +\infty}|v_n(x)| = 0$ uniformly in $n$. 
		Now, let $\alpha>4\pi$. Since $|\nabla v_n|_2\to0$, there exist
		$\nu\in\mathbb N$ and $\varrho>0$ such that
		\[
		\sup_{n\geq\nu}|\nabla v_n|_2^2\leq\varrho
		\quad\text{and}\quad \alpha\varrho<4\pi.
		\]
		For $n\geq\nu$, the function
		$v_n/\sqrt{\varrho}$ satisfies
		$|\nabla(v_n/\sqrt{\varrho})|_2\leq1$ and
		$|v_n/\sqrt{\varrho}|_2
		=1/\sqrt{\varrho}$. Lemma \ref{2.1}, with exponent
		$\alpha\varrho<4\pi$, therefore yields
		\[
		\sup_{n\geq\nu}\int_{\mathbb R^2}
		\bigl(e^{\alpha v_n^2}-1\bigr)dx
		=\sup_{n\geq\nu}\int_{\mathbb R^2}
		\left(e^{\alpha\varrho
			(v_n/\sqrt{\varrho})^2}-1\right)dx<+\infty.
		\]
		Hence
		\[
		\sup_{n\in\mathbb{N}}\int_{\mathbb R^2}
		\bigl(e^{\alpha v_n^2}-1\bigr)dx
		<+\infty
		\]
		since, by the first part of Lemma \ref{2.1}, the finitely many integrals for $n<\nu$ are finite.
		Then, setting
		$P(s):=F(s^2)$ and $Q(s):=e^{\alpha s^2}-1$, by \eqref{f1} and \eqref{f2},
		we have\footnote{The second limit follows from \(\lim_{s\to0}P(s)/s^2=0\) and
			\(Q(s)=e^{\alpha s^2}-1\sim\alpha s^2\) as \(s\to0\).}
		\[
		\lim_{|s|\to+\infty}\frac{P(s)}{Q(s)}=0,
		\qquad
		\lim_{s\to0}\frac{P(s)}{Q(s)}=0.
		\]
		Together with the a.e. convergence, the preceding uniform bound for
		$Q(v_n)$, and the uniform radial decay, these are precisely the hypotheses
		of \cite[Theorem A.I]{Berestycki}. Consequently,
		\begin{equation*}
			\int_{\mathbb{R}^2} F(v_n^2) dx\to \int_{\mathbb{R}^2} F(v^2)dx.
		\end{equation*}
		Then, since $\{v_n\}\subset\mathcal{P}_\mu^c$,
		\[
		\int_{\mathbb{R}^{2}} F(v^2)dx=\mu>0,
		\]
		and so $v \neq0$.
		On the other hand, the weak lower semicontinuity of the norm implies
		\[
		0=\liminf _{n \to \infty}  |\nabla v_{n}|_2^{2} \geq  |\nabla v|_2^{2},
		\]
		obtaining a contradiction.
	\end{proof}

	\begin{lemma}\label{lem5.2}
		$\displaystyle \frac{A_\mu^c}{2} \leq  c_\mu^c < \frac{\Theta-2}{16m\Theta}$.
	\end{lemma}
	
	\begin{proof}
		For each $v \in H^{1}(\mathbb{R}^{2},\mathbb{C})\setminus\{0\}$, we consider the function $\xi: [0,+\infty)\to \mathbb{R}$ by
		\[
		\xi(t)=\int_{\mathbb{R}^{2}} H(|t v|^2)dx.
		\]
		By the assumptions (\ref{f1}), (\ref{f3}), and (\ref{f4}), we can conclude that $\xi(t)<0$ for $t$ small enough and $\xi(t)>0$ for $t$ large enough. Thus, there exists $t_{0}>0$ such that $\xi(t_{0})=0$, namely $t_{0} v \in \mathcal{P}_\mu^c$. Hence,
		\[
		A_\mu^c \leq |\nabla(t_{0} v)|_2^{2}=2 I_\mu^c\left(t_{0} v\right) \leq 2\max _{t \geq 0} I_\mu^c(t v),
		\]
		and so, passing to the infimum on $v \in H^{1}\left(\mathbb{R}^{2},\mathbb{C}\right)\setminus\{0\}$, we get the first inequality.
		
		To conclude, let $\psi\in H^{1} (\mathbb{R}^{2},\mathbb{C})\setminus\{0\}$ be such that $
		|\nabla\psi|_2^2+\mu|\psi|_2^2=S_p^\mu=\mu^{2/p}S_p$ (see \cite[Lemma 3.2 and Remark 3.3]{d'Avenia2023}) and $|\psi|_p=1$.
		Using \eqref{f4} we have 
		\begin{equation*}
			c_\mu^c
			\leq \max_{t \geq 0} I_\mu^c(t \psi) 
			\leq \frac{1}{2} \max_{t \geq 0} \left\{ t^2 S_p^\mu - \frac{2C_p }{p}t^p \right\}
			=\frac{p - 2}{2p} \left( \frac{\mu}{C_p} \right)^{\frac{2}{p - 2}}S_p^{\frac{p}{p-2}}
			<\frac{\Theta-2}{16m\Theta} \left( \frac{\mu}{V^*} \right)^{\frac{2}{p - 2}}
			\leq \frac{\Theta-2}{16m\Theta}.
		\end{equation*}
	\end{proof}

	Now we are ready to prove Lemma \ref{existenceC}.
	
	\begin{proof}
		First we prove that $A_\mu^c$ is attained by a nonnegative and nonincreasing function in $H^{1}_{\text{rad}}(\mathbb{R}^{2},\mathbb{R})$.\\
		Arguing as in Lemma \ref{lem5.3}, let $\{u_{n}\}$ be a minimizing sequence
		of nonnegative and nonincreasing functions in
		$H^{1}_{\text{rad}}(\mathbb{R}^{2},\mathbb{R})$ for $A_\mu^c$ with
		$|u_{n}|_2=1$.
		Then, by Lemma \ref{lem5.2},
		\[
		\limsup _{n \to \infty} |\nabla u_{n}|_2^{2} \leq  A_\mu^c \leq  2 c_\mu^c<\frac{\Theta-2}{8m\Theta}<1.
		\]
		Moreover, passing to a subsequence, there exists a nonnegative and nonincreasing
		$u\in H_{\rm rad}^{1}(\mathbb{R}^{2},\mathbb{R})$ such that
		$u_n\rightharpoonup u$ in $H^{1}(\mathbb{R}^{2},\mathbb{R})$.
		
		Then, if we fix $\varrho\in(A_\mu^c,1)$ and
		$\alpha\in(4\pi,4\pi/\varrho)$, since for $n$ sufficiently large, we may assume that $|\nabla u_n|_2^2\leq\varrho$, applying Lemma \ref{2.1} to
		$u_n/\sqrt{\varrho}$ and using
		$|u_n|_2=1$, we obtain
		\[
		\sup_{n\in\mathbb{N}}\int_{\mathbb R^2}
		\left(e^{\alpha u_n^2}-1\right)dx<+\infty.
		\]
		Therefore, arguing as in the proof of Lemma \ref{lem5.3},
		\[
		\int_{\mathbb R^2}F(u_n^2)\,dx
		\longrightarrow
		\int_{\mathbb R^2}F(u^2)\,dx
		=\mu
		\]
		and so $u \neq0$. Moreover, by the weak lower semicontinuity of the norm,
		\[
		\int_{\mathbb{R}^{2}}|\nabla u|^{2}dx \leq A_\mu^c.
		\]
		If we show that $u\in\mathcal{P}_\mu^c$, we can conclude this part of the proof. To this end, observe that
		\[
		1=\liminf _{n \to \infty} |u_{n}|_2^{2} \geq |u|_2^{2}
		\]
		and so,
		\[
		\int_{\mathbb{R}^{2}} H(u^2)dx=\int_{\mathbb{R}^{2}} F(u^2)dx- \mu |u|_2^{2}=\mu(1- |u|_2^{2}) \geq 0.
		\]
		Assume by contradiction that the strict inequality holds.
		Then, consider the function $\xi$ defined in the proof of Lemma \ref{lem5.2} and obtained starting from $u$, which is negative for $t$ small enough and
		\[
		\xi(1)=\int_{\mathbb{R}^{2}} H(u^2)dx>0.
		\]
		Thus, there exists $t_{0} \in(0,1)$ such that
		\[
		\xi(t_{0})=\int_{\mathbb{R}^{2}} H\left(t_{0}^2 u^2\right)dx=0.
		\]
		Hence,
		\[
		0<|\nabla (t_{0} u)|_2^{2} \leq t_{0}^{2} A_\mu^c<A_\mu^c,
		\]
		which is absurd, and so $u\in\mathcal{P}_\mu^c$ and $|\nabla u|_2^2=A_\mu^c$.

		Then, since for every $u\in\mathcal{P}_\mu^c$, by \eqref{f3},
		\begin{equation}\label{codim1}
			2\int_{\mathbb R^2}\bigl(f(|u|^2)|u|^2-\mu |u|^2\bigr)dx
			\geq\int_{\mathbb R^2}
			\left(\Theta F(|u|^2)-2\mu |u|^2\right)dx
			=(\Theta-2)\mu \int_{\mathbb R^2}|u|^2dx>0,
		\end{equation}
		we have that $\mathcal{P}_\mu^c$ is a manifold of codimension $1$ and so, by the Lagrange multiplier theorem,
		there exists $\vartheta \in \mathbb{R}$ such that
		\begin{equation*}
			\operatorname{Re}\int_{\mathbb{R}^{2}}\nabla u\,\overline{\nabla v}\,dx
			=\vartheta\operatorname{Re}\int_{\mathbb{R}^{2}}
			\bigl(f(u^2)-\mu\bigr)u\overline v\,dx,
			\qquad v\in H^{1}\left(\mathbb{R}^{2},\mathbb{C}\right).
		\end{equation*}
		Hence, taking $v=u$,
		\[
		|\nabla u|_2^2
		=\vartheta\int_{\mathbb R^2}\bigl(f(u^2)-\mu\bigr)u^2dx
		\]
		and so, using \eqref{codim1}, $\vartheta>0$.
		
		Hence, the rescaled function $u_{\sqrt{\vartheta}}:=u\left(\cdot/\sqrt{\vartheta}\right)$ is a nontrivial solution of \eqref{fude} and satisfies
		\[
		|\nabla u_{\sqrt{\vartheta}}|_2^{2}=|\nabla u|_2^{2} \quad \text{and} \quad \int_{\mathbb{R}^{2}} H(u_{\sqrt{\vartheta}}^2)dx=0.
		\]
		Thus, since $u_{\sqrt{\vartheta}}\in\mathcal{N}_\mu^c$,
		\[
		c_\mu^c 
		\leq I_\mu^c(u_{\sqrt{\vartheta}})
		= \frac{|\nabla u|_2^{2}}{2}
		=\frac{A_\mu^c}{2}
		\]
		and so, by Lemma \ref{lem5.2} we have that $c_\mu^c$ is attained by $u_{\sqrt{\vartheta}}$ which is a nonnegative and nonincreasing function in $H^{1}_{\text{rad}}(\mathbb{R}^{2},\mathbb{R})$.
		
		Moreover, since $u_{\sqrt{\vartheta}}\in\mathcal{N}_\mu\subset \mathcal{N}_\mu^c$, we have
		\[
		c_\mu^c
		=I_\mu^c(u_{\sqrt{\vartheta}})
		= I_\mu(u_{\sqrt{\vartheta}})
		\geq c_\mu
		=\inf_{\mathcal{N}_\mu} I_\mu
		=\inf_{\mathcal{N}_\mu} I_\mu^c
		\geq \inf_{\mathcal{N}_\mu^c} I_\mu^c
		=c_\mu^c
		\]
		concluding the proof.
	\end{proof}

	\subsection*{Conflict of Interest}
	The authors declare no conflict of interest.
	
	\subsection*{Data Availability Statements}
	Data sharing not applicable to this article as no datasets were generated or analysed during the current study.
	
	\subsection*{Acknowledgements}
	Pietro d'Avenia is member of GNAMPA (INdAM) and is partially supported by the GNAMPA Project 2026 (CUP E53C25002010001).
	He is also supported by the Italian Ministry of University and Research under the Program Department of Excellence L. 232/2016 (CUP D93C23000100001).
	Chao Ji is supported by the National Natural Science Foundation of China (No. 12571117).

\end{document}